\documentclass[11pt, one side, article]{memoir}

\settrims{0pt}{0pt} % page and stock same size
\settypeblocksize{*}{35.8pc}{*} % {height}{width}{ratio}
\setlrmargins{*}{*}{1} % {spine}{edge}{ratio}
\setulmarginsandblock{.98in}{.98in}{*} % height of typeblock computed
\setheadfoot{\onelineskip}{2\onelineskip} % {headheight}{footskip}
\setheaderspaces{*}{1.5\onelineskip}{*} % {headdrop}{headsep}{ratio}
\checkandfixthelayout

\usepackage{amsthm}
\usepackage{mathtools}

\usepackage[inline]{enumitem}
\usepackage{ifthen}
\usepackage[utf8]{inputenc} %allows non-ascii in bib file
\usepackage{xcolor}

\usepackage[backend=biber, backref=true, maxbibnames = 10, style = alphabetic]{biblatex}
\usepackage[bookmarks=true, colorlinks=true, linkcolor=blue!50!black,
citecolor=orange!50!black, urlcolor=orange!50!black, pdfencoding=unicode]{hyperref}
\usepackage[capitalize]{cleveref}

\usepackage{tikz}

\usepackage{amssymb}
\usepackage{newpxtext}
\usepackage[varg,bigdelims]{newpxmath}
\usepackage{mathrsfs}
\usepackage{dutchcal}
\usepackage{mathalfa}
\usepackage{fontawesome}
\usepackage{ebproof}
\usepackage{stmaryrd}
\usepackage{ebproof}
\usepackage{graphicx}

  \crefformat{enumi}{\card#2#1#3}
  \crefalias{chapter}{section}

  \AtBeginBibliography{\emergencystretch=3em}

  \hypersetup{final}

  \setlist{nosep}
  \setlistdepth{6}

  \usetikzlibrary{ 
  	cd,
  	math,
  	decorations.markings,
		decorations.pathreplacing,
  	positioning,
  	arrows.meta,
  	shapes,
		shadows,
		shadings,
  	calc,
  	fit,
  	quotes,
  	intersections,
    circuits,
    circuits.ee.IEC,
    backgrounds
  }
  
  \tikzset{
biml/.tip={Glyph[glyph math command=triangleleft, glyph length=.95ex]},
bimr/.tip={Glyph[glyph math command=triangleright, glyph length=.95ex]},
}

\tikzset{
	tick/.style={postaction={
  	decorate,
    decoration={markings, mark=at position 0.5 with
    	{\draw[-] (0,.4ex) -- (0,-.4ex);}}}
  }
} 
\tikzset{
	parto/.style={-{To[harpoon]}}
}

\colorlet{wdCtrl}{blue!70}
\colorlet{wdCtrlInner}{blue!50!cyan!70}   % tubes nested inside an inner box
\newlength\wdunit      % one data-slot (default)
\newlength\halfunit    % half-slot
\newlength\wdroofgap   % roof/floor gap
\newlength\wdtubewidth 

\newdimen\wiresep
\newcount\wdpadunits   \wdpadunits=1

\newcommand{\innerScale}{0.5}   % scale of inner boxes
\newlength\innerOffset
\makeatletter
\tikzset{
  wdunit/.code={%
    \setlength\wdunit{#1}%
    \setlength\halfunit{\dimexpr\wdunit/2\relax}%
    \setlength\wdroofgap{2\wdunit}%
  },
  halfunit/.code={%
    \setlength\halfunit{#1}%
  },
}
\makeatother

\tikzset{
  bypass tube width/.store in=\bypasstubewidth,
  bypass tube width=2\halfunit,      % thickness of the gray surround
  bypass tube color/.initial=gray,
  bypass tube opacity/.initial=.50,
}

\tikzset{
  bypass wire/.style={
    preaction={
      draw=\pgfkeysvalueof{/tikz/bypass tube color},
      line width=\bypasstubewidth,
      opacity=\pgfkeysvalueof{/tikz/bypass tube opacity},
      line join=round
    }
  }
}

\tikzset{
  WD/.style        ={semithick,font=\small},
  port dot/.style  ={circle,draw,fill=white,inner sep=0pt,minimum size=3pt},
  tube/.style      ={draw=wdCtrl!60,fill=wdCtrl,fill opacity=.3,line join=round},
  data wire/.style ={-{Stealth[length=3pt]},semithick,line cap=round,line join=round},
  loop stroke/.style={draw=wdCtrl,line width=\wdtubewidth,opacity=.3,
                      line cap=round,line join=round},
}

\makeatletter
\newcommand{\WDsumcsv}[2]{%
  #2=0 %
  \edef\WD@tmp{#1}%
  \@for\WD@t:=\WD@tmp\do{\advance#2 by \WD@t\relax}%
}
\newcommand{\WDcountcsv}[2]{%
  #2=0 %
  \edef\WD@tmp{#1}%
  \@for\WD@t:=\WD@tmp\do{\advance#2 by 1\relax}%
}

\newcount\WDUnits
\newcount\WDLslots \newcount\WDLk  \newcount\WDLunits
\newcount\WDRslots \newcount\WDRk  \newcount\WDRunits
\newcount\WDi      \newcount\WDj   \newcount\WDn

\newcommand{\wdAbsY}[2]{\dimen0=#1\wdunit\edef#2{\the\dimen0}}

\newcommand{\wdboxUnits}[6][]{%
  \def\BoxName{#2}\def\BoxLabel{#3}%
  \edef\WD@Llist{#4}\edef\WD@Rlist{#5}%
  \newcount\ForceUnits \ForceUnits=#6\relax

  \WDsumcsv{\WD@Llist}{\WDLslots}%
  \WDcountcsv{\WD@Llist}{\WDLk}%
  \WDLunits=\WDLslots \advance\WDLunits by \WDLk \advance\WDLunits by 1
  \WDsumcsv{\WD@Rlist}{\WDRslots}%
  \WDcountcsv{\WD@Rlist}{\WDRk}%
  \WDRunits=\WDRslots \advance\WDRunits by \WDRk \advance\WDRunits by 1

  \ifnum\WDLk=0 \PackageError{control-slots}{Left CSV empty}{Provide ≥1 entry}\fi
  \ifnum\WDRk=0 \PackageError{control-slots}{Right CSV empty}{Provide ≥1 entry}\fi

  \WDUnits=\WDLunits
  \ifnum\WDRunits>\WDUnits \WDUnits=\WDRunits\fi
  \ifnum\ForceUnits>\WDUnits \WDUnits=\ForceUnits\fi
  \advance\WDUnits by \wdpadunits \advance\WDUnits by \wdpadunits
  \ifnum\WDUnits<1 \WDUnits=1\fi

  \dimen0=\WDUnits\wdunit
  \node[draw,rounded corners=2pt,minimum width=2.6cm,
        minimum height=\the\dimen0,#1] (\BoxName) {\BoxLabel};

  \coordinate (\BoxName-nw) at (\BoxName.north west);
  \coordinate (\BoxName-ne) at (\BoxName.north east);

  \dimen2=\WDUnits\wdunit   \advance\dimen2 by -\WDLslots\wdunit
  \dimen4=\dimen2           \divide\dimen4 by \numexpr\WDLk+1\relax
  \dimen6=\dimen4           \WDi=0
  \def\WD@loopL##1{%
    \advance\WDi by 1
    \edef\YTop{\the\dimen6}\dimen8=##1\wdunit
    \dimen2=\dimen6         \advance\dimen2 by \dimen8
    \edef\YBot{\the\dimen2}%
    \coordinate (\BoxName-L-reg\the\WDi-top)
      at ([yshift=-\YTop]\BoxName-nw);
    \coordinate (\BoxName-L-reg\the\WDi-bot)
      at ([yshift=-\YBot]\BoxName-nw);
    \coordinate (\BoxName-L-reg\the\WDi-mid)
      at ([yshift=(-\YTop-\YBot)/2]\BoxName-nw);
    \WDn=##1 \WDj=0
    \loop\ifnum\WDj<\WDn
      \advance\WDj by 1
      \pgfmathsetmacro{\t}{(2*(\the\WDj)-1)/(2*\the\WDn)}%
      \coordinate (\BoxName-L-\the\WDi-D\the\WDj)
        at ($(\BoxName-L-reg\the\WDi-top)!\t!(\BoxName-L-reg\the\WDi-bot)$);
      \node[port dot] at (\BoxName-L-\the\WDi-D\the\WDj) {};
    \repeat
    \advance\dimen6 by \dimen8 \advance\dimen6 by \dimen4%
  }
  \@for\WD@t:=\WD@Llist\do{\expandafter\WD@loopL\expandafter{\WD@t}}

  \dimen2=\WDUnits\wdunit   \advance\dimen2 by -\WDRslots\wdunit
  \dimen4=\dimen2           \divide\dimen4 by \numexpr\WDRk+1\relax
  \dimen6=\dimen4           \WDi=0
  \def\WD@loopR##1{%
    \advance\WDi by 1
    \edef\YTop{\the\dimen6}\dimen8=##1\wdunit
    \dimen2=\dimen6         \advance\dimen2 by \dimen8
    \edef\YBot{\the\dimen2}%
    \coordinate (\BoxName-R-reg\the\WDi-top)
      at ([yshift=-\YTop]\BoxName-ne);
    \coordinate (\BoxName-R-reg\the\WDi-bot)
      at ([yshift=-\YBot]\BoxName-ne);
    \coordinate (\BoxName-R-reg\the\WDi-mid)
      at ([yshift=(-\YTop-\YBot)/2]\BoxName-ne);
    \WDn=##1 \WDj=0
    \loop\ifnum\WDj<\WDn
      \advance\WDj by 1
      \pgfmathsetmacro{\t}{(2*(\the\WDj)-1)/(2*\the\WDn)}%
      \coordinate (\BoxName-R-\the\WDi-D\the\WDj)
        at ($(\BoxName-R-reg\the\WDi-top)!\t!(\BoxName-R-reg\the\WDi-bot)$);
      \node[port dot] at (\BoxName-R-\the\WDi-D\the\WDj) {};
    \repeat
    \advance\dimen6 by \dimen8 \advance\dimen6 by \dimen4%
  }
  \@for\WD@t:=\WD@Rlist\do{\expandafter\WD@loopR\expandafter{\WD@t}}
}

\makeatother

\newcommand{\tubeN}[7]{%
  \path[tube]
    ([yshift= .5\halfunit]#1-#2-reg#3-top)
      to[out=0,in=180]
    ([yshift= .5\halfunit]#4-#5-reg#6-top)
    -- ([yshift=-.5\halfunit]#4-#5-reg#6-bot)
      to[out=180,in=0]
    ([yshift=-.5\halfunit]#1-#2-reg#3-bot)
    -- cycle;

  \ifnum#7>0
    \foreach \k in {1,...,#7}{%
      \draw[data wire]
        (#1-#2-#3-D\k) to[out=0,in=180] (#4-#5-#6-D\k);
    }%
  \fi
}

\makeatletter
\newcommand{\loopTubeSmooth}[9][above]{%
  \ifthenelse{\equal{#1}{above}}{%
    \coordinate (Helper) at ($(#9.north)+(0,-\wdroofgap/2)$);
  }{%
    \coordinate (Helper) at ($(#9.south)-(0,\wdroofgap/2)$);
  }%
  \coordinate (Rp) at ($(#2.east|-Helper)$);
  \coordinate (Lp) at ($(#5.west|-Helper)$);

  \coordinate (St) at (#2-#3-reg#4-top);
  \coordinate (Sb) at (#2-#3-reg#4-bot);
  \coordinate (Sc) at ($(St)!0.5!(Sb)$);
  \coordinate (Tt) at (#5-#6-reg#7-top);
  \coordinate (Tb) at (#5-#6-reg#7-bot);
  \coordinate (Tc) at ($(Tt)!0.5!(Tb)$);

  \pgfmathsetlengthmacro{\loopwidth}{(#8+1)*\halfunit}%
  \begin{pgfonlayer}{background}
    \draw[loop stroke,line cap=butt,line width=\loopwidth]
      (Sc) to[out=0,in=0] (Rp)
      -- (Lp)
      to[out=180,in=180] (Tc);
  \end{pgfonlayer}

  \ifnum#8>0
    \foreach \k in {1,...,#8}{%
      \coordinate (pS) at (#2-#3-#4-D\k);
      \coordinate (pT) at (#5-#6-#7-D\k);
      \pgfmathsetmacro{\ratio}{2*\k/(#8+1)-1}%
      \pgfmathsetlengthmacro{\offset}{\ratio*\wiresep/2}%
      \coordinate (Rp\k) at ($(Rp)+(0,\offset)$);
      \coordinate (Lp\k) at ($(Lp)+(0,\offset)$);
      \draw[data wire]
        (pS) to[out=0,in=0] (Rp\k)
        -- (Lp\k)
        to[out=180,in=180] (pT);
    }%
  \fi
}
\makeatother

\tikzset{
  oriented WD/.style={%
    every to/.style={out=0,in=180,draw},
    label/.style={
      font=\everymath\expandafter{\the\everymath\scriptstyle},
      inner sep=0pt,
      node distance=2pt and -2pt},
    semithick,
    node distance=1 and 1,
    decoration={markings, mark=at position \stringdecpos with \stringdec},
    ar/.style={postaction={decorate}},
    execute at begin picture={\tikzset{
      x=\bbx, y=\bby,
      every fit/.style={inner xsep=\bbx, inner ysep=\bby}}}
  },
  string decoration/.store in=\stringdec,
  string decoration={\arrow{stealth};},
  string decoration pos/.store in=\stringdecpos,
  string decoration pos=.7,
  bbx/.store in=\bbx,
  bbx = 1.5cm,
  bby/.store in=\bby,
  bby = 1.5ex,
  bb port sep/.store in=\bbportsep,
  bb port sep=1.5,
  bb port length/.store in=\bbportlen,
  bb port length=4pt,
  bb penetrate/.store in=\bbpenetrate,
  bb penetrate=0,
  bb min width/.store in=\bbminwidth,
  bb min width=1cm,
  bb rounded corners/.store in=\bbcorners,
  bb rounded corners=2pt,
  bb spider/.style={
    bb port sep=1, bb port length=10pt, bbx=.4cm, bb min width=.4cm, bby=.8ex},
  bb small/.style={
    bb port sep=1, bb port length=2.5pt, bbx=.4cm, bb min width=.4cm, bby=.7ex},
  bb medium/.style={
    bb port sep=1, bb port length=2.5pt, bbx=.4cm, bb min width=.4cm, bby=.9ex},
  bb/.code 2 args={%
    \pgfmathsetlengthmacro{\bbheight}{\bbportsep * (max(#1,#2)+1) * \bby}
    \pgfkeysalso{draw,minimum height=\bbheight,minimum
     width=\bbminwidth,outer sep=0pt,
       rounded corners=\bbcorners,thick,
       prefix after command={\pgfextra{\let\fixname\tikzlastnode}},
       append after command={\pgfextra{\draw
          \ifnum #1=0{} \else foreach \i in {1,...,#1} {
            ($(\fixname.north west)!{\i/(#1+1)}!(\fixname.south west)$) +(-\bbportlen,0) coordinate (\fixname_in\i) -- +(\bbpenetrate,0) coordinate (\fixname_in\i')}\fi 
            \ifnum #2=0{} \else foreach \i in {1,...,#2} {
            ($(\fixname.north east)!{\i/(#2+1)}!(\fixname.south east)$) +(-\bbpenetrate,0) coordinate (\fixname_out\i') -- +(\bbportlen,0) coordinate (\fixname_out\i)}\fi;
         }}}
  },
  bb name/.style={
    append after command={
      \pgfextra{\node[anchor=north] at (\fixname.north) {#1};}
    }
  }
}

\newcommand{\adj}[5][30pt]{%[size] Cat L, Left, Right, Cat R.
\begin{tikzcd}[ampersand replacement=\&, column sep=#1]
  #2\ar[r, shift left=5pt, "#3"]
  \ar[r, phantom, "\scriptstyle\Rightarrow"]\&
  #5\ar[l, shift left=5pt, "#4"]
\end{tikzcd}
}

\newcommand{\xtickar}[1]{\begin{tikzcd}[baseline=-0.5ex,cramped,sep=small,ampersand 
replacement=\&,>={Straight Barb}, arrows=->]{}\ar[r,tick, "{#1}"]\&{}\end{tikzcd}}

\theoremstyle{definition}
\newtheorem{definitionx}{Definition}[chapter]
\newtheorem{examplex}[definitionx]{Example}
\newtheorem{remarkx}[definitionx]{Remark}

\theoremstyle{plain}

\newtheorem{theorem}[definitionx]{Theorem}
\newtheorem{proposition}[definitionx]{Proposition}
\newtheorem{corollary}[definitionx]{Corollary}
\newtheorem{lemma}[definitionx]{Lemma}
\newtheorem{warning}[definitionx]{Warning}
\newtheorem*{theorem*}{Theorem}
\newtheorem*{proposition*}{Proposition}
\newtheorem*{corollary*}{Corollary}
\newtheorem*{lemma*}{Lemma}
\newtheorem*{warning*}{Warning}

\newenvironment{example}
  {\pushQED{\qed}\examplex}
  {\popQED\endexamplex}
  
 \newenvironment{remark}
  {\pushQED{\qed}\remarkx}
  {\popQED\endremarkx}
  
  \newenvironment{definition}
  {\pushQED{\qed}\definitionx}
  {\popQED\enddefinitionx}

\crefname{definitionx}{Definition}{Definitions}

\crefname{remarkx}{Remark}{Remarks}

\crefname{examplex}{Example}{Examples}

\crefname{warning}{Warning}{Warnings}

\DeclareSymbolFont{stmry}{U}{stmry}{m}{n}
\DeclareMathSymbol\fatsemi\mathop{stmry}{"23}

\DeclareFontFamily{U}{mathx}{\hyphenchar\font45}
\DeclareFontShape{U}{mathx}{m}{n}{
      <5> <6> <7> <8> <9> <10>
      <10.95> <12> <14.4> <17.28> <20.74> <24.88>
      mathx10
      }{}
\DeclareSymbolFont{mathx}{U}{mathx}{m}{n}
\DeclareFontSubstitution{U}{mathx}{m}{n}
\DeclareMathAccent{\widecheck}{0}{mathx}{"71}

\ExplSyntaxOn
\NewDocumentEnvironment{sequation}{O{\fontsize{15pt}{15pt}\selectfont
}b}
 {
  \yufip_sequation:nnn {equation}{#1}{#2}
 }{}
\NewDocumentEnvironment{sequation*}{O{\fontsize{16pt}{16pt}\selectfont
}b}
 {
  \yufip_sequation:nnn {equation*}{#1}{#2}
 }{}
\cs_new_protected:Nn \yufip_sequation:nnn
 {
  \begin{#1}
  \mbox{#2$\displaystyle#3$}
  \end{#1}
 }
\ExplSyntaxOff

\DeclarePairedDelimiter{\copair}{[}{]}

\DeclareMathOperator{\Hom}{Hom}

\DeclareMathOperator{\ob}{Ob}

\newcommand{\cat}[1]{\mathcal{#1}}%a generic category
\newcommand{\Cat}[1]{\mathsf{#1}}%a named category

\newcommand{\notsafe}[1]{{#1}}
\makeatletter
\NewDocumentCommand{\trackTermSymbol}{m m}{%
  \expandafter\newcommand\csname termraw@#1\endcsname{\notsafe{#2}}%
}
\NewDocumentCommand{\termraw}{m}{%
  \@ifundefined{termraw@#1}%
    {\PackageError{termraw}{Term '#1' not registered}{Use \string\trackTermSymbol{#1}{...} first.}}%
    {\@nameuse{termraw@#1}}%
}
\DeclareRobustCommand{\termref}[2]{\notsafe{\hypersetup{linkcolor=black}\hyperlink{term.#1}{#2}}}
\NewDocumentCommand{\trackTerm}{m m m}{%
  \trackTermSymbol{#2}{#3}%
  \newcommand{#1}{\termref{#2}{\termraw{#2}}}%
}
\NewDocumentCommand{\defineTerm}{m}{\notsafe{\hypertarget{term.#1}{\termraw{#1}}}}
\makeatother
\makeatletter
\ExplSyntaxOn
\NewDocumentCommand{\trackWord}{m m m}{%
  \tl_set:Nn \l_tmpa_tl {#1}%
  \tl_remove_all:Nn \l_tmpa_tl {~}%
  \tl_remove_all:Nn \l_tmpa_tl {-}%
  \clist_map_inline:nn {#1,#3}{%
    \cs_if_exist:cT { wordtgt@##1 }
      {\PackageError{trackword}{Inflection '##1' is already registered}%
        {Two definitions are claiming the same word.}}%
    \cs_gset:cpx { wordtgt@##1 } {\l_tmpa_tl}%
    \cs_gset:cpn { wordlbl@##1 } {#2}%
  }%
}
\ExplSyntaxOff
\newcommand{\word@err}[1]{\PackageError{trackword}{Word '#1' is not registered}%
  {Add it to a \string\trackWord\space inflection list.}}
\newcommand{\word@do}[2]{\@ifundefined{wordtgt@#1}{\word@err{#1}}%
  {\edef\word@@t{\@nameuse{wordtgt@#1}}%
   \edef\word@@d{term.\@nameuse{wordtgt@#1}}#2}}
\newcommand{\word@ucf}[1]{\word@ucf@#1\relax}
\def\word@ucf@#1#2\relax{\MakeUppercase{#1}#2}
\NewDocumentCommand{\word}{O{#2} m}%
  {\word@do{#2}{\expandafter\termref\expandafter{\word@@t}{#1}}}
\NewDocumentCommand{\Word}{O{#2} m}%
  {\word@do{#2}{\expandafter\termref\expandafter{\word@@t}{\word@ucf{#1}}}}
\NewDocumentCommand{\defWord}{O{#2} m}%
  {\word@do{#2}{\notsafe{\expandafter\hypertarget\expandafter{\word@@d}{\emph{#1}}}}}
\NewDocumentCommand{\wordref}{O{#2} m}{\word[#1]{#2}~(\cref{\@nameuse{wordlbl@#2}})}
\NewDocumentCommand{\Wordref}{O{#2} m}{\Word[#1]{#2}~(\cref{\@nameuse{wordlbl@#2}})}
\makeatother

\newcommand{\id}{\mathrm{id}}

\newcommand{\then}{\mathbin{\fatsemi}}

\trackTerm{\bang}{bang}{\,\mathbin{!}\,}

\newcommand{\too}{\longrightarrow}

\newcommand{\parto}{\rightharpoonup}
\newcommand{\To}[2][]{\xrightarrow[#1]{\tn{$#2$}}}

\newcommand{\Too}[1]{\xrightarrow{\;\;#1\;\;}}
\newcommand{\from}{\leftarrow}

\newcommand{\inj}{\rightarrowtail}
\newcommand{\jni}{\leftarrowtail}

\renewcommand{\iff}{\Leftrightarrow}

\newcommand{\tickar}{\xtickar{}}
\newcommand{\card}{\,^{\#}}
\trackTermSymbol{Ext}{\Cat{Ext}}
\newcommand{\Ext}[1]{\termref{Ext}{\termraw{Ext}}_{#1}}% def.\ just above eqn.set_L_basics

\newcommand{\op}{^\tn{op}}
\newcommand{\co}{^\tn{co}}
\newcommand{\dual}{^\ast}

\newcommand{\tn}[1]{\textnormal{#1}}

\newcommand{\tot}[1]{#1^\natural}

\newcommand{\nn}{\mathbb{N}}

\newcommand{\zz}{\mathbb{Z}}

\trackTerm{\smset}{smset}{\Cat{Set}}
\trackTerm{\fdvect}{fdvect}{\Cat{FDVect}}
\trackTerm{\smcat}{smcat}{\Cat{Cat}}
\newcommand{\dd}{\mathbb{D}}

\trackTerm{\para}{para}{\Cat{Para}}% def.\ at def.para
\trackTerm{\uniftrcat}{uniftrcat}{\Cat{UnifTrCat}}% def.\ at def.strict_uniform

\newcommand{\set}{\tn{-}\Cat{Set}}
\newcommand{\alg}{\tn{-}\Cat{Alg}}
\trackTerm{\cob}{cob}{1\tn{-}\Cat{Cob}}% def.\ at ex.traced_cob_alg

\trackTerm{\yon}{yon}{{\mathcal{y}}}
\trackTerm{\poly}{poly}{\Cat{Poly}}
\trackTerm{\polystar}{polystar}{\Cat{Poly}_\star}% def.\ at lemma.lex_star
\trackTerm{\smsetstar}{smsetstar}{\Cat{Set}_\star}% def.\ at lemma.lex_star

\trackTermSymbol{sem}{\llbracket-\rrbracket}% partial-function algebra, def.\ at eqn.partial_fn_algebra
\newcommand{\sem}[1]{\termref{sem}{\llbracket #1\rrbracket}}

\trackTerm{\IInt}{IInt}{\mathbb{I}\Cat{nt}}% def.\ at thm.IInt
\trackTerm{\Int}{Int}{\Cat{Int}}% def.\ at sec.int
\trackTerm{\W}{W}{\cat{W}}% the wiring diagram operad, def.\ at eqn.operad_W
\trackTerm{\Wp}{Wprime}{\cat{W}'}% the bypass operad, def.\ at eqn.w'

\trackTerm{\Tr}{Tr}{\operatorname{Tr}}% trace, def.\ at def.traced
\trackTerm{\iter}{iter}{\operatorname{Iter}}% iteration operator, def.\ at rem.iter
\trackTerm{\val}{val}{\operatorname{val}}% value of a pointed token, def.\ after prop.pointed_value
\trackTerm{\ev}{ev}{\operatorname{ev}}% evaluation of a string, def.\ at def.ev
\trackTerm{\stp}{stp}{\operatorname{stp}}% step, def.\ at def.operational
\trackTerm{\rem}{rem}{\operatorname{rem}}% remainder, def.\ at def.operational
\trackTerm{\dom}{dom}{\operatorname{dom}}% domain of definition, def.\ at def.extensive
\let\rawGamma\Gamma
\trackTermSymbol{Gamma}{\rawGamma}% elements functor, def.\ at lem.elements_functor
\renewcommand{\Gamma}{\termref{Gamma}{\termraw{Gamma}}}

\newcommand{\inleg}{\mathrm{in}}% inner leg: kept by the cell inside
\newcommand{\outleg}{\mathrm{out}}% outer leg: kept by the cell outside

\newcommand{\biglens}[2]{
     \begin{bmatrix}{\vphantom{f_f^f}#2} \\ {\vphantom{f_f^f}#1} \end{bmatrix}
}
\newcommand{\littlelens}[2]{
     \begin{bsmallmatrix}{\vphantom{f}#2} \\ {\vphantom{f}#1} \end{bsmallmatrix}
}
\newcommand{\lens}[2]{
  \relax\if@display
     \biglens{#1}{#2}
  \else
     \littlelens{#1}{#2}
  \fi
}

\newcommand{\hh}[2][]{#1 \tn{#2} #1}
\newcommand{\qqand}{\hh[\qquad]{and}}

\renewcommand{\iff}[1][\;\;]{#1\Leftrightarrow#1}

\newcommand{\qimplies}{\hh[\quad]{$\implies$}}
\newcommand{\qqimplies}{\hh[\qquad]{$\implies$}}
\newcommand{\qqiff}{\hh[\qquad]{iff}}
\newcommand{\qqor}{\hh[\qquad]{or}}
\newcommand{\qqmeans}{\hh[\qquad]{means}}
\newcommand{\qqie}{\hh[\qquad]{i.e.,}}

\newcommand{\thanksAFOSR}[1]{This material is based upon work supported by the Air Force Office of Scientific Research under award number #1}

\trackWord{computed elementwise by iteration}{def.trace_iteration}{}

\trackWord{leg}{sec.operational}{legs}
\trackWord{corner}{sec.operational}{corners}
\trackWord{enabled}{def.operational}{}
\trackWord{executable}{def.executable}{}
\trackWord{operational}{def.operational}{operational structure}

\trackWord{pre-configuration}{def.preconfiguration}{pre-configurations}
\trackWord{environment legs}{def.preconfiguration}{environment leg}
\trackWord{inner leg}{def.preconfiguration}{inner legs}
\trackWord{outer leg}{def.preconfiguration}{outer legs}
\trackWord{quiescent}{def.preconfiguration}{}
\trackWord{token}{def.token}{tokens}
\trackWord{value}{def.token}{values}
\trackWord{enables}{def.token}{enable,enabling}
\trackWord{configuration}{def.configuration}{configurations}

\trackWord{step}{def.run}{steps,stepping}
\trackWord{run}{def.run}{runs}
\trackWord{result}{def.run}{results}

\trackWord{initial configuration}{def.init_config}{initial configurations}

\allowdisplaybreaks
\setsecnumdepth{section}
\settocdepth{section}
\begin{document}

\title{\huge The compact double category $\IInt(\poly_\star)$\\models control flow and data transformations}
%\author{
%	\vspace{1mm}
%  \begin{minipage}{0.38\textwidth}
%    \centering
%    \textbf{Grigory Kondyrev}\\\vspace{-1mm}
%    \small\itshape Noeon Research
%  \end{minipage}
%  \hspace{1em}
%  \begin{minipage}{0.38\textwidth}
%    \centering
%    \textbf{David I.\ Spivak}\\\vspace{-1mm}
%    \small\itshape Topos Institute
%  \end{minipage}
%}
%\date{\vspace{-2mm}}
\author{}\date{\vspace{-2cm}}

\maketitle

\begin{center}
\begin{tabular}{ccc}
\large \textbf{Grigory Kondyrev}&~\hspace{1.3cm}~&\large \textbf{David I.\ Spivak}\\[1.5mm]
\small\itshape Noeon Research&&\small\itshape Topos Institute \& Conexus AI
\end{tabular}
\end{center}
\vspace{.5cm}

\begin{abstract}
Hasegawa showed that control flow in programming languages---while loops and if-then-else statements---can be modeled using traced cocartesian categories, such as the category $\smsetstar$ of pointed sets. In this paper we define an operad $\W$ of wiring diagrams that provides syntax for categories whose control flow moreover includes data transformations, including deleting, duplicating, permuting, and applying pre-specified functions to variables. In the most basic version, the operad underlies $\Int(\polystar)$, where $\Int(\cat{T})$ denotes the free compact category on a traced category $\cat{T}$, as defined by Joyal, Street, and Verity; to do so, we show that $\polystar$, as well as any multivariate version of it, is traced. We show moreover that whenever $\cat{T}$ is uniform---a condition also defined by Hasegawa and satisfied by $\polystar$---the resulting $\Int$-construction extends to a double category $\IInt(\cat{T})$, which is compact in the sense of Patterson. Finally, we construct an operational structure on the double categories $\IInt(\polystar)$ and $\IInt(\smsetstar)$ by which one can track trajectories as they move through the control flow associated to a wiring diagram, and we prove that the resulting operational semantics is computationally adequate, in the sense of Plotkin, for the denotational semantics in which a filled wiring diagram evaluates to a partial function.
\end{abstract}

%\tableofcontents*

\chapter{Introduction}\label[section]{chap.intro}

In this paper we give a wiring diagram syntax for modeling control flow and data transformations, as found throughout computer programming. The wiring diagram syntax looks like this:
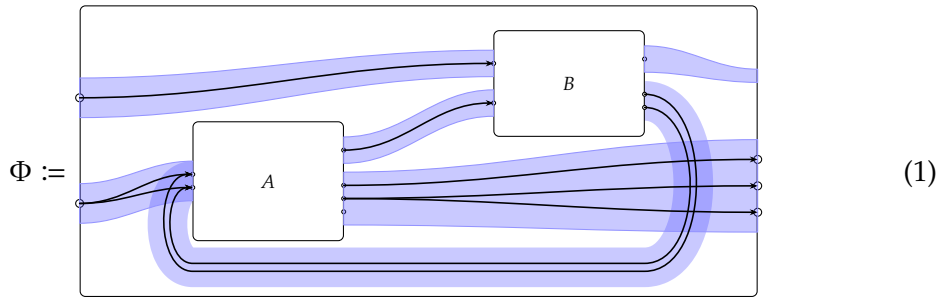
\begin{equation}\label{eqn.WD_poly}
\Phi\coloneqq\begin{tikzpicture}[WD, baseline=(BoxA)]
  \begin{pgfonlayer}{background}
    % only one region per side on Outer2 now
    \wdboxUnits[minimum width=9cm]{Outer2}{}{1,1}{0,3}{9}
  \end{pgfonlayer}

  % Box A (half size, left)
  \begin{scope}[shift={(Outer2.center)},xshift=-\innerOffset,
                yshift=-.4cm, transform shape,scale=\innerScale]
    \wdboxUnits[minimum width=4cm]{BoxA}{\Large $A$}{2}{1,3}{0}
  \end{scope}

  % Box B (half size, right, lifted)
  \begin{scope}[shift={(Outer2.center)},xshift=\innerOffset,
                yshift=.9cm,transform shape,scale=\innerScale]
    \wdboxUnits[minimum width=4cm]{BoxB}{\Large $B$}{1,1}{1,2}{0}
  \end{scope}

	% Tubes

  \tubeN{Outer2}{L}{2}{BoxA}{L}{1}{0}
  \draw[data wire] (Outer2-L-2-D1)
    to[out=0,in=180] (BoxA-L-1-D1);
  \draw[data wire] (Outer2-L-2-D1)
    to[out=0,in=180] (BoxA-L-1-D2);

	\tubeN{Outer2}{L}{1}{BoxB}{L}{1}{1}

  \tubeN{BoxA}{R}{2}{Outer2}{R}{2}{2}
  \draw[data wire] (BoxA-R-2-D2)
    to[out=0,in=180] (Outer2-R-2-D3);

  \tubeN{BoxA}{R}{1}{BoxB}{L}{2}{1}

	\tubeN{BoxB}{R}{1}{Outer2}{R}{1}{0}

  \loopTubeSmooth[below]{BoxB}{R}{2}{BoxA}{L}{1}{2}{BoxA}
\end{tikzpicture}
\end{equation}
Here we see two inner boxes, $A$ and $B$, and one outer box. A box is formally a pair, e.g.\ $A=(A^-,A^+)$, consisting of a left side and a right side---representing input and output respectively---each with some number of blue regions that themselves contain some number of ports. These boxes can be represented by pairs of polynomial functors, as we will explain in subsequent sections; for example $A^-\coloneqq\yon^2$ and $A^+\coloneqq\yon+\yon^3$. Connecting these boxes are blue tubes that carry the control flow---i.e.\ case logic, sequential computation, and while loops---and within them thin black wires that carry the data as it is passed around the diagram. 

The syntax is \emph{operadic} in the sense that the boxes are the objects, and the diagrams are the multimorphisms, of a colored operad $\W$: a diagram like $\Phi$ has several input boxes and one output box, and these diagrams compose by nesting: if $A$ or $B$ is itself filled by a still smaller diagram of the same sort, substituting it into $\Phi$ yields another, more detailed diagram, and such substitutions are associative and unital. We make this operad precise in \cref{sec.int_polystar_wd_syntax}; see \eqref{eqn.nesting_pic} for a depiction of the nesting operation.

This syntax is formalized categorically by a compact closed category denoted $\Int(\polystar)$, where $\polystar$ is the traced monoidal category $\polystar\coloneqq 1/\poly$ of pointed polynomial functors, and $\Int(\cat{T})$ denotes Joyal-Street-Verity's free construction \cite[Proposition 5.2]{Joyal.Street.Verity:1996a} of a compact closed category from a traced monoidal category $\cat{T}$. For example, the diagram in \cref{eqn.WD_poly} is a map in $\Int(\polystar)$ of the form
\[
\Phi\colon\big((\yon^2,\yon+\yon^3)+(2\yon,\yon+\yon^2)\big)\tickar(2\yon,1+\yon^3).
\]
%i.e.\ it represents a certain morphism $\Phi\colon\yon+\yon^3+\yon+\yon^2+2\yon\to\yon^2+2\yon+1+\yon^3$ in $\polystar$.

Cocartesian traced monoidal categories can be understood as handling control flow \cite[Remark 6.7]{selinger2010survey}.%
\footnote{According to the B\"ohm-Jacopini theorem \cite{bohm1966flow}, handling control flow is sufficient for computing any computable function.}
A key example is $\smsetstar$, the category of sets and partial functions: given a partial function $f\colon A+U\parto B+U$, one obtains a partial function $\Tr^U_{A,B}(f)\colon A\parto B$ by running a ``while loop'': repeatedly running $f$ and plugging any $U$-outputs back in as inputs to $f$ until either we terminate with a $B$ or fail to do so.
In \cite[Definitions 3.1--3.2]{hasegawa2004uniformity}, Hasegawa introduced a naturality condition called \emph{uniformity}, which exists in certain traced monoidal categories, including $\smsetstar$.

It turns out that uniformity is a very powerful principle for traced categories. If $\cat{U}$ is uniform traced monoidal, then so is the category $\Cat{Fun}(\cat{C},\cat{U})$ for any category $\cat{C}$: both the monoidal structure and the trace can be given pointwise. This fact will be used to show that $\polystar$, as well as a multivariate version of it, is traced, as we stated earlier without justification. Moreover, the $\Int$ construction described above actually fits into a double category structure, $\IInt(\cat{U})$. We will discuss background on traced monoidal categories, the one-dimensional $\Int$ construction, uniformity, and the traced structure on $\polystar$ in \cref{chap.background}.

Before moving on, we briefly touch on what we mean by \emph{data transformations} in the title of this paper. At the most primitive level, this means deleting, duplicating, and permuting variables, e.g.\ the map
\begin{equation}\label{eqn.ex_data_transform}
((x_1,x_2,x_3)\mapsto (x_3,x_2,x_3))\colon X\times X\times X\to X\times X\times X,
\end{equation}
which is natural in $X$. In wiring diagrams, this is done by terminating, splitting, or swapping data wires within a control region. One can also add typing to the variables, i.e.\ have a set $L$ and a type $X_\ell$ for each $\ell:L$, and consider natural maps of the form $\prod_{a:A}X_{\ell(a)}\to\prod_{b:B}X_{\ell(b)}$. Even more generally, one could consider the oplax slice over some category $\cat{L}$:%
\footnote{In \eqref{eqn.ex_data_transform} these were $\cat{L}=1$, $A=B=3$, $F(`1`)=`3`, F(`2`)=`2`, F(`3`)=`3`$, and $\varphi=\id$.}
\begin{equation}\label{eqn.oplax_slice}
\begin{tikzcd}
	\cat{A}\ar[rd, "\ell_A"']&\ar[d, phantom, pos=.4, "\xRightarrow{\varphi}"]&\cat{B}\ar[dl, "\ell_B"]\ar[ll, "F"']\\&
	\cat{L}
\end{tikzcd}
\end{equation}
Note that a diagram such as \eqref{eqn.oplax_slice}, together with a functor $X\colon\cat{L}\to\smset$, induces a function $\lim_{a\in\cat{A}}X_{\ell_A(a)}\to\lim_{b\in\cat{B}}X_{\ell_B(b)}$. We refer to these maps as \emph{data transformations} because they are exactly the natural transformations between conjunctive database queries, and we note that they allow applying pre-specified functions ($X$ applied to maps in $\cat{L}$) to variables. One may always rewrite a multivariate polynomial such as $x_1x_2x_2+x_3x_3+1$ instead using the syntax $\yon^{\{1,2,2\}}+\yon^{\{3,3\}}+\yon^{\{\}}$. The case of nondiscrete $\cat{L}$ allows one to write diagrams in the exponents, e.g.\
\begin{equation}\label{eqn.spannypoly}
	\yon^{2\to 1\from 2}+\yon^{\{3,3\}}+\yon^{\{\}},
\end{equation}
if there is a map $2\to 1$ in $\cat{L}$. In this case, data transformations would include maps such as $\yon^{2\to 1\from 2}\to\yon^{\{2, 2\}}$. Whenever we discuss $\polystar$ below, we implicitly mean to include any of the above sorts of data transformations.

\cref{chap.int_poly_wds} is about using the operad $\W$ underlying $\Int(\polystar)$ as a wiring diagram syntax. After some preliminaries on wiring diagrams, we will discuss the wiring diagram syntax provided by $\W$ and hence justify the above discussion about \eqref{eqn.WD_poly}. We then show that algebras $F\colon\W\to\smset$ correspond to specifying what is allowed to ``fill'' each box in a wiring diagram. We will show that such functors $F$ give rise to traced monoidal categories in their own right. Finally we discuss what we call \emph{bypassing}, whereby certain data can be stored while the computation within a given box is running. This is achieved mathematically using a $\para$-construction. In \cref{ex.factorial} we consider the example of the factorial function as a case where bypassing is useful. 

Finally in \cref{chap.functoriality_and_int} we describe the double-categorical $\IInt$ construction on uniform traced categories, and show that the result is compact closed in the sense of Patterson \cite[Definition 1]{patterson2024toward}. We give a condition on a category $\cat{C}$, which we call being \emph{executable}, under which its category $\cat{C}_\star$ of partial maps is uniform traced and the corresponding double category $\IInt(\cat{C}_\star)$ is \emph{operational}; both $\smset$ and $\poly$ are executable. This in turn lets us track the step-by-step operation of the trace, computed by iteration, through factorizations of cells in the double category, yielding the \emph{trajectory} of control as tokens enter and leave boxes at various levels in a nested wiring diagram. The target of the development is computational adequacy in the sense of Plotkin \cite[Theorem 3.1]{Plotkin:1977a}; the paper culminates in \cref{thm.adequacy}, that a token's run through a filled wiring diagram terminates with a given output exactly when the denotational semantics returns that output (\cref{cor.adequacy_sem}).

\paragraph{Basic background and notation.}

We assume the reader is familiar with categories, coproducts $(0,+,\sum)$ and products $(1,\times,\prod)$, functors, natural transformations, monads and their Kleisli categories, symmetric monoidal categories (SMCs), lax and strong monoidal functors, string diagrams, and (colored) operads. See \cite{Leinster:2014a,macLane1998categories,selinger2010survey,fong2019seven} for background.

We denote the set of maps $A\to B$ in a category $\cat{C}$ by $\cat{C}(A,B)$ or $\Hom_\cat{C}(A,B)$. We allow ourselves to write either $g\circ f$ or $f\then g$ for the composite $\bullet\To{f}\bullet\To{g}\bullet$ and either $A$ or $\id_A$ for the identity on an object $A$. We denote the initial (resp.\ terminal) object of a category by $0$ (resp.\ $1$), and we write $\defineTerm{bang}$ for the unique map out of $0$ or into $1$ whenever its domain and codomain are clear from context. We will denote symmetric monoidal categories by $(\cat{C},I,\otimes)$, where $\cat{C}$ is a category, $I\in\ob(\cat{C})$ is the unit object, and $(\otimes)\colon\cat{C}\times\cat{C}\to\cat{C}$ is the multiplication, leaving implicit the unitor and associator; the braiding isomorphism for $c,d$ is also usually implicit but we sometimes write it as $\sigma_{c,d}\colon c\otimes d\to d\otimes c$.

We denote the category of sets by $\defineTerm{smset}$ and the category of small categories by $\defineTerm{smcat}$, and for any category $\cat{C}$, we write $\cat{C}\set\coloneqq\smset^\cat{C}=\Cat{Fun}(\cat{C},\smset)$ to denote the category of $\cat{C}$-sets, or copresheaves on $\cat{C}$. For any natural number $N:\nn$, we may write $N\coloneqq\{`1`,\ldots,`N`\}$ to denote a standardized set with $N$ elements; e.g.\ $0$ denotes the empty set, $1$ denotes a terminal set, etc. For any $A:\smset$, we write $\defineTerm{yon}^A\coloneqq\smset(A,-)\colon\smset\to\smset$ to denote the functor \emph{represented} by $A$; when $A=0$ we obtain $\yon^0=1$, the constant functor sending each $X\mapsto 1$. 

\paragraph{Background on polynomial functors.}

Since functors $\smset\to\smset$ are closed under coproducts, the notation $\sum_{i:I}\yon^{A_i}$ for a functor $\smset\to\smset$ has by now been defined; it is given by $X\mapsto\sum_{i:I}\smset(A_i,X)$. Such functors (sums of representables) are called \emph{polynomial functors} on $\smset$, even if $I$ or any $A_i$ is infinite; for example $\yon+1$ is the functor sending $X\mapsto X+1$, and $\yon^\nn+\zz$ is the functor sending $X\mapsto X^\nn+\zz$. 

For a polynomial $p=\sum_{i:I}\yon^{A_i}$ we call $I$ its set of \emph{positions} and, for each position $i:I$, we call $A_i$ its set of \emph{directions} at $i$; for instance $\yon^2+\yon$ has two positions, with two and one direction respectively. Since $p(1)=\sum_{i:I}1=I$, we may write $p(1)$ for the set of positions of $p$. And for each position $i:p(1)$, we write $p[i]\coloneqq A_i$ to denote the set of directions at $i$, so that
\begin{equation}\label{eqn.poly_pd}
p=\sum_{i:p(1)}\yon^{p[i]}.
\end{equation}

We denote the category of polynomial functors and natural transformations by $\defineTerm{poly}$; it is the free completely distributive category on one object, meaning it has all sums and products, they distribute as usual, and no further imposed structure or properties. Coproducts and products in $\poly$ are given by the usual polynomial sum $(0,+,\sum)$ and product $(1,\times,\prod)$, e.g.\ $(\yon+1)\times(\yon+1)\cong\yon^2+2\yon+1$. For a set $A$ and polynomial $p$ we write $Ap\coloneqq\sum_{a:A}p$, e.g.\ $2\yon=\yon+\yon$, and we write $A$ for the constant polynomial $A\yon^0$. We sometimes write $pq\coloneqq p\times q$, e.g.\ $A\yon^2=\sum_{a:A}\yon^2\cong(\sum_{a:A}\yon^0)\times\yon^2=A\times\yon^2$.

Using the universal property of coproducts, the Yoneda lemma, the definition $\yon^J(I)=\smset(J,I)$, and distributivity, the set of morphisms $\varphi\colon p\to p'$ in $\poly$---i.e.\ natural transformations of the underlying functors---may be computed as follows:
\begin{align}\label{eqn.poly_map}
\nonumber
\poly(p,p')&=\poly\Big(\sum_{i:p(1)}\yon^{p[i]},\ \sum_{j:p'(1)}\yon^{p'[j]}\Big)\\\nonumber
&\cong\prod_{i:p(1)}\poly\Big(\yon^{p[i]},\ \sum_{j:p'(1)}\yon^{p'[j]}\Big)\\\nonumber
&\cong\prod_{i:p(1)}\;\sum_{j:p'(1)}\;\smset\big(p'[j],\,p[i]\big)\\
&\cong\sum_{\varphi_1\colon p(1)\to p'(1)}\;\prod_{i:p(1)}\;\smset\big(p'[\varphi_1(i)],\,p[i]\big).
\end{align}
The two factors of the bottom line exhibit $\varphi$ as a \emph{forward} function on positions $\varphi_1\colon p(1)\to p'(1)$ (the component of the natural transformation $\varphi$ at $1:\smset$) together with, for each position $i:p(1)$, a \emph{backward} function on directions $\varphi[i]\colon p'[\varphi_1(i)]\to p[i]$. This ``forward on positions, backward on directions'' pattern is the basic intuition for maps of polynomials, and it informs the wiring-diagram syntax of \cref{sec.wds}.

The next paragraph introduces a typed generalization $\smset[\cat{L}]$ of $\poly$ that recurs throughout the paper. On a first reading, however, one may safely take $\cat{L}\coloneqq\{\yon\}$ to be the terminal category and read $\smset[\cat{L}]$ simply as $\poly$ everywhere, turning to the general case only when interested in more advanced---e.g.\ typed---data transformations (\cref{rem.duc_query}). This paper should not require additional background on polynomial functors, but the interested reader can see \cite{niu2024polynomial} for more.

For any set $L:\smset$, consider the category $\smset[\yon_1,\ldots,\yon_L]$ whose objects are polynomials in $L$-many variables, with coefficients and exponents given by objects in $\smset$, and whose morphisms are natural transformations between the induced functors $L\set\to\smset$. Even more generally, for any category $\cat{L}$, define the category $\smset[\cat{L}]$ whose objects are formal sums of representables $\sum_{i:I}\yon^{A_i}$, where each exponent $A_i:\cat{L}\set$ is an $\cat{L}$-set and $\yon^{A_i}\coloneqq\cat{L}\set(A_i,-)$ is the corresponding representable, and whose morphisms are the natural transformations between the corresponding functors $\cat{L}\set\to\smset$. This is exactly parallel to the definition of basic polynomial functors above, with $\smset$ replaced by $\cat{L}\set$; equivalently, $\smset[\cat{L}]$ is the coproduct completion of $(\cat{L}\set)\op$. As mentioned in \eqref{eqn.spannypoly}, one can think of its objects $p$ as polynomials with $\cat{L}$-labeled diagrams in the exponents, so when $\cat{L}=L$ is discrete, this returns the multivariate polynomial case. There is a fully faithful coproduct-preserving functor
\[
\defineTerm{Ext}_{\cat{L}}\colon\smset[\cat{L}]\to\Cat{Fun}(\cat{L}\set,\smset)
\]
sending each $A:(\cat{L}\set)\op$ to the functor $\cat{L}\set(A,-)$ and completing under coproducts. Specializing the construction $\smset[\cat{L}]$ to particular choices of $\cat{L}$ recovers
\begin{equation}\label{eqn.set_L_basics}
  \smset\cong\smset[0]
  \qqand
  \poly\cong\smset[\{\yon\}],
\end{equation} 
where $\{\yon\}$ denotes a category with one object. In both cases $\Ext{\cat{L}}$ is as expected. Following \cite{spivak2025functorial}, we may refer to objects in $\smset[\cat{L}]$ as polynomials or multivariate polynomials, even though the latter is a special case.

\paragraph{Acknowledgments.}
\thanksAFOSR{FA9550-23-1-0376} as well as by Noeon Research. We also thank Kris Brown and Evan Patterson for useful conversations. We thank Andrei Krutikov for explaining the original graphical representation that inspired this work and for generously sharing insights.

\chapter{Background on traced categories}\label[section]{chap.background}

In \cref{sec.def_trace}, we recall the notion of traced monoidal categories; in \cref{sec.int} we recall the free compact closed category construction $\Int$, in \cref{sec.uniformity} we recall the notion of uniformity, in \cref{sec.new_from_known} we give some results by which new uniform traced categories can be obtained from known ones, and in \cref{sec.int_polystar} we apply some of these to show that $\polystar$ is uniform traced.

\section{Definitions and examples of traced categories}\label[section]{sec.def_trace}
Intuitively, traced symmetric monoidal categories are monoidal categories that support wiring diagrams that look like this:
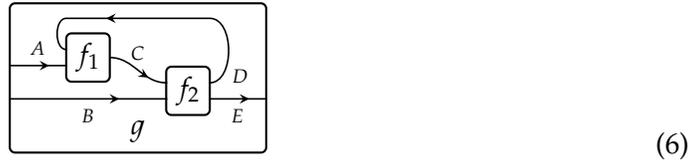
\begin{equation}\label{eqn.WD_traced}
	\begin{tikzpicture}[oriented WD, bb port sep=1, bb port length=0, bb min width=.2, bby=.2cm]
 		\node[bb={2}{1}] (Trf) {$f_1$};
  	\node[bb={2}{2}, below right=-1 and .5 of Trf] (Trg) {$f_2$};
  	\node[bb={0}{0}, fit={($(Trf.north west)+(.5,1)$) ($(Trg.south east)+(-.5,-1.5)$)}] (Tr) {};
  	\node[coordinate] at (Tr.west|-Trf_in2) (Tr_in1) {};
  	\node[coordinate] at (Tr.west|-Trg_in2) (Tr_in2) {};
  	\node[coordinate] at (Tr.east|-Trf_out1) (Tr_out1) {};
  	\node[coordinate] at (Tr.east|-Trg_out2) (Tr_out2) {};
  	\draw[ar] (Tr_in1) -- node[above, font=\tiny] {$A$} (Trf_in2);
  	\draw[ar] (Trf_out1) to node[above, font=\tiny] {$C$} (Trg_in1);
  	\draw[ar] (Tr_in2) -- node[below, font=\tiny] {$B$} (Trg_in2);
  	\draw[ar] (Trg_out2) -- node[below, font=\tiny] {$E$} (Tr_out2);
  	\draw[ar] let \p1=(Trg.east), \p2=(Trf.north west), \n1=\bbportlen, \n2=\bby in
  		(Trg_out1) to[in=0] node[right, pos=.2, font=\tiny] {$D$} (\x1+\n1,\y2+\n2) -- (\x2-\n1,\y2+\n2) to[out=180] (Trf_in1);
		\node[above=0pt of Tr.south] {$g$};
	\end{tikzpicture}
\end{equation}
where $f_1\colon A\otimes D\to C$ and $f_2\colon B\otimes C\to E\otimes D$, and where $g\colon B\otimes A\to E$ is defined using a \emph{trace structure} $g\coloneqq\Tr^{D}_{B\otimes A,E}((B\otimes f_1)\then f_2)$, as we now define. We recommend the reader consider the axioms of \cref{def.traced} in terms of their string diagram syntax, e.g.\ as shown in the \href{https://en.wikipedia.org/wiki/Traced_monoidal_category}{Wikipedia article on traced categories}.

%\begin{definition}[Traced category]\label[definition]{def.traced}
\begin{definition}[Traced category {\cite[p.~448]{Joyal.Street.Verity:1996a}}]\label[definition]{def.traced} %
Let $(\cat{T},I,\otimes)$ be a symmetric monoidal category. A \emph{trace} structure on it is a family of functions
\[
	\defineTerm{Tr}^U_{A,B}\colon\cat{T}(A\otimes U,B\otimes U)\to\cat{T}(A,B)
\]
parameterized by objects $A,B,U\in\ob(\cat{T})$, satisfying the following axioms:
\begin{description}[labelindent=1.5em]
	\item[Naturality in $A,B$:] for any $f\colon A'\to A$ and $g\colon B\to B'$, the following diagram commutes
	\begin{equation}\label{eqn.tr_naturality}
	\begin{tikzcd}[column sep=35pt]
		\cat{T}(A\otimes U,B\otimes U)\ar[r, "\Tr^U_{A,B}"]\ar[d, "{\cat{T}(f\otimes U,g\otimes U)}"']&
			\cat{T}(A,B)\ar[d, "{\cat{T}(f,g)}"]\\
		\cat{T}(A'\otimes U,B'\otimes U)\ar[r, "\Tr^U_{A',B'}"']&
			\cat{T}(A',B')
	\end{tikzcd}
	\end{equation}
	\item[Dinaturality in $U$:] for any $h\colon U\to V$, the following diagram commutes:
	\begin{equation}\label{eqn.tr_dinaturality}
	\begin{tikzcd}[column sep=60pt]
		\cat{T}(A\otimes V,B\otimes U)\ar[r, "{\cat{T}(A\otimes V,B\otimes h)}"]\ar[d, "{\cat{T}(A\otimes h,B\otimes U)}"']&
			\cat{T}(A\otimes V,B\otimes V)\ar[d, "\Tr^V_{A,B}"]\\
		\cat{T}(A\otimes U,B\otimes U)\ar[r, "\Tr^U_{A,B}"']&
			\cat{T}(A,B)
	\end{tikzcd}
  \end{equation}
	\item[Monoidality in $U$:] the following diagrams commute:
	\begin{equation}\label{eqn.monoidal_in_U}
	\begin{tikzcd}[column sep=24pt]
	\cat{T}(A\otimes I,B\otimes I)\ar[d, "\cong"']\ar[r, "\Tr^I_{A,B}"]&
		\cat{T}(A,B)\ar[d, equal]\\
	\cat{T}(A,B)\ar[r, equal]&
		\cat{T}(A,B)
	\end{tikzcd}	
	\hspace{.3in}
	\begin{tikzcd}[column sep=25pt]
		\cat{T}(A\otimes U\otimes V,B\otimes U\otimes V)\ar[d,"\Tr^V_{A\otimes U,B\otimes U}"']\ar[r, "\Tr^{U\otimes V}_{A, B}"]&
			\cat{T}(A,B)\ar[d, equal]\\
		\cat{T}(A\otimes U,B\otimes U)\ar[r, "\Tr^U_{A,B}"']&
			\cat{T}(A,B)
	\end{tikzcd}
	\end{equation}
	\item[Superposing:] the following diagram commutes:
	\[
	\begin{tikzcd}
		\cat{T}(A',B')\times\cat{T}(A\otimes U,B\otimes U)\ar[r, "(\otimes)"]\ar[d, "{\cat{T}(A',B')\times\Tr^U_{A,B}}"']&
			\cat{T}(A'\otimes A\otimes U,B'\otimes B\otimes U)\ar[d, "\Tr^U_{A'\otimes A,B'\otimes B}"]\\
		\cat{T}(A',B')\times\cat{T}(A,B)\ar[r, "(\otimes)"']&
			\cat{T}(A'\otimes A,B'\otimes B)
	\end{tikzcd}
	\]
	\item[Yanking:] the equation $\Tr^U_{U,U}(\sigma_{U,U})=\id_U$ holds, where $\sigma_{U,U}$ is the braiding.
\end{description}

\medskip
We refer to an SMC equipped with a trace structure as a \emph{traced category}.
\end{definition}

\begin{lemma}[Composition via tracing]\label[lemma]{lem.comp_via_trace} 
Let $(\cat{T},I,\otimes,\Tr)$ be a traced category and suppose $f\colon P\to Q\otimes R$ and $g\colon R\otimes S\to T$ are maps in $\cat{T}$. We can obtain their composite $P\otimes S\to Q\otimes T$ along $R$ by tensoring, swapping wires, and tracing:
\begin{equation}\label{eqn.comp_via_trace} 
	(Q\otimes g)\circ(f\otimes S)=\Tr^{R}_{P\otimes S,\,Q\otimes T}\big((Q\otimes\sigma_{R,T})\circ(f\otimes g)\circ(P\otimes\sigma_{S,R})\big).
\end{equation} 
\end{lemma} 
\begin{proof} 
By naturality of the braiding, the traced map factors as $\big((Q\otimes g)\otimes R\big)\circ(Q\otimes\tau)\circ\big((f\otimes S)\otimes R\big)$, where $\tau\colon R\otimes S\otimes R\to R\otimes S\otimes R$ interchanges the two copies of $R$. Hence by naturality \eqref{eqn.tr_naturality}, superposing, and yanking,
\[ 
	\Tr^{R}_{P\otimes S,\,Q\otimes T}\big((Q\otimes\sigma_{R,T})\circ(f\otimes g)\circ(P\otimes\sigma_{S,R})\big)=(Q\otimes g)\circ\Tr^{R}(Q\otimes\tau)\circ(f\otimes S)=(Q\otimes g)\circ(f\otimes S).\qedhere
\] 
\end{proof} 

\begin{remark}\label[remark]{rem.iter}
Let $(\cat{C},0,+)$ be a cocartesian monoidal category. Then a trace structure on it is equivalent to an \emph{iteration operator}
\[
\defineTerm{iter}^A_B\colon\cat{C}(A,B+A)\to\cat{C}(A,B)
\]
satisfying certain properties \cite[Proposition 6.8]{selinger2010survey}, such as the one depicted here:
\begin{equation}\label{eqn.iter}
\begin{tikzpicture}[oriented WD, bb small, bb port length=1mm]
	\node[bb={1}{2}] (f) {$f$};
	\node[circle, draw=black, fill=black, inner sep=.7pt, left=0mm] at (f_in1) (dot1) {};
	\draw (f_in1) -- (dot1) -- +(-3mm, 0);
	\draw (f_out1) -- +(3mm,0);
	\draw let \p1 = (f.south west), \p2 = (f.south east), \n1={\y1-\bby}, \n2=\bbportlen in
		(f_out2) to[out=0, in=0] (\x2+\n2,\n1) -- (\x1-\n2,\n1) to[out=180, in=260] (dot1);
	\node[bb={1}{2}, right=4 of f] (f2) {$f$};
	\node[bb={1}{2}, right=of f2_out2] (f3) {$f$};
	\node[circle, draw=black, fill=black, inner sep=.7pt, left=0mm] at ($(f2_out2)!.5!(f3_in1)$) (dot2) {};
	\node[circle, draw=black, fill=black, inner sep=.7pt, right=2mm of f3_out1] (dot3) {};
	\draw (f2_out2) -- (dot2) -- (f3_in1);
	\draw let \p1 = (f3.south west), \p2 = (f3.south east), \n1={\y1-\bby}, \n2=\bbportlen in
		(f3_out2) to[out=0, in=0] (\x2+\n2,\n1) -- (\x1-\n2,\n1) to[out=180, in=260] (dot2);
	\draw let \p1 = (f3.north west), \p2 = (f3.north east), \n1={\y1+\bby}, \n2=\bbportlen in
		(f2_out1) to (\x1-\n2,\n1) -- (\x2+\n2,\n1) to (dot3);
	\draw (f3_out1) -- (dot3) -- +(3mm, 0);
	\node at ($(f)!.5!(f2)$) {$=$};
\end{tikzpicture}
\end{equation}
Indeed, we can inter-define trace and iteration in terms of one another as follows:
\[
	\iter^A_B(f)\coloneqq\Tr^{A}_{A,B}(\nabla_A\then f)
	\qqand
	\Tr^U_{A,B}(g)\coloneqq (A+!_U)\then\iter^{A+U}_B(g\then(B+!_A+U))
\]
for $f\colon A\to B+A$ and $g\colon A+U\to B+U$, where $\nabla_A\colon A+A\to A$ is the fold and where $!_U\colon 0\to U$ and $!_A\colon 0\to A$ are the unique maps.
\end{remark}

Below are some examples of traced categories.
\begin{example}[Cancellative monoids are discrete traced categories]\label[example]{ex.cancellative}
	Let $(M,i,\cdot)$ denote a commutative monoid in $\smset$, and let $(\cat{M}, I,\odot)$ denote the corresponding discrete symmetric monoidal category. Recall that $M$ is said to be \emph{cancellative} if $a\cdot u=b\cdot u$ implies $a=b$, for all $a,b,u\in M$. For example, $(\nn,0,+)$ is cancellative because if $a+u=b+u$ then $a=b$. The monoid $(\{0,1\},1,\ast)$ is \emph{not} cancellative because $0\ast 0=1\ast 0$ but $0\neq 1$.
	
	The monoid $M$ is cancellative iff the monoidal category $\cat{M}$ is traced by the functions $\Tr^U_{A,A}(\id_{A\cdot U})\coloneqq\id_A$.
\end{example}

\begin{example}\label[example]{ex.fdvect}
The category $(\defineTerm{fdvect}_k, k, \otimes)$ of finite dimensional vector spaces over any field $k$ is traced. Here is a simple, though basis-dependent, way of thinking about the trace structure. Let $A,B,U$ be vector spaces with dimensions $a,b,u$ respectively. If each is equipped with a choice of basis, then an element $M:\fdvect_k(A\otimes U,B\otimes U)$ can be thought of as a $(b\times u)\times(a\times u)$ matrix, or equivalently a $b\times a$ block matrix of $u\times u$ blocks. The trace $\Tr^U_{A,B}(M):\fdvect_k(A,B)$ is the $b\times a$ matrix given by ``tracing''---summing up the diagonal elements---of each $u\times u$-block in $M$. 
\end{example}

\subsection{Partial maps in lextensive categories}\label[subsection]{sec.partial_maps}

Our main example of a traced category---and the one that will underlie the wiring-diagram syntax of \cref{sec.wds}---is the category $\smsetstar$ of sets and \emph{partial} functions, whose trace runs the ``while loop'' described in \cref{prop.setstar_traced}. Rather than handling partiality by hand, it is cleaner to obtain $\defineTerm{smsetstar}$, together with its polynomial variants $\defineTerm{polystar}$ and $\smset[\cat{L}]_\star$, uniformly as categories of partial maps in a \emph{lextensive} category. We therefore pause to recall lextensivity and the handful of related notions we will use; a reader interested only in $\smsetstar$ may keep the case $\cat{C}=\smset$ in mind throughout.

Intuitively, a category is lextensive when its coproducts behave like disjoint unions: any object lying over a coproduct $y_1+y_2$ decomposes canonically into the part over $y_1$ and the part over $y_2$ (this is the content of \cref{def.extensive}). This is exactly what lets us regard the domain of definition of a partial map as a genuine summand of its source, at least in the case of \emph{decidably} partial maps defined below.

\begin{definition}\label[definition]{def.extensive}
A category $\cat{C}$ is \emph{lextensive} \cite{carboni1993introduction} if it has coproducts $(0,+)$, it has finite limits $(1,\times_c)$, and whenever
\[
\begin{tikzcd}
	c_1\ar[r]\ar[d]&c\ar[d]&c_2\ar[d]\ar[l]\\
	d_1\ar[r, tail]&d_1+d_2&d_2\ar[l, tail]
\end{tikzcd}
\]
is a commutative diagram for which the bottom row is a coproduct diagram, the following holds: both squares are pullbacks (i.e.\ $c_1\cong c\times_{d_1+d_2}d_1$ and $c_2\cong c\times_{d_1+d_2}d_2$) iff the top row is a coproduct diagram (i.e.\ $c_1+c_2\cong c$).

If $\cat{C}$ is lextensive, a \emph{decidably partial map} $k\colon c\parto d$ is a span $c\jni c'\to d$, defined up to isomorphism, where $c'\inj c$ is a coproduct inclusion;
we call $c'$ the \emph{domain of definition}, write $c'=\defineTerm{dom}(k)$, and write $\tot{k}\colon\dom(k)\to d$ for the second leg.
We say an object $r$ is \emph{connected} if every map from $r$ to a coproduct factors through exactly one of the coproduct inclusions. We say that $\cat{C}$ has \emph{all points isolated} if every map out of $1$ is a coproduct inclusion and $1$ is connected.

Let $r$ be a connected object. For a map $v\colon r\to c$
and a coproduct inclusion $c_1\inj c$, we write $v\in c_1$ if $v$ factors through it. For a partial map $g\colon c\parto d$, we write $g(v)\coloneqq (g\circ v)\colon r\parto d$, and we say $g$ is \emph{defined} at $v$ if $v\in\dom(g)$ and \emph{undefined} otherwise.
\end{definition}

In a lextensive category, all coproduct inclusions are monic, so decidably partial maps are in particular partial in the usual sense, and they compose by pullback in a lextensive category. In $\smsetstar$, every partial map is decidably partial since every monic map is a coproduct inclusion. Throughout, in text and in diagrams, an arrow $c\parto d$ is a map that need not be total, while $c\to d$ is a total map, of which $c\inj d$---a coproduct inclusion---is a special case.

The following is straightforward, so we leave its proof to the reader.
\begin{lemma}\label[lemma]{lemma.lex_star}
Suppose that $\cat{C}$ is a lextensive category. The following categories are isomorphic and have the same collection of objects:
\begin{enumerate}
	\item the category of decidably partial maps in $\cat{C}$;
	\item the Kleisli category of the monad $c\mapsto c+1$ on $\cat{C}$.
\end{enumerate}
We denote this category-up-to-isomorphism by $\cat{C}_\star$.

If $\cat{C}$ has all points isolated then there is also an equivalence $\cat{C}_\star\simeq 1/\cat{C}$, where $1/\cat{C}$ is the coslice category of objects equipped with a map $1\to c$.
\end{lemma}

By \eqref{eqn.set_L_basics}, the following proposition in particular tells us that both $\smset$ and $\poly$ are lextensive with all points isolated.

\begin{proposition}\label[proposition]{prop.complemented_points}
For any category $\cat{L}$, the category $\smset[\cat{L}]$ is lextensive with all points isolated, and its connected objects are exactly the representables $\yon^A$.
\end{proposition}
\begin{proof}
This is well-known for $\smset$ and lextensivity is inherited for any functor category into a lextensive category, e.g.\ $\Cat{Fun}(\cat{L}\set,\smset)$. Since the functor $\Ext{\cat{L}}\colon\smset[\cat{L}]\to\Cat{Fun}(\cat{L}\set,\smset)$ creates coproducts and finite limits, $\smset[\cat{L}]$ is lextensive. Next, any natural transformation $1\to p=\sum_{i:I}\cat{L}\set(A_i,-)$ factors through some coproduct inclusion $1=\cat{L}\set(0,-)\to\cat{L}\set(A_i,-)$, and by Yoneda's lemma, it can be identified with a map $A_i\to 0$. Since $0$ is a strict initial object, $A_i\cong 0$. Finally, the connected objects are the representables. A representable $\yon^A$ is connected: writing each side of a coproduct $p+q$ as a sum of representables presents $p+q$ as a coproduct of representables, and by the Yoneda lemma as in \eqref{eqn.poly_map}, a map $\yon^A\to p+q$ factors through a single representable summand, which lies in $p$ or in $q$; it cannot factor through both, since it would then factor through their pullback $0$, and $\smset[\cat{L}](\yon^A,0)\cong\Ext{\cat{L}}(0)(A)=\emptyset$. In particular $1=\yon^0$ is connected, so all points are isolated. Conversely, let $p=\sum_{i:I}\yon^{A_i}$ be connected; the empty coproduct $0$ is not connected, so $I\neq\emptyset$. Any partition $I=I_1+I_2$ splits $p=p_1+p_2$, and the map $p_1+p_2\to 1+1$ factors through one of the two inclusions, say the first; then for $i:I_2$ the composite $\yon^{A_i}\to p_2\to p\to 1+1$ factors through both inclusions, hence through their pullback $0$---impossible---so $I_2=\emptyset$. Thus $I$ admits no nontrivial partition and is nonempty: $|I|=1$, and $p$ is representable.
\end{proof}

By \cref{def.extensive}, any map $h\colon x\to d_1+d_2$ in a lextensive category decomposes as $h=h_1+h_2$, where $h_j\colon x_j\to d_j$ is the pullback of $h$ along $d_j\inj d_1+d_2$.
\begin{lemma}\label[lemma]{lem.ofs_cstar}
If $\cat{C}$ is lextensive then $\cat{C}_\star$ has an orthogonal factorization system $(\cat{R},\cat{C})$, where $\cat{R}$ (``restriction'') is the opposite of the category of coproduct inclusions in $\cat{C}$ and the class $\cat{C}$ consists of the \emph{total} maps, those with $\dom(k)=c$; the factorization of $k\colon c\parto d$ is
\[
c\parto\dom(k)\to d.
\]
Every map of $\cat{R}$ has a section in $\cat{C}$.

For partial maps $f\colon b\parto d$ and $g\colon c\parto d$, the copair $\copair{f,g}\colon b+c\parto d$ satisfies $\dom\copair{f,g}=\dom(f)+\dom(g)$ and $\tot{\copair{f,g}}=\copair{\tot{f},\tot{g}}$; if moreover $d=d_1+d_2$, then $\tot{\copair{f,g}}_j=\copair{\tot{f}_j,\tot{g}_j}$ for $j=1,2$.
\end{lemma}
\begin{proof}
By definition every decidably partial map $k\colon c\parto d$ decomposes uniquely as $c\jni\dom(k)\to d$, where the first is a map $c\parto\dom(k)$ in $\cat{R}$ and the second is a map $\dom(k)\to d$ in $\cat{C}$. The coproduct inclusion $c'\inj c$ associated to any $r\colon c\parto c'$ in $\cat{R}$ is a section of $r$ in $\cat{C}$ because the pullback of a monic along itself is an isomorphism.

For orthogonality, suppose $r\then w=u\then t$ where $r\in\cat{R}$ and $t$ is total, and let $\sigma$ be the section of $r$. Then $\dom(u)=\dom(u\then t)=\dom(r\then w)\subseteq\dom(r)$, so $d\coloneqq\sigma\then u$ is the required lift: it satisfies $r\then d=u$ and $d\then t=\sigma\then r\then w=w$, and it is unique, since any lift $d'$ satisfies $d'=\sigma\then r\then d'=\sigma\then u=d$.

The span presenting $\copair{f,g}$ is the coproduct of the spans presenting $f$ and $g$, since a sum of coproduct inclusions is a coproduct inclusion.
Finally, write $\tot{f}_j\colon x_j\to d_j$ and $\tot{g}_j\colon y_j\to d_j$ for $j=1,2$. In the diagram
\[
\begin{tikzcd}[column sep=20pt]
	x_1+y_1\ar[d, "{\copair{\tot{f}_1,\tot{g}_1}}"']\ar[r, tail]&\dom(f)+\dom(g)\ar[d, "{\copair{\tot{f},\tot{g}}}"]&x_2+y_2\ar[l, tail]\ar[d, "{\copair{\tot{f}_2,\tot{g}_2}}"]\\
	d_1\ar[r, tail]&d_1+d_2&d_2\ar[l, tail]
\end{tikzcd}
\]
both squares commute, since their restrictions to the summands $x_j$ and $y_j$ are the pullback squares decomposing $\tot{f}$ and $\tot{g}$, and the top row is a coproduct diagram. Hence both squares are pullbacks by \cref{def.extensive}, and they decompose $\copair{\tot{f},\tot{g}}=\tot{\copair{f,g}}$.
\end{proof}

\begin{lemma}[Elementwise reasoning]\label[lemma]{lem.elementwise}
Let $\cat{C}$ be lextensive, let $r$ be a connected object, and let $v\colon r\to c$ be a map.
\begin{enumerate}
	\item\label{elt.exactly_one} If $c=c_1+c_2$, then either $v\in c_1$ or $v\in c_2$, and not both.
	\item\label{elt.split} If $h\colon c\to d_1+d_2$ is a map in $\cat{C}$, decomposed by extensivity as $h_1+h_2$ with $h_j\colon c_j\to d_j$, then $v\in c_j$ iff $h(v)\in d_j$.
	\item\label{elt.dom} If $g\colon c\parto d$ is defined at $v$ then $g(v)\colon r\parto d$ is total; if not, $\dom(g(v))=0$.
\end{enumerate}
\end{lemma}
\begin{proof}
\eqref{elt.exactly_one} is connectedness (\cref{def.extensive}). For \eqref{elt.split}: if $v$ factors through $c_j$ then $h(v)$ factors through $d_j$; conversely, we know $v\in c_k$ for $k=1$ or $k=2$ by \eqref{elt.exactly_one}, so $h(v)\in d_k$, and $k=j$ by \eqref{elt.exactly_one} applied to $h(v)$. For \eqref{elt.dom}: $\dom(g(v))\inj r$ is the pullback of $\dom(g)\inj c$ along $v$. A factorization of $v$ through $\dom(g)$ induces a section of this monic, making it an isomorphism, so $g(v)$ is total; if instead $v$ factors through the complement of $\dom(g)$, the same argument makes the complement of $\dom(g(v))$ an isomorphism, and hence $\dom(g(v))\cong 0$ by disjointness of coproducts.
\end{proof}

\begin{proposition}[Traced structure on $\smsetstar$]\label[proposition]{prop.setstar_traced}
The category $\smsetstar$ is cocartesian monoidal, and its monoidal structure is traced, with trace given elementwise by iteration: for every $g\colon A+U\parto B+U$ and every $u\in A$---setting $u_0\coloneqq u$, and $u_{s+1}\coloneqq g(u_s)$ whenever $s=0$, or $u_s$ is defined and lies in $U$---
\begin{equation}\label{eqn.elementwise_trace}
\Tr^U_{A,B}(g)(u)=u_s\ \text{ for the least $s$ with $u_s\in B$},\qquad\text{undefined if there is no such $s$}.
\end{equation}
\end{proposition}
\begin{proof}
For any monad on a cocartesian category, the corresponding Kleisli category has a cocartesian monoidal structure, so $\smsetstar$ is cocartesian by \cref{lemma.lex_star}. This monoidal structure is traced (\cite[Example 6.5]{selinger2010survey}), or equivalently by \cref{rem.iter}, it has an iteration operator, which we now construct. Given $f\colon A\parto B+A$, i.e.\ a pointed function $f\colon A+1\to B+A+1$: using the functor $p\coloneqq \yon+B$, it can be considered as a coalgebra $A+1\to p(A+1)$ equipped with a map from the coalgebra $1\to p(1)$. The terminal $p$-coalgebra is $(\nn\times B)+1$, and the induced map $1\to(\nn\times B)+1$ sends $1\mapsto 1$. Thus the following composite is pointed, and we define $\iter^A_B(f)$ to be the map $A\parto B$ it represents:
\[
	A+1\to(\nn\times B)+1\to B+1.
\]

For the displayed formula, set $g'\coloneqq g\then(B+!_A+U)\colon A+U\parto B+A+U$, so that $\Tr^U_{A,B}(g)(u)=\iter^{A+U}_B(g')(u)$ by \cref{rem.iter}, and the iterates of $g'$ from $u$ are exactly $u_1,u_2,\ldots$. The function $A+U+1\to(\nn\times B)+1$ sending each $x$ to $(s-1,x_s)$ for the least $s\geq1$ with $x_s\in B$---where $x_0\coloneqq x$ and $x_{s+1}\coloneqq g'(x_s)$---and to the point when there is no such $s$, is a map of $p$-coalgebras $(A+U+1,g')\to(\nn\times B)+1$, hence is the map defining $\iter^{A+U}_B(g')$; postcomposing with $(\nn\times B)+1\to B+1$ gives \eqref{eqn.elementwise_trace}.
\end{proof}

\section{The \texorpdfstring{$\termraw{Int}$}{Int} construction}\label[section]{sec.int}
Recall that a compact category (also known as a compact closed category), is a symmetric monoidal category $(\cat{C},I,\otimes)$ for which every object $X$ has a dual $X\dual$, in the sense that there are maps
\[
\eta_X\colon I\to X\otimes X\dual
\qqand
\epsilon_X\colon X\dual\otimes X\to I
\]
such that the composite $X\To{\eta_X\otimes X} X\otimes X\dual\otimes X\To{X\otimes\epsilon_X}X$ is equal to $\id_X$, and the composite $X\dual\To{X\dual\otimes\eta_X} X\dual\otimes X\otimes X\dual\To{\epsilon_X\otimes X\dual}X\dual$ is equal to $\id_{X\dual}$. 

\begin{example}
The category $\fdvect$ is compact. Indeed, the dual $V\dual$ of $V$ is the vector space of linear maps $V\to k$, the map $\epsilon_V$ is given by evaluation, and $\eta_V$ can be defined as the sum $\sum_{b:B}b\otimes b\dual$ over any given $V$-basis $B$ and its dual. 

The trace structure discussed in \cref{ex.fdvect} is the one that exists on any compact category, as we now explain.
\end{example}

Every compact category has a trace structure: given $f\colon A\otimes U\to B\otimes U$, define
\[
\begin{tikzcd}[column sep=40pt]
	A\ar[r, dashed, "\Tr^U_{A,B}(f)"]\ar[d, "A\otimes\eta_U"']&B\\
	A\otimes U\otimes U\dual\ar[r, "f\otimes U\dual"']&B\otimes U\otimes U\dual\ar[u,"B\otimes\epsilon_U"']
\end{tikzcd}
\]

In \cite[p.~448 and Proposition 5.2]{Joyal.Street.Verity:1996a}, Joyal, Street, and Verity defined traced categories and showed that the above \emph{underlying traced category} functor has a left adjoint, which they called the $\defineTerm{Int}$-construction, sending any traced category to the free compact category on it.%
\footnote{The authors of \cite[Proposition 5.2]{Joyal.Street.Verity:1996a} claimed that $\Int\colon\Cat{TrCat}\to\Cat{CompCat}$ was furthermore 2-functorial in the natural transformations between traced functors, but this was later proved false in \cite[Example 1 and Proposition 1]{hasegawa2010note}, where it was shown that $\Int$ is only 2-functorial with respect to natural isomorphisms.}
\[
\adj{\Cat{TrCat}}{\Int}{\Cat{Und}}{\Cat{CompCat}}.
\]
For a traced category $(\cat{T},I,\otimes)$, the category $\Int(\cat{T})$ has the following objects, morphisms, monoidal structure, and duals:
\begin{equation}\label{eqn.int_construction}
\begin{aligned}
	\ob\Int(\cat{T})&\coloneqq\{(A^-,A^+)\in\ob(\cat{T})\times\ob(\cat{T})\},\\
	\Int(\cat{T})((A^-,A^+),(B^-,B^+))&\coloneqq\cat{T}(B^-\otimes A^+,B^+\otimes A^-),\\
	(A^-,A^+)\otimes(B^-,B^+)&\coloneqq(A^-\otimes B^-,A^+\otimes B^+),\\
	(A^-,A^+)\dual&\coloneqq(A^+,A^-).
\end{aligned}
\end{equation}
Its monoidal unit is $(I,I)$. We denote maps in $\Int$ by $(A^-,A^+)\tickar(B^-,B^+)$. Given $f\colon A\otimes U\otimes A'\to B\otimes U\otimes B'$, we allow ourselves to write $\Tr^U_{A\otimes A',B\otimes B'}(f)$ to denote what technically should be $\Tr^U_{A\otimes A',B\otimes B'}((A\otimes\sigma_{A',U})\then f\then(B\otimes\sigma_{U,B'}))$, where the $\sigma$'s are braidings. The composite of $(A^-,A^+)\xtickar{f}(B^-,B^+)\xtickar{g}(C^-,C^+)$ is defined by:
\begin{equation}\label{eqn.int_composite}
\Tr^{B^+}_{C^-\otimes A^+,C^+\otimes A^-}
\left(
	C^-\otimes B^+\otimes A^+\To{g\otimes A^+}
	C^+\otimes B^-\otimes A^+\To{C^+\otimes f}
	C^+\otimes B^+\otimes A^-
\right).
\end{equation}

\begin{remark}[Relation to geometry of interaction and game semantics]\label[remark]{rem.goi}
The \(\Int\) construction is the categorical abstraction underlying the modern categorical account of \emph{geometry of interaction} (GoI). In GoI, an object \((A^-,A^+)\) may be read as a two-sided interface of inputs and outputs; a morphism is a pattern of interaction between such interfaces; and composition \eqref{eqn.int_composite} is \emph{execution}. Indeed two morphisms interact using the feedback process, implemented by the trace, until a result emerges or the process diverges. Girard's original formulation of GoI was operator-algebraic; this traced-monoidal view was developed later by Abramsky, Blute, Panangaden, Haghverdi, and Scott \cite{Abramsky.Blute.Panangaden:1999a,abramsky2002geometry,Abramsky:2005a}. 

In our setting the base is \(\polystar\), and the partiality of its trace (\cref{prop.setstar_traced})---a ``while loop'' that may fail to terminate---captures the possible non-termination of GoI execution. The same construction underlies the compact closed (hence $\ast$-autonomous) structure found in \emph{game semantics}: the two sides $A^-,A^+$ of an object play the roles of Opponent and Player, the dual $(A^-,A^+)\dual=(A^+,A^-)$ interchanges these roles, and composition is ``parallel composition plus hiding''
%\cite{abramsky1994games}.
\cite[p.~13]{abramsky1994games}. %

Our aim in this paper is somewhat different: rather than interpreting a logic or a functional language, we use $\Int(\polystar)$ as an operadic \emph{syntax} for control flow together with data transformations, and we equip it with semantics---denotationally via $\W$-algebras such as the partial-function algebra \eqref{eqn.partial_fn_algebra} (\cref{sec.w_alg_cats}), and operationally via the \emph{trajectories} of \cref{sec.trajectories}. Here the polynomial structure of $\poly$ carries the data along the wires---typed, if we generalize to the multivariate version $\smset[\cat{L}]$---and the case structure on each interface, while the trace supplies the looping control flow.
\end{remark}

In the reading of \cref{rem.goi}, a string 
\[
(X_0^-,X_0^+)\xtickar{g_0}(X_1^-,X_1^+)\xtickar{g_1}\cdots\xtickar{g_{n-1}}(X_n^-,X_n^+)
\]
of composable maps need not execute two layers at a time: it executes all at once, tracing out every inner interface $(X_i^-,X_i^+)$, $0<i<n$, in one shot. In the trace of \eqref{eqn.execution} below, each $X_i^-$ (resp.\ $X_i^+$) occurs in exactly two of the factors: in the source of $g_{i-1}$ (resp.\ $g_i$) and in the target of $g_i$ (resp.\ $g_{i-1}$). The trace we see in \cref{eqn.execution} ties them together.

Extending the convention above, we suppress the braidings that rearrange $\otimes$-factors.

\begin{lemma}[Execution formula]\label[lemma]{lem.execution}
Let $g_i\colon(X_i^-,X_i^+)\tickar(X_{i+1}^-,X_{i+1}^+)$, for $0\leq i<n$, be maps in $\Int(\cat{T})$, and set $M\coloneqq\bigotimes_{0<i<n}(X_i^-\otimes X_i^+)$. Then
\begin{equation}\label{eqn.execution}
g_{n-1}\circ\cdots\circ g_0
=\Tr^{M}_{X_n^-\otimes X_0^+,\ X_n^+\otimes X_0^-}
\big(g_{n-1}\otimes\cdots\otimes g_0\big).
\end{equation}
\end{lemma}
\begin{proof}
For $n=1$ the trace is over the monoidal unit, and \eqref{eqn.execution} holds by monoidality \eqref{eqn.monoidal_in_U}. For $n\geq2$, set $h\coloneqq g_{n-2}\circ\cdots\circ g_0$ and $M'\coloneqq\bigotimes_{0<i<n-1}(X_i^-\otimes X_i^+)$, so that $h=\Tr^{M'}\big(g_{n-2}\otimes\cdots\otimes g_0\big)$ by induction. In the composite \eqref{eqn.int_composite}
\[
g_{n-1}\circ h
=\Tr^{X_{n-1}^+}\big((X_n^+\otimes h)\circ(g_{n-1}\otimes X_0^+)\big),
\]
the maps $g_{n-1}$ and $h$ are composed along $X_{n-1}^-$, a factor in both the target of $g_{n-1}$ and the source of $h$; by \eqref{eqn.comp_via_trace} that composition is itself a trace,
\[
	(X_n^+\otimes h)\circ(g_{n-1}\otimes X_0^+)=\Tr^{X_{n-1}^-}\big(g_{n-1}\otimes h\big),
\]
so merging the two traces by monoidality \eqref{eqn.monoidal_in_U} we obtain
\begin{align*}
g_{n-1}\circ h
=\Tr^{X_{n-1}^+}\Tr^{X_{n-1}^-}\big(g_{n-1}\otimes h\big)
&=\Tr^{X_{n-1}^-\otimes X_{n-1}^+}\Big(g_{n-1}\otimes\Tr^{M'}\big(g_{n-2}\otimes\cdots\otimes g_0\big)\Big)\\
&=\Tr^{M}\big(g_{n-1}\otimes\cdots\otimes g_0\big),
\end{align*}
where the last equality is again by superposing and monoidality.
\end{proof}

\begin{example}[Integer construction]
Given the monoid of natural numbers $(\nn,0,+)$, one can define the group $\zz$ of integers to have underlying set $(\nn\times\nn)/\sim$, where $(a^-,a^+)\sim(b^-,b^+)$ iff $b^-+a^+=b^++a^-$; this discrete category is equivalent to the essentially discrete groupoid $\Int(\cat{N})$, where $\cat{N}$ is the discrete traced category from \cref{ex.cancellative}. Moreover, the zero, addition, and negation operations in $\zz$ are the monoidal unit, monoidal product, and duals in $\Int(\cat{N})$. 
\end{example}

\begin{example}
Consider the cocartesian traced category $(\smsetstar,0,+,\Tr)$ from \cref{prop.setstar_traced}. Then $\Int(\smsetstar)$ has objects given by pairs of sets $(A^-,A^+)$. We can denote such an object by a box with $A^-$-many regions on the left and $A^+$-many regions on the right.

For example, let $A\coloneqq(1,2)$, $B\coloneqq(2,2)$, and $C\coloneqq(2,2)$. The following depicts a map $A+B\tickar C$, sending the two elements of $C^-$ to one element from $B^-$ and the one from $A^-$, sending the two elements of $A^+$ to an element from $B^-$ and an element from $C^+$, and sending the two elements of $B^+$ to an element of $C^+$ and an element of $A^-$:
\begin{equation}\label{eqn.WD_set}
\begin{tikzpicture}[WD, baseline=(BoxA)]
  \begin{pgfonlayer}{background}
    % only one region per side on Outer2 now
    \wdboxUnits[minimum width=7cm]{Outer2}{\raisebox{2.5cm}{$C$}}{0,0}{0,0}{6}
  \end{pgfonlayer}

  % Box A (half size, left)
  \begin{scope}[shift={(Outer2.center)},xshift=-1.5cm,
                yshift=-.25cm, transform shape,scale=\innerScale]
    \wdboxUnits[minimum width=4cm]{BoxA}{\Large $A$}{0}{0,0}{0}
  \end{scope}

  % Box B (half size, right, lifted)
  \begin{scope}[shift={(Outer2.center)},xshift=1.5cm,
                yshift=.5cm,transform shape,scale=\innerScale]
    \wdboxUnits[minimum width=4cm]{BoxB}{\Large $B$}{0,0}{0,0}{0}
  \end{scope}

	% Tubes

  \tubeN{Outer2}{L}{2}{BoxA}{L}{1}{0}
 
	\tubeN{Outer2}{L}{1}{BoxB}{L}{1}{0}

  \tubeN{BoxA}{R}{2}{Outer2}{R}{2}{0}

  \tubeN{BoxA}{R}{1}{BoxB}{L}{2}{0}

	\tubeN{BoxB}{R}{1}{Outer2}{R}{1}{0}

  \loopTubeSmooth[below]{BoxB}{R}{2}{BoxA}{L}{1}{0}{BoxA}
\end{tikzpicture}
\end{equation}

\end{example}

\section{Strict maps and uniformity}\label[section]{sec.uniformity}

In \cite[Definitions 3.1--3.2]{hasegawa2004uniformity}, Hasegawa makes the following definition.
\begin{definition}[Strict map, uniformity]\label[definition]{def.strict_uniform}
Let $\cat{T}$ be a traced category. A morphism $s\colon U\to V$ is called \emph{strict} if, for every pair of objects $A,B:\ob\cat{T}$ and morphisms $f\colon A\otimes U\to B\otimes U$ and $g\colon A\otimes V\to B\otimes V$, the following implication holds:
\begin{equation}\label{eqn.uniformity}
\begin{tikzcd}
	A\otimes U\ar[r, "f"]\ar[d, "A\otimes s"']&B\otimes U\ar[d, "B\otimes s"]\\
	A\otimes V\ar[r, "g"']&B\otimes V\ar[ul, phantom, "\tiny\textit{commutes}"]
\end{tikzcd}
%	f\then(B\otimes h) = (A\otimes h)\then g
	\qimplies
	\Tr^U_{A,B}(f) = \Tr^V_{A,B}(g).
\end{equation}
A traced category is called \emph{uniform} if every morphism is strict. A \emph{traced functor} $F\colon\cat{T}\to\cat{T}'$ between traced monoidal categories is a strong symmetric monoidal functor that preserves traces, in the sense that $F(\Tr^U_{A,B}(f))=\Tr^{FU}_{FA,FB}(Ff)$ for every $f$. We denote the category of uniform traced categories and traced functors by $\defineTerm{uniftrcat}$.
\end{definition}

Unlike the string diagrams for traced monoidal categories, the string diagram representation of strict maps or uniformity is not intuitive. However, the corresponding uniformity requirement (see \cite[Definition 5.3 and Theorem 5.2]{hasegawa2004uniformity}, \cite[Definition 2 and Figure 1]{goncharov2016complete}) for cocartesian traced categories
\[
	f'\then (B+s) = s\then g'
	\qimplies
	\iter^U_B(f') = s\then\iter^V_B(g'),
\]
shown in string diagrams below, is more intuitive, e.g.\ by employing \cref{eqn.iter}:
\[
\begin{tikzpicture}[oriented WD, bb small, bb port length=3mm, baseline=(g')]
	\node[bb port sep=3.5, bb={1}{2}] (f') {$f'$};
	\node[bb={1}{1}, right=.5 of f'_out2] (h1) {$s$};
	\node[bb port sep=3.5, bb={1}{2}, right=10 of f'] (g') {$g'$};
	\node[bb={1}{1}, left=.5 of g'_in1] (h2) {$s$};
	\draw (f'_out2) -- (h1_in1);
	\draw (f'_out1) -- (h1_out1|-f'_out1);
	\draw (h2_out1) -- (g'_in1);
	\node at ($(f')!.5!(g')$) {$=$};
	\begin{scope}[font=\tiny]
  	\node[above=0 of f'_in1] {$U$};
  	\node[above=0 of f'_out1] {$B$};
  	\node[above=0 of f'_out2] {$U$};
  	\node[above=0 of h1_out1] {$V$};
  	\node[above=0 of h2_in1] {$U$};
  	\node[above=0 of h2_out1] {$V$};
  	\node[above=0 of g'_out1] {$B$};
  	\node[above=0 of g'_out2] {$V$};
	\end{scope}
\end{tikzpicture}
\qqimplies
\begin{tikzpicture}[oriented WD, bb small, bb port length=1mm, baseline=(f')]
	\node[bb port sep=3.5, bb={1}{2}] (f') {$f'$};
	\node[bb port sep=3.5, bb={1}{2}, right=8 of f'] (g') {$g'$};
	\node[bb={1}{1}, left=1.5 of g'] (h) {$s$};
	\node[circle, draw=black, fill=black, inner sep=.7pt, left=2mm] at (f'_in1) (dot1) {};
	\node[circle, draw=black, fill=black, inner sep=.7pt] at ($(h)!.5!(g')$) (dot2) {};
	\draw (f'_in1) -- (dot1) -- +(-3mm, 0);
	\draw (f'_out1) -- +(3mm,0);
	\draw let \p1 = (f'.south west), \p2 = (f'.south east), \n1={\y1-\bby}, \n2=\bbportlen in
		(f'_out2) to[out=0, in=0] (\x2+\n2,\n1) -- (\x1-\n2,\n1) to[out=180, in=260] (dot1);
	\draw (h_in1) -- +(-3mm, 0);
	\draw (h_out1) -- (dot2) -- (g'_in1);
	\draw let \p1 = (g'.south west), \p2 = (g'.south east), \n1={\y1-\bby}, \n2=\bbportlen in
		(g'_out2) to[out=0, in=0] (\x2+\n2,\n1) -- (\x1-\n2,\n1) to[out=180, in=260] (dot2);
	\draw (g'_out1) -- +(3mm,0);
	\node at ($(f')!.5!(h)$) {$=$};
	\begin{scope}[font=\tiny]
 		\node[above left=0 of f'_in1] {$U$};
  	\node[above right =0 of f'_out1] {$B$};
  	\node[above right=0 of h_out1] {$V$};
  	\node[above left=0 of h_in1] {$U$};
  	\node[above right=0 of g'_out1] {$B$};
	\end{scope}	
\end{tikzpicture}
\]

\begin{lemma}\label[lemma]{lemma.hasegawa}
In any traced category, each isomorphism is strict, and the monoidal product $s\otimes s'$ of strict maps $s,s'$ is again strict. 
\end{lemma}
\begin{proof}
The first follows directly from dinaturality \eqref{eqn.tr_dinaturality} and the second follows from naturality \eqref{eqn.tr_naturality} and monoidality \eqref{eqn.monoidal_in_U}; see \cite[Lemmas 4.1, 4.3]{hasegawa2004uniformity}.
\end{proof}

Not all morphisms in a traced category are strict; for example, Hasegawa in \cite[Example 3.1]{hasegawa2004uniformity} shows that in $\fdvect$ a morphism is strict iff it is an isomorphism.

\begin{warning}\label[warning]{warn.cant_compose_strict}
The composite of strict maps need not be strict. See \cite[Proposition 4.1]{hasegawa2004uniformity}.
\end{warning}

We next come to our main example of interest.

\begin{proposition}[Hasegawa {\cite[Example 3.1]{hasegawa2004uniformity}}]\label[proposition]{prop.setstar_uniform}
The cocartesian traced category $\smsetstar$ is uniform.
\end{proposition}
\begin{proof}
Suppose given a commutative diagram of the following form in pointed sets:
\begin{equation}\label{eqn.coalg_map}
\begin{tikzcd}[sep=small]
	U+1\ar[rr, "f"]\ar[dd, "s"']&&
		B+U+1\ar[dd, "B+s"]\\
	&{\color{gray!80}1}\ar[ul, gray!80]\ar[ur, gray!80]\ar[dl, gray!80]\ar[dr, gray!80]\\
	V+1\ar[rr, "g"']&&
		B+V+1
\end{tikzcd}
\end{equation}
Taking $p\coloneqq\yon+B$ as in \cref{prop.setstar_traced}, diagram \eqref{eqn.coalg_map} constitutes a map $s\colon f\to g$ of $p$-coalgebras, and since $(\nn\times B)+1$ is terminal, we indeed have $\iter^U_B(f)=s\then\iter^{V}_B(g)\colon U+1\to B+1$, as desired.
\end{proof}

\section{New uniform traced categories from known ones}\label{sec.new_from_known}
Hasegawa shows in \cite[Proposition 6.1]{hasegawa2004uniformity} that if $\cat{T}$ is uniform traced, then its arrow category is traced, and sketches the same for functor categories $\cat{T}^{G}$ indexed by a graph $G$ \cite[p.~1008]{hasegawa2004uniformity}; he does not show that either is uniform, but the same argument does. The proof below is completely analogous.

Given a symmetric monoidal category $(\cat{T},I,\otimes)$ and a category $\cat{C}$, let $(I^\cat{C},\otimes^\cat{C})$ be the symmetric monoidal structure defined pointwise on $\cat{T}^\cat{C}$:
\[
	I^\cat{C}(c)\coloneqq I
  \qqand
  (F\otimes^\cat{C} G)(c)\coloneqq F(c)\otimes G(c).
\]

\begin{proposition}\label[proposition]{prop.uniform_traced_exponent}
If $(\cat{T},I,\otimes)$ is a uniform traced category, then so is $(\cat{T}^\cat{C},I^\cat{C},\otimes^\cat{C})$.
\end{proposition}
\begin{proof}
For $A,B,U\colon\cat{C}\to\cat{T}$, the proposed trace $\Tr^{U}_{A,B}\colon\cat{T}^\cat{C}(A\otimes U,B\otimes U)\to\cat{T}^\cat{C}(A,B)$ is again given pointwise, i.e.\ for any natural transformation $\alpha\colon A\otimes U\to B\otimes U$, the component of $\Tr^U_{A,B}(\alpha)$ at $c$ is given by
\[
\left(\Tr^U_{A,B}(\alpha)\right)_c\coloneqq\Tr^{U(c)}_{A(c),B(c)}(\alpha_c)\colon A(c)\to B(c).
\]

To see that this is natural, i.e.\ well-defined in $\cat{T}^\cat{C}$, suppose given any map $s\colon c\to c'$. The following diagram commutes because $\alpha$ is natural:
\[
\begin{tikzcd}[column sep=50pt]
	A(c)\otimes U(c)\ar[r, "\alpha_c"]\ar[d, "A(c)\otimes U(s)"']&B(c)\otimes U(c)\ar[r, "B(s)\otimes U(c)"]&B(c')\otimes U(c)\ar[d, "B(c')\otimes U(s)"]\\
	A(c)\otimes U(c')\ar[r, "A(s)\otimes U(c')"']&A(c')\otimes U(c')\ar[r, "\alpha_{c'}"']&B(c')\otimes U(c')	
\end{tikzcd}
\]
so we may apply uniformity \eqref{eqn.uniformity}. Together with naturality \eqref{eqn.tr_naturality}, we have
\begin{align*}
  \Tr^{U(c)}_{A(c),B(c)}(\alpha_c)\then B(s)&\overset{\eqref{eqn.tr_naturality}}{=}
  \Tr^{U(c)}_{A(c),B(c')}(\alpha_c\then B(s)\otimes U(c))\\&\overset{\eqref{eqn.uniformity}}{=}  \Tr^{U(c')}_{A(c),B(c')}(A(s)\otimes U(c')\then\alpha_{c'})\\&\overset{\eqref{eqn.tr_naturality}}{=}
  A(s)\then\Tr^{U(c')}_{A(c'),B(c')}(\alpha_{c'})
\end{align*}
meaning that the required naturality square for $s$ and $\Tr^U_{A,B}(\alpha)$ commutes. 

The axioms of uniform traced categories (\cref{def.traced,def.strict_uniform}) are straightforward, because all definitions are pointwise.
\end{proof}

\begin{proposition}
For any uniform traced category $\cat{T}$, the assignment $\cat{C}\mapsto\cat{T}^\cat{C}$ extends to a functor $\cat{T}^-\colon\smcat\op\to\uniftrcat$.
\end{proposition}
\begin{proof}
We gave this functor on objects in \cref{prop.uniform_traced_exponent}. Given a functor $\cat{C}\to\cat{D}$, the induced functor $\cat{T}^\cat{D}\to\cat{T}^\cat{C}$ is strong monoidal and preserves trace because all these structures are defined pointwise. 
\end{proof}

In particular, applying this functor to the object-selecting functor $c\colon 1\to\cat{C}$ for any $c\in\ob\cat{C}$, we recover the evaluation $F\mapsto F(c)$ as a traced functor $\cat{T}^\cat{C}\to\cat{T}$.

\begin{proposition}\label[proposition]{prop.ff_uniform_traced}
If $\cat{T}$ is (uniform) traced and $F\colon\cat{T}'\to\cat{T}$ is a fully faithful strong monoidal functor, then $\cat{T}'$ inherits a (uniform) traced structure for which $F$ is a traced functor.
\end{proposition}
\begin{proof}
Given objects $A,B,U\in\cat{T}'$ and a map $g'\colon A\otimes U\to B\otimes U$ in $\cat{T}'$, we want its trace in $\cat{T}'$. So let $g\coloneqq F(g')$ be its image in $\cat{T}$ and take the trace $f\coloneqq\Tr^{F(U)}_{F(A),F(B)}(F(g'))\colon F(A)\to F(B)$. Because $F$ is fully faithful, there is a unique map $f'\colon A\to B$ in $\cat{T}'$ with $F(f')=f$, and we can take $f'$ to be the trace of $g'$. All axioms of traced categories hold automatically. 

Similarly, a map $s'\colon A\to B$ is strict in $\cat{T}'$ if its image $F(s')$ is strict in $\cat{T}$. In particular, if $\cat{T}$ is uniform then so is $\cat{T}'$.
\end{proof}

For each connected $r$, passing to maps out of $r$ is a functor; all that \cref{sec.trajectories} will use of a trace on $\cat{C}_\star$ is that these functors preserve it.

\begin{lemma}[Elements functor]\label[lemma]{lem.elements_functor}
Let $\cat{C}$ be lextensive and let $r$ be a connected object. Sending an object $c$ to the set $\cat{C}(r,c)$ of maps, and a partial map $g\colon c\parto d$ to the partial function $v\mapsto g(v)$, which is defined at $v$ iff $g$ is, constitutes a strong monoidal functor
\[
\defineTerm{Gamma}_r
\colon\cat{C}_\star\to\smsetstar.
\]
\end{lemma}
\begin{proof}
$\Gamma_r$ is well defined and functorial by \cref{lem.elementwise}\eqref{elt.dom}: $g(v)$ is total when $g$ is defined at $v$ and has $\dom(g(v))=0$ otherwise, so $h(g(v))$ is defined iff $g(v)$ is and $h$ is defined at it. It is strong monoidal by \cref{lem.elementwise}\eqref{elt.exactly_one}: the map $\Gamma_r(a)+\Gamma_r(b)\to\Gamma_r(a+b)$ is a bijection, and $\Gamma_r(0)=\emptyset$: a map $r\to0$ would factor through one of the empty coproduct's inclusions, of which there are none.
\end{proof}

\begin{definition}[Trace computed elementwise by iteration]\label[definition]{def.trace_iteration}
Let $\cat{C}$ be lextensive and suppose that the cocartesian category $\cat{C}_\star$ (\cref{lemma.lex_star}) is traced. The trace is
\defWord{computed elementwise by iteration} if $\Gamma_r\colon\cat{C}_\star\to\smsetstar$ is a traced functor for every connected object $r$.
\end{definition}

\section{\texorpdfstring{$\termraw{polystar}$}{Poly*} is uniform traced}\label[section]{sec.int_polystar}

Recall $\polystar$ from \cref{lemma.lex_star,prop.complemented_points}; it is cocartesian monoidal because it is the Kleisli category of a monad $p\mapsto p+1$ on $\poly$. In this section we establish that $\polystar$, and its multivariate version $\smset[\cat{L}]_\star$ for any $\cat{L}:\smcat$, is uniform traced.

\begin{lemma}\label[lemma]{lemma.ff_strong_polystar_setsetstar}
There is a fully faithful strong monoidal functor $\polystar\to\Cat{Fun}(\smset,\smsetstar)$. 

More generally, there is a fully faithful strong monoidal functor $\smset[\cat{L}]_\star\to\Cat{Fun}(\cat{L}\set,\smsetstar)$.
\end{lemma}
\begin{proof}
First note that for any $\cat{C}$ there is an isomorphism $1/\Cat{Fun}(\cat{C},\smset)\cong\Cat{Fun}(\cat{C},1/\smset)$.
%Indeed, this follows from the universal property associated to the fact that $1/\smset$ is the lax slice.
%\[
%\begin{tikzcd}
%	1/\smset\ar[r]\ar[d]&1\ar[d, "1"]\\
%	\smset\ar[r, equal]&\smset\ar[ul, phantom, "\Swarrow"]
%\end{tikzcd}
%\]
Suppose $F\colon\cat{C}\to\cat{D}$ preserves the terminal object. If it is fully faithful then so is the induced functor $1/F\colon 1/\cat{C}\to1/\cat{D}$; thus the functor
\[\smset[\cat{L}]_\star\simeq1/\smset[\cat{L}]\To{1/\Ext{\cat{L}}} 1/\Cat{Fun}(\cat{L}\set,\smset)\simeq\Cat{Fun}(\cat{L}\set,\smsetstar)\]
is fully faithful. 

It remains to show that $1/\smset[\cat{L}]\to1/\Cat{Fun}(\cat{L}\set,\smset)$ is strong monoidal, i.e.\ preserves coproducts, i.e.\ that $\smset[\cat{L}]\to\Cat{Fun}(\cat{L}\set,\smset)$ preserves pushouts of spans with the form $p\from 1\to q$. By \cref{prop.complemented_points}, the above span has the form $p'+1\from 1\to q'+1$, and its pushout in $\Cat{Fun}(\cat{L}\set,\smset)$ is again a polynomial, namely $(p'+q'+1)\in\smset[\cat{L}]$.
\end{proof}

\begin{corollary}\label[corollary]{cor.polystar_traced}
For any small category $\cat{L}:\smcat$, the cocartesian monoidal category $(\smset[\cat{L}]_\star,0,+)$ is uniform traced, with trace \word{computed elementwise by iteration};
in particular, so is $(\polystar,0,+)$.

Moreover, for each $X:\cat{L}\set$, the functor $\smset[\cat{L}]_\star\to\smsetstar$ sending $(1\to p)\mapsto(1\to \Ext{\cat{L}}(p)(X))$ is a traced functor.
\end{corollary}
\begin{proof}
Uniform tracedness, and the evaluation at each $X:\cat{L}\set$ being traced, follow from \cref{prop.setstar_traced,prop.setstar_uniform}, \cref{prop.uniform_traced_exponent,prop.ff_uniform_traced}, and \cref{lemma.ff_strong_polystar_setsetstar}.

For the trace being \word{computed elementwise by iteration}:
the connected objects of $\smset[\cat{L}]$ are the representables (\cref{prop.complemented_points}), and the Yoneda lemma gives $\Gamma_{\yon^A}(p)=\smset[\cat{L}](\yon^A,p)\cong\Ext{\cat{L}}(p)(A)$ naturally in $p$, so each $\Gamma_{\yon^A}$ is isomorphic to the evaluation at $A:\cat{L}\set$; being isomorphic to a traced functor, it is traced.
\end{proof}

All of the upcoming results about $\polystar$ extend to the multivariate $\smset[\cat{L}]_\star$-case, but we leave the latter case implicit for typographical simplicity.

\chapter{\texorpdfstring{$\termraw{Int}(\termraw{polystar})$}{Int(Poly*)} as syntax for control flow and data transformations}\label[section]{chap.int_poly_wds}

Wiring diagrams are an operadic approach to compositionality; in \cref{sec.wds}, we recall this perspective. In \cref{sec.int_polystar_wd_syntax} we discuss the wiring diagram syntax corresponding to the operad $\W$ underlying $\Int(\polystar)$. In \cref{sec.w_alg_cats} we explain how $\W$-algebras give rise to traced categories. Finally in \cref{sec.bypass}, we define a parameterized version of this operad to deal with programs for which some functions need to store---rather than acting on---certain variables.

\section{Wiring diagrams}\label[section]{sec.wds}

In this section we explain the formalism for wiring diagrams as in \cref{eqn.WD_poly}, reproduced in \cref{ex.wd_full}. In the operadic perspective on wiring diagrams \cite{Spivak:2013b,Rupel.Spivak:2013a,Vagner.Spivak.Lerman:2015a,Yau:2015a,patterson2021wiring,selby2025generalised}, a (colored) operad $\W$ has objects $A$ that act as \emph{boxes} and morphisms $\varphi\colon A_1,\ldots,A_N\to B$ that act as \emph{composition patterns} (often called \emph{wiring diagrams}) drawn as a finite number of ``small'' boxes $A_1,\ldots,A_N$ inside of a ``big'' box $B$. An operad functor $F\colon\W\to\smset$ assigns to each box $A:\ob\W$ the set $F(A):\ob\smset$ of all ``$F$-style fillers'' for that box, and assigns to each wiring diagram $\varphi$ the function $F(\varphi)\colon F(A_1)\times\cdots\times F(A_N)\to F(B)$ that takes a filler for each of the small boxes and puts them together under pattern $\varphi$ to make a filler for the big box.

The following example says that the composition patterns for traced categories are cobordisms; it is included for the reader's edification and is not essential.

\begin{example}\label[example]{ex.traced_cob_alg}
The composition patterns for traced categories $\cat{T}$ include the picture shown in \eqref{eqn.WD_traced}: the boxes are signed sets $(A^-,A^+)$, or equivalently oriented $0$-manifolds, and the wiring diagrams are oriented 1-cobordisms $(A_1^-,A_1^+)+\cdots+(A_N^-,A_N^+)\to(B^-,B^+)$. Here, an oriented $1$-cobordism from a signed set $S$ to a signed set $T$ is a compact oriented $1$-manifold up to diffeomorphism (this is a highfalutin way of saying: a disjoint union of line segments and circles) whose oriented boundary is identified with $\overline{S}+T$, where $\overline{S}$ reverses the orientation of $S$. Concretely, it amounts to a perfect matching of the ports of $\overline{S}$ and $T$ that pairs each port with a port of opposite orientation, together with finitely many free loops. If $\cat{T}$'s objects are generated under $\otimes$ by a set $L$, then all the signed sets and cobordisms should be sliced over $L$. Here is a depiction of such a map with $N=2$ and $L=1$:
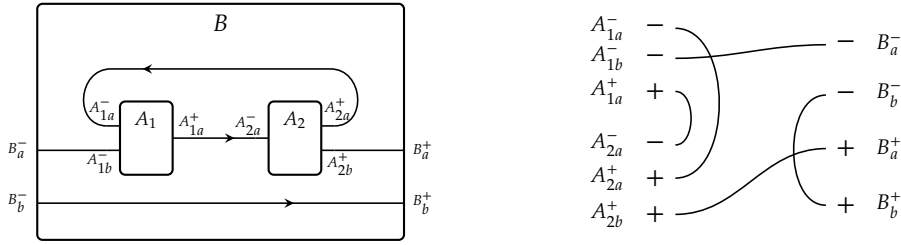
\[
\parbox{2.3in}{
\begin{tikzpicture}[oriented WD, bb min width =.7cm, bby=1.2ex, bbx=1.1cm,bb port length=5pt,
	label/.append style={node distance=1pt and -3pt}] 
  \node[bb={2}{1},, bb name={\tiny $A_1$}] (X1) {};
  \node[bb={1}{2}, right=1 of X1_out1, bb name={\tiny $A_2$}] (X2) {};
  \node[bb={0}{0}, fit={(X1) (X2) ($(X1.south)-(0,3)$) ($(X1.north)+(0,5)$)}, bb name={\footnotesize$B$}] (Y) {};
  \coordinate (Y_in1') at (Y.west|-X1_in2);
  \coordinate (Y_in2') at ($(Y_in1')-(0,20pt)$);
  \coordinate (Y_out1') at (Y.east|-X2_out2);
  \coordinate (Y_out2') at (Y.east|-Y_in2');
  \draw (Y_in1') to[in looseness=1.25] (X1_in2);
  \draw[ar,pos=.8] (X1_out1) to (X2_in1);
  \draw (X2_out2) to[out looseness=1.25] (Y_out1');
  \draw[ar] (Y_in2') -- (Y_out2');
  \draw[ar] let \p1=(X2.north east), \p2=(X1.north west), \n1={\y2+2*\bby}, \n2=\bbportlen in
   	(X2_out1) to[in=0, looseness=1.5] (\x1+.7*\n2,\n1) -- (\x2-.7*\n2,\n1) to[out=180, looseness=1.5] (X1_in1);
  \draw[label] 
        node[above left=2pt and -4pt of X1_in1] {\tiny$A^-_{1a}$}
        node[below left=of X1_in2] {\tiny$A^-_{1b}$}
        node[above right=of X1_out1] {\tiny$A^+_{1a}$}
        node[above left=of X2_in1] {\tiny$A^-_{2a}$}
        node[above right=2pt and -4pt of X2_out1] {\tiny$A^+_{2a}$}
        node[below right=of X2_out2] {\tiny${A}^+_{2b}$}
        node[left=3pt of Y_in1'] {\tiny$B^-_{a}$}
        node[left=3pt of Y_in2'] {\tiny$B^-_{b}$}
        node[right=3pt of Y_out1'] {\tiny$B^+_{a}$}
        node[right=3pt of Y_out2'] {\tiny$B^+_{b}$}
        ;
\end{tikzpicture}
}
\hspace{.6in}
\parbox{1.75in}{
\begin{tikzpicture}[x=1cm,y=1ex,node distance=1 and 1,semithick,every label quotes/.style={font=\everymath\expandafter{\the\everymath\scriptstyle}},every to/.style={out=0,in=180},baseline=(current bounding box.center)]
  \node ["$A^-_{1a}$" left] (X1a) {$-$};
  \node [below=0 of X1a, "$A^-_{1b}$" left] (X1b) {$-$};
  \node [below=0 of X1b, "$A^+_{1a}$" left] (X1c) {$+$};
  \node [below=1.5 of X1c, "$A^-_{2a}$" left] (X2a) {$-$};
  \node [below=0 of X2a, "$A^+_{2a}$" left] (X2b) {$+$};
  \node [below=0 of X2b, "$A^+_{2b}$" left] (X2c) {$+$};
  \node [below right=-1 and 2 of X1a, "$B^-_a$" right] (Ya) {$-$};
  \node [below=1.5 of Ya, "$B^-_b$" right] (Yb) {$-$};
  \node [below=1.5 of Yb, "$B^+_a$" right] (Yc) {$+$};
  \node [below=1.5 of Yc, "$B^+_b$" right] (Yd) {$+$};
  \draw (X1a) to[in=0] (X2b);
  \draw (X1b) to (Ya);
  \draw (X1c) to[in=0] (X2a);
  \draw (X2c) to (Yc);
  \draw (Yb) to[out=180] (Yd);
\end{tikzpicture}
}
\]
For more general $L$, each box-port (e.g.\ $A^-_{1a}$, $A^+_{2b}$, $B^+_a$, etc.) would be labeled with an element of $L$ and the wires would preserve these labels. We denote the operad of oriented $L$-labeled 1-cobordisms by $\defineTerm{cob}/L$.

The data of a traced category $\cat{T}$ with objects freely generated by $L$ is equivalent to the data of a functor $T\colon\cob/L\to\smset$. For any two objects, i.e.\ lists of generating objects, $A^-=(\ell^-_1,\ldots,\ell^-_M)$ and $A^+=(\ell_1^+,\ldots,\ell^+_{N})$, the set of morphisms $\Hom_\cat{T}(A^-,A^+)$ corresponds with the set $T(A^-,A^+)$. Any composition pattern---any combination of compositions, tensors, and traces---applied to morphisms is determined by $T$ on the corresponding morphism in $\cob$. See \cite[Theorem 0]{Spivak.Schultz.Rupel:2016a} for a proof of this.
\end{example}

Something like the above idea will show up in \cref{sec.w_alg_cats}.

%\begin{remark}
%Let $\finpoly_\star\ss\polystar$ denote the full subcategory spanned by the finite polynomials $p$, i.e.\ those for which $p(X)$ is finite whenever $X:\finset$ is. 
%Since $(\finpoly_\star,0,+)$ is traced and its objects are generated under $+$ by $\ob(\finset)$, we said in \cref{ex.traced_cob_alg} that it can equivalently be specified by a functor $P\colon\cob/\ob(\finset)\to\smset$ and has an associated wiring diagram syntax; for example, two maps $A\colon I\to\ob(\finset)$ and $B\colon J\to\ob(\finset)$ together constitute a single object in $\cob/\ob(\finset)$ and $P$ assigns to it the hom-set
%\[
%	P((I,A),(B,J))\coloneqq\polystar\Big(\sum_{i:I}\yon^{A(i)},\sum_{j:J}\yon^{B(j)}\Big).
%\]
%Although the wiring diagram syntax given by $P$ would be useful for thinking about maps of polynomials behave under various sorts of composites, it is not the one to use for thinking about control and data flow. The latter instead corresponds to a different sort wiring diagram operad, namely $\Int(\finpoly_\star)$. Explaining the meaning of operad functors $\Int(\finpoly_\star)\to\smset$ is the purpose of the remainder of this section.
%\end{remark}

\section{\texorpdfstring{$\termraw{W}\coloneqq\termraw{Int}(\termraw{polystar})$}{W = Int(Poly*)} as wiring diagram syntax}\label[section]{sec.int_polystar_wd_syntax}

By \cref{cor.polystar_traced} we know that $(\polystar,0,+)$, and the multivariate version $\smset[\cat{L}]_\star$ for any category $\cat{L}$, is a traced category. Hence, we can use the definition of $\Int$ in \cref{sec.int} to form the compact category $(\Int(\polystar),0,+)$. An object in it is a pair $P=(p^-,p^+)$ of polynomials, its monoidal product is given by pointwise sum $P_1+P_2\coloneqq(p_1^-+p_2^-,p_1^++p_2^+)$, and its morphisms are given by $(+1)$-Kleisli maps of polynomials
\[
  \phi\colon P\tickar Q
  \qqmeans
  \phi\colon q^-+p^+\to q^++p^-+1.
\]
\begin{warning}
Note that $\Int(\polystar)$ is not cocartesian, even though $\polystar$ is. Thus the inherited notation $(0,+)$ for the monoidal structure on $\Int(\polystar)$ could be misleading.
\end{warning}

Underlying the monoidal category $\Int(\polystar)$ is an operad $\defineTerm{W}$ with objects and morphisms given by
\begin{equation}\label{eqn.operad_W}
\ob(\W)\coloneqq\ob(\Int(\polystar))
\qqand
\W(P_1,\cdots,P_N; Q)\coloneqq\Int(\polystar)(P_1+\cdots+P_N,Q)
\end{equation}
Operad functors $\W\to\smset$ are in bijection with lax monoidal functors $(\Int(\polystar),0,+)\to(\smset,1,\times)$. We may switch freely between $\W$ and $\Int(\polystar)$.

\begin{remark}[The formal status of a wiring diagram]\label[remark]{rem.wd_formal}
The wiring diagrams in this paper are notation for morphisms: a wiring diagram with inner boxes $P_1,\dots,P_N$ and outer box $Q$ is by definition a morphism $\varphi\in\W(P_1,\dots,P_N;Q)$---equivalently, a map $P_1+\cdots+P_N\tickar Q$ in $\Int(\polystar)$, as in \eqref{eqn.operad_W}---and nesting a diagram inside a box of another is operadic composition in $\W$. The pictures are the string-diagram language of the compact closed category $\Int(\polystar)$, drawn with each wire expanded so that its polynomial type is visible: the blue regions are the positions of the polynomial (its case structure, i.e.\ control) and the thin wires are its directions (the data). Two diagrams are equal precisely when the polynomial maps they denote agree; by the coherence theorems for such graphical languages \cite{joyal1991geometry,Joyal.Street.Verity:1996a}, the structure-preserving deformations---bending and sliding wires, and the feedback and yanking moves of the traced and compact closed structure---are a sound way to establish such equalities.

One could instead try to present this graphical language by generators and relations, in the style of presented PROPs \cite{lack2004composing}---a fixed alphabet of elementary combinators (copying, merging, caps and cups, and so on) together with the equations between them. However, there's no need in our case because equality of diagrams is just a matter of comparing polynomial maps. More fundamentally, the operadic viewpoint lets a whole wiring diagram act as a single operation on all of its boxes at once, whereas a generators-and-relations presentation forces one to decompose each diagram into a word in elementary combinators; we believe that this is neither how one builds, nor how one wants to reason about, wiring diagrams in practice.

What the syntax buys us, beyond the formula \eqref{eqn.poly_map} for maps of polynomials, is legibility and modularity: a diagram displays at a glance how control and data flow through a system and how subsystems nest, and the algebraic identities of $\Int(\polystar)$ become local moves on pictures.
\end{remark}

\begin{example}\label[example]{ex.wd_full}
In the diagram below
\begin{equation}\label{eqn.wd_full}
\Phi\coloneqq\begin{tikzpicture}[WD, wdunit=7pt, halfunit=2pt, baseline=(BoxA.north)]
  \begin{pgfonlayer}{background}
    % only one region per side on Outer2 now
    \wdboxUnits[minimum width=6cm]{Outer2}{}{1,1}{0,3}{8.5}
  \end{pgfonlayer}
  \node[below=0pt] at (Outer2.north) {$Q$};

  % Box A (half size, left)
  \begin{scope}[shift={(Outer2.center)},xshift=-1.5cm,
                yshift=-.45cm, transform shape,scale=.3]
    \wdboxUnits[minimum width=4cm]{BoxA}{\huge $P_1$}{2}{1,3}{0}
  \end{scope}

  % Box B (half size, right, lifted)
  \begin{scope}[shift={(Outer2.center)},xshift=1.5cm,
                yshift=.45cm,transform shape,scale=.3]
    \wdboxUnits[minimum width=4cm]{BoxB}{\huge $P_2$}{1,1}{1,2}{0}
  \end{scope}

	% Tubes

  \tubeN{Outer2}{L}{2}{BoxA}{L}{1}{0}
  \draw[data wire] (Outer2-L-2-D1)
    to[out=0,in=180] (BoxA-L-1-D1);
  \draw[data wire] (Outer2-L-2-D1)
    to[out=0,in=180] (BoxA-L-1-D2);

	\tubeN{Outer2}{L}{1}{BoxB}{L}{1}{1}

  \tubeN{BoxA}{R}{2}{Outer2}{R}{2}{2}
  \draw[data wire] (BoxA-R-2-D2)
    to[out=0,in=180] (Outer2-R-2-D3);

  \tubeN{BoxA}{R}{1}{BoxB}{L}{2}{1}

	\tubeN{BoxB}{R}{1}{Outer2}{R}{1}{0}

  \loopTubeSmooth[below]{BoxB}{R}{2}{BoxA}{L}{1}{2}{BoxA}
\end{tikzpicture}
\end{equation}
we see the following objects in $\W=\Int(\polystar)$:
\[
	P_1\coloneqq(\yon^2,\yon+\yon^3),\qquad
	P_2\coloneqq(\yon+\yon,\yon+\yon^2),\qqand
	Q\coloneqq(\yon+\yon,1+\yon^3).
\]
The diagram represents a morphism $\Phi\colon P_1+P_2\tickar Q$; see also the following map of polynomials, from $Q^-+P_1^++P_2^+$ to $Q^++P_1^-+P_2^-$:
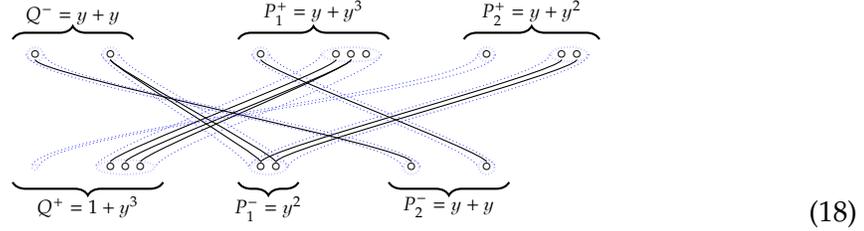
\begin{equation}\label{eqn.poly_depiction_of_wd}
\makeatletter
\tikzset{
  % define a key “bot y” whose value gets stored in \boty
  bot y/.store in=\boty,
  bot y=-1.5   % default
}
\makeatother\begin{tikzpicture}[decoration={brace, amplitude=5pt}, font=\tiny, inner sep=0]
    \node at (1,0) (topQ-1-1) {$\circ$};
    \node at (2,0) (topQ-2-1) {$\circ$};
    \node at (4,0) (topP1-1-1) {$\circ$};
    \node at (5,0) (topP1-2-1) {$\circ$};
    \node at (5.2,0) (topP1-2-2) {$\circ$};
    \node at (5.4,0) (topP1-2-3) {$\circ$};
    \node at (7,0) (topP2-1-1) {$\circ$};
    \node at (8,0) (topP2-2-1) {$\circ$};
    \node at (8.2,0) (topP2-2-2) {$\circ$};
		\node at (1,\boty) (botQ-1) {};
    \node at (2,\boty) (botQ-2-1) {$\circ$};
    \node at (2.2,\boty) (botQ-2-2) {$\circ$};
    \node at (2.4,\boty) (botQ-2-3) {$\circ$};
    \node at (4, \boty) (botP1-1-1) {$\circ$};
    \node at (4.2, \boty) (botP1-1-2) {$\circ$};
    \node at (6, \boty) (botP2-1-1) {$\circ$};
    \node at (7, \boty) (botP2-2-1) {$\circ$};
    \begin{scope}[every node/.style={draw, thick, blue!70, ellipse, dash pattern=on 0pt off 1pt}]
    	\node[fit=(topQ-1-1)] (topQ-1) {};
    	\node[fit=(topQ-2-1)] (topQ-2) {};
    	\node[fit=(topP1-1-1)] (topP1-1) {};
    	\node[fit=(topP1-2-1)(topP1-2-3)] (topP1-2) {};
    	\node[fit=(topP2-1-1)] (topP2-1) {};
    	\node[fit=(topP2-2-1)(topP2-2-2)] (topP2-2) {};
    	\node[fit=(botQ-1)] (botQ-1) {};
    	\node[fit=(botQ-2-1)(botQ-2-3)] (botQ-2) {};
    	\node[fit=(botP1-1-1)(botP1-1-2)] (botP1-1) {};
    	\node[fit=(botP2-1-1)] (botP2-1) {};
    	\node[fit=(botP2-2-1)] (botP2-2) {};
    \end{scope}
    \draw[decorate, thick] ($(topQ-1-1)+(-.3,.2)$) -- node[above=5pt] {$Q^{-}=\yon+\yon$} ($(topQ-2-1)+(.3,.2)$);
    \draw[decorate, thick] ($(topP1-1-1)+(-.3,.2)$) -- node[above=5pt] {$P_1^{+}=\yon+\yon^3$} ($(topP1-2-3)+(.3,.2)$);
    \draw[decorate, thick] ($(topP2-1-1)+(-.3,.2)$) -- node[above=5pt] {$P_2^{+}=\yon+\yon^2$} ($(topP2-2-2)+(.3,.2)$);
    \draw[decorate, thick] ($(botQ-2-3)+(.3,-.2)$) -- node[below=5pt] {$Q^+=1+\yon^3$} ($(botQ-1)+(-.3,-.2)$);
    \draw[decorate, thick] ($(botP1-1-2)+(.3,-.2)$) -- node[below=5pt] {$P_1^-=\yon^2$} ($(botP1-1-1)+(-.3,-.2)$);
    \draw[decorate, thick] ($(botP2-2-1)+(.3,-.2)$) -- node[below=5pt] {$P_2^{-}=\yon+\yon$} ($(botP2-1-1)+(-.3,-.2)$);
   \begin{scope}[every to/.style={out=90, in=270, looseness=.1}] 
%group 1
		\draw (botQ-2-1) to (topP1-2-1);
		\draw (botQ-2-2) to (topP1-2-2);
		\draw (botQ-2-3) to (topP1-2-2);
%group 2
    \draw (botP1-1-1) to (topQ-2-1);
    \draw (botP1-1-2) to (topQ-2-1);
%group 3
    \draw (botP1-1-1) to (topP2-2-1);
    \draw (botP1-1-2) to (topP2-2-2);
%group 4
    \draw (botP2-1-1) to (topQ-1-1);
%group 5
    \draw (botP2-2-1) to (topP1-1-1);
  \end{scope}
  \begin{scope}[
    draw=blue!70,
    densely dotted,
    every to/.style={out=270, in=90, looseness=0.2}
  ]
 \foreach \i/\j in {
		topQ-1/botP2-1,%
		topQ-2/botP1-1,%
		topP1-1/botP2-2,%
		topP1-2/botQ-2,%
		topP2-1/botQ-1,%
		topP2-2/botP1-1}{
  \draw (\i.west) to (\j.west);
  \draw (\i.east) to (\j.east);
}
  \end{scope}
\end{tikzpicture}
\end{equation}
For example, from the upper right to the lower left part of \eqref{eqn.poly_depiction_of_wd} we see the visualization of a map from the $\yon^1$ summand of $P_2^+$ to the $\yon^0$ summand of $Q^+$. Namely, it is given by the unique element of the hom-set $\poly(\yon^1,\yon^0)\cong\smset(0,1)$.

In terms of control flow and data flow, each box has some number of blue regions which we call \emph{control regions}. We refer to a control region on the box's left as an  \emph{entrance} and to one on the right as an \emph{exit}. Each control region has some number of wires, which we call \emph{data slots}. The control is \emph{passed} forward from one box to another, given by $\Phi$ on polynomial summands, and the data slots in the latter \emph{receive} data from the former, arising from the Yoneda lemma again, $\poly(\yon^A,\yon^B)\cong\smset(B,A)$.

So in \eqref{eqn.wd_full} we would say that box $Q$ has two entrances, two exits, and four control regions. The two data slots in the entrance to $P_1$ both receive data from the single data slot in the second entrance of $Q$. And so on.
\end{example}

\begin{remark}[Database reading]\label[remark]{rem.duc_query}
For general $\cat{L}$, our wiring diagram syntax, e.g.\ as in \cref{ex.wd_full}, has a database reading \cite{Spivak:2014c}. Namely $\cat{L}$ is a database schema, an $\cat{L}$-set $X$ is an instance, an $\cat{L}$-set $A$ is a (possibly infinitary) conjunctive query---variables
of various types and constraints between them---and $\cat{L}\set(A,X)$ is the set of answers to $A$ in $X$. 

In this reading, the left (resp. right) side of each box is a disjoint union of conjunctive queries, saying what sort of thing can enter and leave the box. Each control region is one conjunctive query, its data slots are the variables of that query, possibly constrained (though constraints are not shown in the pictures), and $\Ext{\cat{L}}(p)(X)$ is the set of answers to the disjoint union $p$ of the box's queries. A data transformation $\yon^A\to\yon^B$ is a map $B\to A$ of $\cat{L}$-sets, i.e.\ a homomorphism of queries \cite{chandra1977optimal}: it exhibits every answer to $A$ as an answer to $B$. 

In the algebra $\sem{-}_X$ discussed in the next section, \cref{sec.w_alg_cats}, a filler of $(p^-,p^+)$ is any partial function from the answers to $p^-$ to the answers to $p^+$; unlike the wiring diagrams, it need not be natural in $X$, and so it may test, branch, or perform any partial function whatsoever.
\end{remark}

\begin{remark}
In \cref{ex.wd_full}, even if $Q^+$ had not included the $\yon^0$ summand, we could have drawn a very similar wiring diagram, 
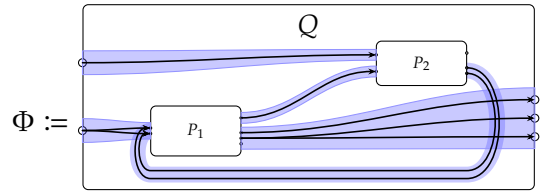
\[
\Phi\coloneqq
\begin{tikzpicture}[WD, wdunit=7pt, halfunit=2pt, baseline=(BoxA)]
  \begin{pgfonlayer}{background}
    % only one region per side on Outer2 now
    \wdboxUnits[minimum width=6cm]{Outer2}{}{1,1}{3}{8.5}
  \end{pgfonlayer}
  \node[below=0pt] at (Outer2.north) {$Q$};

  % Box A (half size, left)
  \begin{scope}[shift={(Outer2.center)},xshift=-1.5cm,
                yshift=-.45cm, transform shape,scale=.3]
    \wdboxUnits[minimum width=4cm]{BoxA}{\huge $P_1$}{2}{1,3}{0}
  \end{scope}

  % Box B (half size, right, lifted)
  \begin{scope}[shift={(Outer2.center)},xshift=1.5cm,
                yshift=.45cm,transform shape,scale=.3]
    \wdboxUnits[minimum width=4cm]{BoxB}{\huge $P_2$}{1,1}{1,2}{0}
  \end{scope}

	% Tubes

  \tubeN{Outer2}{L}{2}{BoxA}{L}{1}{0}
  \draw[data wire] (Outer2-L-2-D1)
    to[out=0,in=180] (BoxA-L-1-D1);
  \draw[data wire] (Outer2-L-2-D1)
    to[out=0,in=180] (BoxA-L-1-D2);

	\tubeN{Outer2}{L}{1}{BoxB}{L}{1}{1}

  \tubeN{BoxA}{R}{2}{Outer2}{R}{1}{2}
  \draw[data wire] (BoxA-R-2-D2)
    to[out=0,in=180] (Outer2-R-1-D3);

  \tubeN{BoxA}{R}{1}{BoxB}{L}{2}{1}

  \loopTubeSmooth[below]{BoxB}{R}{2}{BoxA}{L}{1}{2}{BoxA}
\end{tikzpicture}
\]
using the fact that we're working in $\polystar$. That is, in the Kleisli map $Q^-+P_1^++P_2^+\to Q^++P_1^-+P_2^-+1$ we could have sent the $\yon$ component of $P_2^+$ to the (implicit) $1$.
\end{remark}

Composition in $\W$ corresponds to nesting of wiring diagrams. That is, for any $M,N:\nn$, objects $P_1,\ldots,P_M$ and $Q_1,\ldots,Q_N$ and element $1\leq n\leq N$, the function
\begin{multline}\label{eqn.circle_i_comp}
	\circ_n\colon
	\W(Q_1,\ldots, Q_N;R)
	\times
	\W(P_{1},\ldots,P_{M};Q_n)
	\\\to
	\W(Q_1,\ldots,Q_{n-1},P_1,\ldots,P_M,Q_{n+1},\ldots,Q_N;R)
\end{multline}
takes a wiring diagram of little boxes $P_1,\ldots,P_M$ inside mid-sized box $Q_n$, as well as a wiring diagram of mid-sized boxes $Q_1,\ldots,Q_N$ inside of large box $R$, and nests them to see the $Q$'s with $P_1,\ldots,P_M$ replacing $Q_n$, inside of large box $R$.
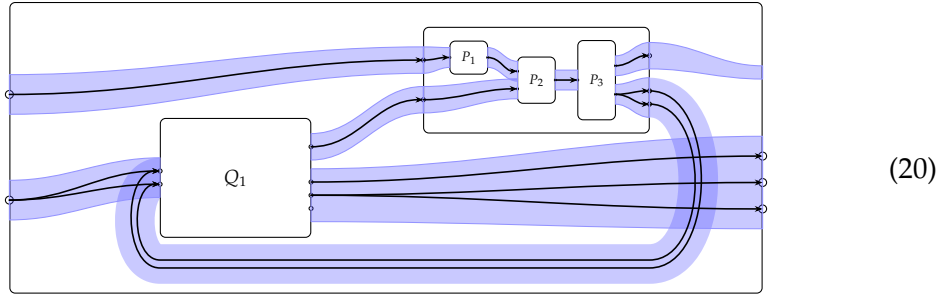
\begin{equation}\label{eqn.nesting_pic}
\begin{tikzpicture}[WD, baseline=(BoxA)]
  \begin{pgfonlayer}{background}
    % only one region per side on Outer2 now
    \wdboxUnits[minimum width=10cm]{Outer2}{}{1,1}{0,3}{9}
  \end{pgfonlayer}

  % Box A (half size, left)
  \begin{scope}[shift={(Outer2.center)},xshift=-\innerOffset,
                yshift=-.4cm, transform shape,scale=\innerScale]
    \wdboxUnits[minimum width=4cm]{BoxA}{\Large $Q_1$}{2}{1,3}{0}
  \end{scope}

  % Box B (half size, right, lifted)
  \begin{scope}[shift={(Outer2.center)},xshift=\innerOffset,
                yshift=.9cm,transform shape,scale=\innerScale]
    \wdboxUnits[minimum width=6cm]{BoxB}{}{1,1}{1,2}{0}
  \end{scope}

	% Big Tubes

  \tubeN{Outer2}{L}{2}{BoxA}{L}{1}{0}
  \draw[data wire] (Outer2-L-2-D1)
    to[out=0,in=180] (BoxA-L-1-D1);
  \draw[data wire] (Outer2-L-2-D1)
    to[out=0,in=180] (BoxA-L-1-D2);

	\tubeN{Outer2}{L}{1}{BoxB}{L}{1}{1}

  \tubeN{BoxA}{R}{2}{Outer2}{R}{2}{2}
  \draw[data wire] (BoxA-R-2-D2)
    to[out=0,in=180] (Outer2-R-2-D3);

  \tubeN{BoxA}{R}{1}{BoxB}{L}{2}{1}

	\tubeN{BoxB}{R}{1}{Outer2}{R}{1}{0}

  \loopTubeSmooth[below]{BoxB}{R}{2}{BoxA}{L}{1}{2}{BoxA}

% --- tiny subdiagram inside BoxA: one inner box with no ports ---
\begin{scope}[shift={(BoxB.center)}, xshift=-.9cm, yshift=3mm, transform shape, scale=0.25]
  \wdboxUnits[minimum width=2cm]{Binner1}{\huge $P_1$}{1}{1}{0}
\end{scope}
\begin{scope}[shift={(BoxB.center)}, xshift=0cm, yshift=0mm, transform shape, scale=0.25]
  \wdboxUnits[minimum width=2cm]{Binner2}{\huge $P_2$}{1,1}{1}{0}
\end{scope}
\begin{scope}[shift={(BoxB.center)}, xshift=.8cm, yshift=0mm, transform shape, scale=0.25]
  \wdboxUnits[minimum width=2cm]{Binner3}{\huge $P_3$}{1}{1,1}{10}
\end{scope}

% Little Tubes

\begin{scope}
	\colorlet{wdCtrl}{wdCtrlInner}
	\tubeN{BoxB}{L}{1}{Binner1}{L}{1}{1}
	\tubeN{BoxB}{L}{2}{Binner2}{L}{2}{1}
	\tubeN{Binner1}{R}{1}{Binner2}{L}{1}{1}
	\tubeN{Binner2}{R}{1}{Binner3}{L}{1}{1}
	\tubeN{Binner3}{R}{1}{BoxB}{R}{1}{1}
	\tubeN{Binner3}{R}{2}{BoxB}{R}{2}{1}
	\draw[data wire] (Binner3-R-2-D1)
    to[out=0,in=180] (BoxB-R-2-D2);
\end{scope}

\end{tikzpicture}
\end{equation}
While the composition formula has already been completely specified abstractly in \eqref{eqn.operad_W}, we spell it out in case doing so is of use.

Suppose given maps of polynomials $\psi\colon r^-+q_1^++\cdots+q_N^+\parto r^++q_1^-+\cdots+q_N^-$ and $\varphi\colon q_n^-+p_1^++\cdots+p_M^+\parto q_n^++p_1^-+\cdots+p_M^-$. Writing $q_{\neq n}\coloneqq q_1+\cdots+q_{n-1}+q_{n+1}+\cdots+q_N$ and $p\coloneqq p_1+\cdots+p_M$, the composite map from \eqref{eqn.circle_i_comp} is given by the following formula
\begin{multline}\label{eqn.composition_circle_i}
  \psi\circ_n\varphi\coloneqq\\
  \Tr^{q_n^+}%_{r^-+q_{\neq n}^++p^+,r^++q_{\neq n}^-+p^-}
  \big(r^-+q_{\neq n}^++q_n^++p^+\parto r^++q_{\neq n}^-+q_n^-+p^+\parto r^++q_{\neq n}^-+q_n^++p^-\big).
\end{multline}
The trace was defined for (multivariate) polynomials in \cref{cor.polystar_traced}.

\section{\texorpdfstring{$\termraw{W}$}{W}-algebras as structured categories}\label[section]{sec.w_alg_cats}

In this section, we explain how $\W$-algebras define structured categories. We first explain how any $C\colon\W\to\smset$ defines a category $\cat{C}$ with objects $\ob(\cat{C})\coloneqq\ob(\poly)$ and hom-sets
\[
\cat{C}(a,b)\coloneqq C(a,b),
\]
where $(a,b)\in\ob(\poly)\times\ob(\poly)\cong\ob(\W)$.

To obtain the identity on $a$, we first consider the map $\iota_a\colon (0,0)\tickar(a,a)$ in $\Int(\polystar)$ given by $\id\colon a+0\to 0+a$. It determines a $0$-ary morphism, i.e.\ an element $\iota_a:\W(;(a,a))$, which is in turn sent by $C$ to a function $C(\iota_a)\colon 1\to C(a,a)$, and we choose it to stand as the identity element in $\cat{C}(a,a)$.

To obtain the composition formula, consider the map $\kappa_{a,b,c}\colon(a,b)+(b,c)\tickar (a,c)$ in $\Int(\polystar)$ given by the isomorphism $a+b+c\to c+a+b$. It determines an element $\kappa_{a,b,c}\colon\W((a,b),(b,c);(a,c))$, which is in turn sent to a function $C(\kappa_{a,b,c})\colon C(a,b)\times C(b,c)\to C(a,c)$, and we choose it to stand as the composition formula.

\begin{example}
Here are the wiring diagrams corresponding to the identity $C(\iota_a)\colon 1\to\cat{C}(a,a)$ and the composition formula $\cat{C}(a,b)\times\cat{C}(b,c)\to\cat{C}(a,c)$, for the arbitrary case of $a\coloneqq\yon+\yon^2$, $b\coloneqq\yon+1$ and $c\coloneqq\yon^3$:
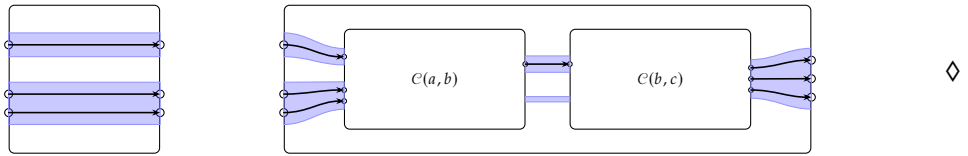
\[
\begin{tikzpicture}[WD, wdunit=7pt, halfunit=2pt, baseline=(Outer2)]
  \begin{pgfonlayer}{background}
    % only one region per side on Outer2 now
    \wdboxUnits[minimum width=2cm]{Outer2}{}{1,2}{1,2}{2}
  \end{pgfonlayer}
  \node[below=0pt] at (Outer2.north) {};
	% Tubes

  \tubeN{Outer2}{L}{1}{Outer2}{R}{1}{1}
  \tubeN{Outer2}{L}{2}{Outer2}{R}{2}{2}
\end{tikzpicture}
\hspace{.6in}
\begin{tikzpicture}[WD, wdunit=7pt, halfunit=2pt, baseline=(Outer2)]
  \begin{pgfonlayer}{background}
    % only one region per side on Outer2 now
    \wdboxUnits[minimum width=7cm]{Outer2}{}{1,2}{3}{2}
  \end{pgfonlayer}
  \node[below=0pt] at (Outer2.north) {};

  % Box A (half size, left)
  \begin{scope}[shift={(Outer2.center)},xshift=-1.5cm,
                transform shape,scale=.6]
    \wdboxUnits[minimum width=4cm]{BoxA}{$\cat{C}(a,b)$}{1,2}{1,0}{7}
  \end{scope}

  % Box B (half size, right, lifted)
  \begin{scope}[shift={(Outer2.center)},xshift=1.5cm,
                transform shape,scale=.6]
    \wdboxUnits[minimum width=4cm]{BoxB}{$\cat{C}(b,c)$}{1,0}{3}{7}
  \end{scope}

	% Tubes
  \tubeN{Outer2}{L}{1}{BoxA}{L}{1}{1}
  \tubeN{Outer2}{L}{2}{BoxA}{L}{2}{2}
  \tubeN{BoxA}{R}{1}{BoxB}{L}{1}{1}
  \tubeN{BoxA}{R}{2}{BoxB}{L}{2}{0}
  \tubeN{BoxB}{R}{1}{Outer2}{R}{1}{3}
\end{tikzpicture}
\qedhere
\]
\end{example}

We now recall a theorem saying that lax monoidal functors $\Int(\cat{T})\to\smset$ can be identified with bijective-on-objects traced functors $\cat{T}\to\cat{T}'$ to other traced categories $\cat{T}'$.

Let $\int^{\cat{T}\in\Cat{TrCat}}\Int(\cat{T})\alg$ denote the Grothendieck construction of the functor $\Cat{TrCat}\to\smcat$ sending $\cat{T}$ to the category of lax monoidal functors $\Int(\cat{T})\to\smset$, i.e.\ $\Int(\cat{T})$-algebras. And let $\Cat{TrCat}^\tn{bo}$ denote the category of traced categories and bijective-on-objects traced functors between them.

\begin{theorem}[{\cite[Theorem B]{Spivak.Schultz.Rupel:2016a}}]\label[theorem]{thm.schultzB}
There is an equivalence of categories
\[
\int^{\cat{T}\in\Cat{TrCat}}\Int(\cat{T})\alg\cong\Cat{TrCat}^\tn{bo}.
\]
\end{theorem}

Recall that the functor $\Cat{MonCatLax}\to\Cat{Operad}$ is fully faithful, so lax monoidal functors out of $\Int(\polystar)$ can be identified with $\W$-algebras.

\begin{corollary}\label[corollary]{cor.schultz_trcat}
The construction from the top of \cref{sec.w_alg_cats} extends to a functor $\W\alg\to\Cat{TrCat}$.
\end{corollary}
\begin{proof}
From \cref{thm.schultzB} we have an equivalence of categories $\W\alg\cong\Cat{TrCat}^\tn{bo}_{\polystar/}$ and we can simply forget the coslice data. One can check from the proof that the construction agrees with the $C\mapsto\cat{C}$ construction from the top of this section.
%The objects, morphisms, identities, and composition formula was defined above. To see that we have defined a category, it remains to prove unitality and associativity. These amount to checking three equations in $\W$. For example, unitality amounts to showing that we have
%\[\kappa_{a,b,b}\circ_2\iota_b=^?\id_{(a,b)}\]
%in $\W$, where $\kappa_{a,b,b}\colon(a,b)+(b,b)\tickar(a,b)$ and $\iota_b:(0,0)\tickar(b,b)$ are as above; this equation can be verified using \eqref{eqn.composition_circle_i}.
%
%We now proceed to define the monoidal and traced structures. For any $a,b,a',b'$, consider the map $\pi\colon(a,b),(a',b')\tickar(a+a',b+b')$ in $\W$ given by the isomorphism
%$\pi\colon a+a'+b+b'\to b+b'+a+a'$. 
%The monoidal structure is given on objects by $(0,+)$, and it's given on morphisms by $C(\pi)\colon C(a,b)\times C(a',b')\to C(a+a',b+b')$. Finally, for any $a,b,u$, consider the map $\tau\colon(a+u,b+u)\tickar(a,b)$ in $\W$ given by the isomorphism $a+b+u\to b+a+u$. The trace structure is given by $C(\tau)\colon C(a+u,b+u)\to C(a,b)$. We leave to the remaining checks to the reader.
\end{proof}
%
%\begin{remark}[Promonadic restatement]\label[remark]{rem.promonads}
%This remark is an aside for readers already fluent in profunctors and promonads; it recasts the preceding results in that language and is used nowhere else in the paper. Recall that promonads on $\cat{T}$ (that is, monoids in the category of profunctors from $\cat{T}$ to itself with the monoidal structure given by composition) can be identified with identity-on-objects functors out of $\cat{T}$.
%
%The symmetric monoidal structure on $\cat{T}$ induces symmetric monoidal structure on the category of promonads on $\cat{T}$, explicitly given by the coend
%\[
%  (\cat{P}_1 \boxtimes \cat{P}_2)(a,b) \coloneqq 
%  \int^{a_1,a_2,b_1,b_2 \in \cat{T}} 
%  \cat{T}(a,a_1 \otimes a_2) \times 
%  \cat{T}(b_1 \otimes b_2, b) \times
%  \cat{P}_1(a_1,b_1) \times 
%  \cat{P}_2(a_2,b_2)
%\]
%for two promonads $\cat{P}_1, \cat{P}_2$ on $\cat{T}$. Commutative monoids in promonads with respect to this monoidal structure can be identified with symmetric monoidal identity-on-objects functors out of $\cat{T}$. 
%
%Finally, one could define a \emph{complete Elgot promonad} to be a commutative promonad such that the corresponding identity-on-objects functor is traced (\cite{goncharov2016complete}). \cref{thm.schultzB} then identifies $\Int(\cat{T})$-algebras in sets with complete Elgot promonads on $\cat{T}$.
%\end{remark}

Recall from \cite[Definition 3.1]{fong2019supplying} that for $\cat{C}$ to have a \emph{supply} of commutative monoids means that its objects are coherently equipped with commutative monoid structures. 

\begin{proposition}\label[proposition]{prop.supply_monoids}
Suppose that $\cat{T}$ is cocartesian traced and that $C\colon\Int(\cat{T})\to\smset$ is lax monoidal. Then the traced category $\cat{C}$ given by \cref{thm.schultzB} has a supply of commutative monoids.
\end{proposition}
\begin{proof}
Every object in a cocartesian category is a $(0,+)$-monoid; in fact they form a supply of commutative monoids, and $C$ preserves that supply.
\end{proof}

\begin{remark}
The category $\cat{C}$ need not itself be cocartesian. Indeed, the bijective-on-objects traced functor $\cat{T}\to\cat{C}$ underlying \cref{cor.schultz_trcat} is strong monoidal, so it preserves the canonical commutative monoids of $\cat{T}$; but it need not preserve the universal property that makes $+$ a coproduct, since $\cat{C}$ generally has more morphisms than $\cat{T}$. What survives is exactly the \emph{supply}: each object $a$ retains a multiplication $\nabla_a\colon a+a\to a$ and a unit $0\to a$, satisfying the commutative-monoid laws but no longer pinned down by a universal property. 

In the wiring-diagram reading these are control-flow operations on the blue regions: $\nabla_a$ \emph{merges} two branches carrying the interface $a$ into one. The data on the wires is carried through such a merge, unchanged.
\end{remark}

We conclude this section by offering an example $\W$-algebra, the \emph{partial-function algebra} $\sem{-}_X$ at a set $X:\smset$. Recall from \cref{cor.polystar_traced} that for any set $X:\smset$ there is a traced functor $-_X\colon\polystar\to\smsetstar$ given by sending $p\mapsto p(X)$. Since $\Int$ is functorial, we obtain a strong monoidal functor $\Int(-_X)\colon\Int(\polystar)\to\Int(\smsetstar)$. Now, for any monoidal category $(\cat{C},I,\otimes)$, the map $\cat{C}(I,-)\colon\cat{C}\to\smset$ is lax monoidal because $I$ is a $\otimes$-comonoid. Thus we obtain the composite
\begin{equation}\label{eqn.partial_fn_algebra}
\defineTerm{sem}_X\colon\W=\Int(\polystar)\To{\Int(-_X)}\Int(\smsetstar)\To{\Int(\smsetstar)((0,0),-)}\smset.
\end{equation}
Unpacking, $\sem{-}_X$ sends any object $(p^-,p^+)\in\ob(\W)$ to $\sem{(p^-,p^+)}_X=\smsetstar(p^-(X),p^+(X))$, the set of partial functions $p^-(X)\parto p^+(X)$.

The above results say that for any $X:\ob\smset$, there is a traced monoidal category with a commutative monoid structure on each object $p:\ob\polystar$, in which a morphism $p\to q$ is a partial function $p(X)\parto q(X)$. One may enforce typing on the wires by using multivariate polynomials; see \cref{cor.polystar_traced,rem.duc_query}.

\section{Bypassing}\label[section]{sec.bypass}

In any cartesian monoidal category, if one is given a map $f_1\colon B\to C$ and $f_2\colon A\times C\to D$, one obtains a map $A\times B\to D$. In the corresponding string diagrams, this setup would be drawn as follows (without the blue):
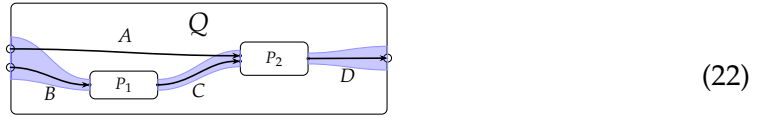
\begin{equation}\label{eqn.bypass}
\begin{tikzpicture}[WD, wdunit=7pt, halfunit=2pt, baseline=(BoxA)]
  \begin{pgfonlayer}{background}
    % only one region per side on Outer2 now
    \wdboxUnits[minimum width=5cm]{Outer2}{}{2}{1}{4}
  \end{pgfonlayer}
  \node[below=0pt] at (Outer2.north) {$Q$};

  % Box A (half size, left)
  \begin{scope}[shift={(Outer2.center)},xshift=-1cm,
                yshift=-.35cm, transform shape,scale=.3]
    \wdboxUnits[minimum width=3cm]{BoxA}{\huge $P_1$}{1}{1}{0}
  \end{scope}

  % Box B (half size, right, lifted)
  \begin{scope}[shift={(Outer2.center)},xshift=1cm,
                yshift=.0cm,transform shape,scale=.3]
    \wdboxUnits[minimum width=3cm]{BoxB}{\huge $P_2$}{2}{1}{0}
  \end{scope}
  
	% Tubes
  \tubeN{Outer2}{L}{1}{BoxA}{L}{1}{0}
  \tubeN{BoxA}{R}{1}{BoxB}{L}{1}{0}
  \tubeN{BoxB}{R}{1}{Outer2}{R}{1}{0}
	\draw[data wire] (Outer2-L-1-D2) to[out=0,in=180] node[below, font=\tiny] {$B$} (BoxA-L-1-D1);
	\draw[data wire] (Outer2-L-1-D1) to[out=0,in=180] node[above, font=\tiny] {$A$} (BoxB-L-1-D1);
	\draw[data wire] (BoxA-R-1-D1) to[out=0,in=180] node[below, font=\tiny] {$C$} (BoxB-L-1-D2);
	\draw[data wire] (BoxB-R-1-D1) to node[below, font=\tiny] {$D$} (Outer2-R-1-D1);
\end{tikzpicture}
\end{equation}
However, it is clear that---when we do include the blue control regions---this picture does not fit within the wiring diagram syntax given by $\W$, since the wire marked $A$ does not stay within the blue control region. To fix this, we introduce the notion of \emph{bypass}, which in turn comes from a $\para$ construction.
The alternative---retyping $P_1$ so that $f_1\colon B\to C$ becomes $A\times f_1\colon A\times B\to A\times C$, merely to carry the unrelated data $A$ past it---would force every box that some data merely passes to be rewritten. Bypassing instead leaves each box at its native type and routes the data around it, which also accommodates storing information from other computations, as in \cref{ex.factorial}. 

We follow \cite[Definition 5.2.4]{capucci2024actegories} for some useful definitions. (The results in this section are technical and probably useful to anyone implementing this theory; see \eqref{eqn.w'} for the definition of $\Wp$. These results will not be needed elsewhere, however, so the reader may choose to skip directly to \cref{chap.functoriality_and_int} for the double categorical $\IInt$-construction.)

\begin{definition}\label[definition]{def.para}
Let $\bullet\colon\cat{M}\times\cat{C}\to\cat{C}$ be a monoidal category $\cat{M}$ acting on a monoidal category $(\cat{C},o,\boxplus)$ whose action distributes over the sum, $m\bullet(c_1\boxplus c_2)\cong(m\bullet c_1)\boxplus(m\bullet c_2)$ and $m\bullet o\cong o$, i.e.\ a \emph{distributive algebroidal actegory} in the sense of \cite[Definition 5.2.4]{capucci2024actegories}. Define $\defineTerm{para}_\bullet(\cat{C})$ to be the operad with object set $\ob\para_\bullet(\cat{C})\coloneqq\ob\cat{C}$ and
\[
\para_\bullet(\cat{C})(c_1,\ldots,c_N;d)\coloneqq\sum_{m_1,\ldots,m_N:\cat{M}}\cat{C}\big((m_1\bullet c_1)\boxplus \cdots\boxplus (m_N\bullet c_N), d\big).
\]
In other words, a map $c_1,\ldots,c_N\to d$ in $\para_\bullet(\cat{C})$ consists of $\cat{M}$-objects $m_1,\ldots, m_N$ and a map $(m_1\bullet c_1)\boxplus \cdots\boxplus (m_N\bullet c_N)\too d$.

The identity chooses the identity object $I:\ob(\cat{M})$ and given maps
\[
(\ell_1\bullet b_1)\boxplus \cdots\boxplus (\ell_M\bullet b_M)\Too{\varphi} c_n
\qqand
(m_1\bullet c_1)\boxplus \cdots\boxplus (m_N\bullet c_N)\Too{\psi} d
\]
we define their composite $\psi\circ_n\varphi$ using distributivity:
\begin{multline*}
(m_1\bullet c_1)\boxplus \cdots\boxplus (m_{n-1}\bullet c_{n-1})
\boxplus((m_n\otimes\ell_1)\bullet b_1)\boxplus \cdots\boxplus ((m_n\otimes\ell_M)\bullet b_M)\\\boxplus(m_{n+1}\bullet c_{n+1})\boxplus\cdots\boxplus(m_N\bullet c_N)\too(m_1\bullet c_1)\boxplus \cdots\boxplus (m_N\bullet c_N)\to d.
\qedhere
\end{multline*}
\end{definition}

The only case we need of this arises as multiplication by a polynomial, which we now see acts on $\Int(\polystar)$ compatibly with its trace.

\begin{lemma}\label[lemma]{lem.polystar_dta}
Multiplication makes $\Int(\polystar)$ a distributive algebroidal actegory
\[
  (\times)\colon\poly^\cong\times\Int(\polystar)\to\Int(\polystar),
\]
compatible with the trace, in the sense that each $m\times-$ is a traced endofunctor of $\polystar$.
\end{lemma}
\begin{proof}
Set $m\bullet p\coloneqq m\times p$; multiplication distributes over $+$, which supplies the isomorphisms required of a distributive algebroidal actegory (\cref{def.para}). The action on $\Int(\polystar)$ is componentwise: a box $P=(p^-,p^+)$ is sent to $m\bullet P\coloneqq(m\times p^-,\,m\times p^+)$, and similarly for morphisms.

To see that it is compatible with the trace, note that the traced structure on $\polystar$ is inherited (see \cref{cor.polystar_traced}) from the fully faithful functor $\polystar\to\Cat{Fun}(\smset,\smsetstar)$, so that for any $m:\poly$, $\alpha\colon p+u\parto q+u$, and $C:\smset$ we have
\begin{align*}
\left(\Tr^{m\times u}_{m\times p,m\times q}(m\times\alpha)\right)(C)&\coloneqq
\Tr^{(m\times u)(C)}_{(m\times p)(C),(m\times q)(C)}(m(C)\times\alpha_C)\\&=
\Tr^{m(C)\times u(C)}_{m(C)\times p(C),m(C)\times q(C)}(m(C)\times\alpha_C)\\&\overset{\eqref{eqn.elementwise_trace}}{=}
m(C)\times\Tr^{u(C)}_{p(C),q(C)}(\alpha_C)\\&=
\left(m\times\Tr^{u}_{p,q}(\alpha)\right)(C),
\end{align*}
as desired.
\end{proof}

Thus by \cref{def.para,lem.polystar_dta} we have an operad
\begin{equation}\label{eqn.w'}
	\defineTerm{Wprime}\coloneqq\para_\times(\Int(\polystar)),
\end{equation} 
whose objects are pairs $P=(p^-,p^+)$ of polynomials and whose morphisms $P_1,\ldots,P_N\to Q$ are maps $m_1P_1,\ldots,m_NP_N\to Q$ in $\W$, i.e.\ morphisms $q^-+m_1p_1^++\cdots+m_Np_N^+\to q^++m_1p_1^-+\cdots+m_Np_N^-$ in $\polystar$ for some choice of polynomials $m_1,\ldots,m_N:\poly$. We refer to each $m_i$ as the $i$th \emph{bypass polynomial}.

In general, a bypass polynomial associated to a box $p$ will be of the form $\sum_{j:J}\yon^{A_j}$. The set $J$ represents the different options $j:J$ for what data ($A_j$) is stored, based on the control exit from which box $p$ is entered. This choice will be available upon exit from box $p$. 

We can now make sense of the diagram from \eqref{eqn.bypass}, where $P_1\coloneqq(\yon,\yon)$, $P_2\coloneqq(\yon^2,\yon)$, and $Q\coloneqq(\yon^2,\yon)$, though it is better to redraw it within the current framework: the passing wire $A$ is now held in storage (shown gray) within $P_1$: 
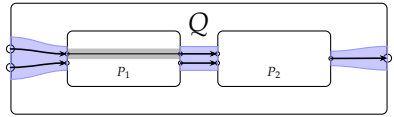
\[
\begin{tikzpicture}[WD, wdunit=7pt, halfunit=2pt, baseline=(BoxA)]
  \begin{pgfonlayer}{background}
    % only one region per side on Outer2 now
    \wdboxUnits[minimum width=5cm]{Outer2}{}{2}{1}{4}
  \end{pgfonlayer}
  \node[below=0pt] at (Outer2.north) {$Q$};

  % Box A (half size, left)
  \begin{scope}[shift={(Outer2.center)},xshift=-1cm,
                yshift=0cm, transform shape,scale=.5]
    \wdboxUnits[minimum width=3cm]{BoxA}{\raisebox{-1cm}{$P_1$}}{2}{2}{0}
  \end{scope}

  % Box B (half size, right, lifted)
  \begin{scope}[shift={(Outer2.center)},xshift=1cm,
                yshift=.0cm,transform shape,scale=.5]
    \wdboxUnits[minimum width=3cm]{BoxB}{\raisebox{-1cm}{$P_2$}}{2}{1}{0}
  \end{scope}
  
	% Tubes
  \tubeN{Outer2}{L}{1}{BoxA}{L}{1}{2}
  \tubeN{BoxA}{R}{1}{BoxB}{L}{1}{2}
  \tubeN{BoxB}{R}{1}{Outer2}{R}{1}{1}
	\draw[bypass wire, bypass tube width=4pt] (BoxA-L-1-D1) to[out=0,in=180] (BoxA-R-1-D1);
 
\end{tikzpicture}
\]
The diagram is given by a ``para''  map $P_1,P_2\to Q$ in $\Wp$, i.e.\ a map $(\yon\times P_1),P_2\to Q$ in $\W$. Unpacking, this is the morphism of type
\[
	\{Q\}\yon^2+\{P_1\}\yon\times\yon+\{P_2\}\yon\to\{Q\}\yon+\{P_1\}\yon\times\yon+\{P_2\}\yon^2
\]
given by isomorphisms $\{Q\}\yon^2\mapsto\{P_1\}\yon\times\yon$, $\{P_1\}\yon\times\yon\mapsto\{P_2\}\yon^2$, and $\{P_2\}\yon\mapsto \{Q\}\yon$.

\begin{example}[Factorial]\label[example]{ex.factorial}
The factorial function requires several instances of bypassing; it can be drawn as follows:
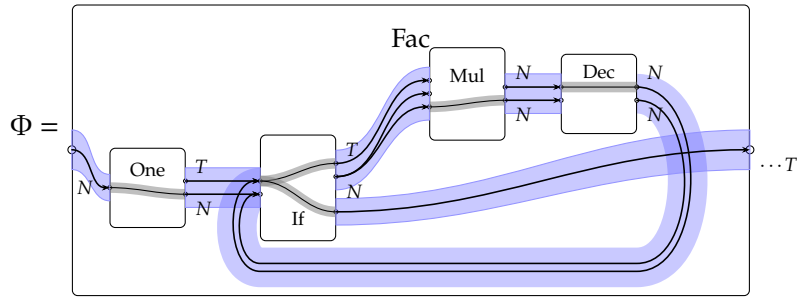
\[
\Phi=\begin{tikzpicture}[WD, baseline=(If.north)]
	  \begin{pgfonlayer}{background}
    % only one region per side on Outer2 now
    \wdboxUnits[minimum width=9cm]{Outer}{\raisebox{3cm}{Fac}}{1}{1}{9}
  \end{pgfonlayer}

  \begin{scope}[shift={(Outer.center)},xshift=-3.5cm,
                yshift=-.5cm, transform shape,scale=\innerScale]
    \wdboxUnits[minimum width=2cm]{One}{\Large \raisebox{1cm}{One}}{1}{2}{0}
  \end{scope}

  \begin{scope}[shift={(Outer.center)},xshift=-1.5cm,
                yshift=-.5cm, transform shape,scale=\innerScale]
    \wdboxUnits[minimum width=2cm]{If}{\Large \raisebox{-2cm}{If}}{2}{2,1}{0}
  \end{scope}

  \begin{scope}[shift={(Outer.center)},xshift=.75cm,
                yshift=.75cm, transform shape,scale=\innerScale]
    \wdboxUnits[minimum width=2cm]{Mul}{\Large  \raisebox{1cm}{Mul}}{3}{2}{0}
  \end{scope}

  \begin{scope}[shift={(Outer.center)},xshift=2.5cm,
                yshift=.75cm, transform shape,scale=\innerScale]
    \wdboxUnits[minimum width=2cm]{Dec}{\Large  \raisebox{1.2cm}{Dec}}{2}{2}{0}
  \end{scope}

  \tubeN{Outer}{L}{1}{One}{L}{1}{1}
 	\draw[bypass wire, bypass tube width=4pt] (One-L-1-D1) to[out=0,in=180] (One-R-1-D2);
	\tubeN{One}{R}{1}{If}{L}{1}{2}
 	\draw[bypass wire, bypass tube width=4pt] (If-L-1-D1) to[out=0,in=180] (If-R-1-D1);
 	\draw[bypass wire, bypass tube width=4pt] (If-L-1-D1) to[out=0,in=180] (If-R-2-D1);
	\tubeN{If}{R}{2}{Outer}{R}{1}{1}
	\tubeN{If}{R}{1}{Mul}{L}{1}{2}
	\draw[data wire] (If-R-1-D2) to[out=0,in=180] (Mul-L-1-D3);
 	\draw[bypass wire, bypass tube width=4pt] (Mul-L-1-D3) to[out=0,in=180] (Mul-R-1-D2);
	\tubeN{Mul}{R}{1}{Dec}{L}{1}{2}
 	\draw[bypass wire, bypass tube width=4pt] (Dec-L-1-D1) to[out=0,in=180] (Dec-R-1-D1);
  \loopTubeSmooth[below]{Dec}{R}{1}{If}{L}{1}{2}{If}
	\begin{scope}[font=\tiny]
		\node[left=2pt of One-L-1-D1] {$N$};
		\node[above right=-1pt and 0pt of One-R-1-D1] {$T$};
		\node[below right=-1pt and 0pt of One-R-1-D2] {$N$};
		\node[above right=-1pt and 0pt of If-R-1-D1] {$T$};
		\node[below right=-1pt and 0pt of If-R-1-D2] {$N$};
		\node[above right=-1pt and 0pt of Mul-R-1-D1] {$T$};
		\node[below right=-1pt and 0pt of Mul-R-1-D2] {$N$};
		\node[above right=-1pt and 0pt of Dec-R-1-D1] {$T$};
		\node[below right=-1pt and 0pt of Dec-R-1-D2] {$N$};
		\node[below right=-1pt and 0pt of Outer-R-1-D1] {$\cdots T$};
	\end{scope}
\end{tikzpicture}
\]
It represents a map $\Phi\colon(\text{One},\text{If},\text{Mul},\text{Dec})\to\text{Fac}$ in $\Wp$, where
\[
	\text{One}\coloneqq(1,\yon),\quad
	\text{If}\coloneqq(\yon,\yon+1),\quad
	\text{Mul}\coloneqq(\yon^2,\yon),\quad
	\text{Dec}\coloneqq(\yon,\yon),\quad
	\text{Fac}\coloneqq(\yon,\yon).	
\]
The bypass polynomial for each of the inner boxes is $m_1=m_2=m_3=m_4=\yon$ because each box carries one variable's worth of storage. So our factorial diagram $\Phi$ represents a map $(\yon\text{One},\yon\text{If},\yon\text{Mul},\yon\text{Dec})\tickar\text{Fac}$ in $\W=\Int(\polystar)$.

The program says to read $N$ and store it while outputting $T\coloneqq 1$, to serve as the ``total''. Then store $T$ and if $N\leq 1$, output $T$. If $N>1$ then send $N$ and $T$ off to be multiplied while $N$ is stored. The resulting new total is stored while the $N$ is decremented; repeat.

We emphasize that on its own $\Phi$ is \emph{syntax}: it is a morphism of $\W=\Int(\polystar)$ and as such records the control flow and wiring of the program but computes nothing. Computation arises only upon evaluating $\Phi$ under a $\W$-algebra $\W\to\smset$, i.e.\ a choice of semantics for the boxes. Concretely, take the partial-function algebra $\sem{-}_\nn$ of \eqref{eqn.partial_fn_algebra} (i.e.\ at $X=\nn$), under which each box $(p^-,p^+)$ is filled by a partial function $p^-(\nn)\parto p^+(\nn)$, and fill One, If, Mul, and Dec by the evident partial functions (output $1$; branch on whether $N\leq 1$; multiply; decrement), and fill each bypassed box $\yon P$ by $\nn\times h$ when $P$ is filled by $h$, so that the stored variable (of type $\nn$) passes through unchanged. The algebra then sends $\Phi$ to a single partial function $\sem{\Phi}_\nn\colon\nn\parto\nn$, in which the feedback loop of the $\Int$-construction is interpreted by the trace on $\smsetstar$---that is, by the while-loop of \cref{prop.setstar_traced}. A direct check shows that $\sem{\Phi}_\nn$ is the partial function $N\mapsto N!$; it is in this sense that the diagram, \emph{together with the chosen fillers}, computes the factorial. The particular control-flow trajectory followed by a given input is made precise in \cref{sec.trajectories}.
\end{example}

% CHAPTER %

\chapter{Double categorical \texorpdfstring{$\termraw{IInt}$}{IInt} construction and applications}\label[section]{chap.functoriality_and_int}

When one speaks of a double category, one generally means a pseudo-double category $\dd$, a pseudo-category internal to $\Cat{CAT}$, a notion we now briefly and roughly review.

Recall that these have two kinds of morphism between objects: \emph{tight} maps (the strictly-unital and associative direction, forming an ordinary category $\dd_0$) and \emph{loose} maps. Each \emph{frame} as shown left
\[
\begin{tikzcd}
	A\ar[r, tick, "f"]\ar[d, "\ell"']&
	B\ar[d, "r"]\\
	A'\ar[r, tick, "g"']&
	B'
\end{tikzcd}
\hspace{.6in}
\begin{tikzcd}
	A\ar[r, tick, "f"]\ar[d, "\ell"']&
	B\ar[d, "r"]\\
	A'\ar[r, tick, "g"']&
	B'\ar[ul, phantom, "\Downarrow"]
\end{tikzcd}
\]
has an associated set of \emph{cells} as shown right, which compose vertically, forming a category $\dd_1$. Loose maps and cells also compose horizontally with unitality and associativity holding only up to coherent isomorphism. We follow the tight/loose terminology of
%\cite{lack2012enhanced,patterson2024toward}.
\cite[p.~4]{lack2012enhanced}, \cite{patterson2024toward}. % 

We say $\dd$ is \emph{thin} if each frame---a square of two tight and two loose maps---carries at most one cell. We call a cell \emph{globular} if its two tight sides are identities, and we say $\dd$ is \emph{locally discrete} if every globular cell is an identity, so that the bicategory formed by the objects, loose maps, and globular cells of $\dd$ is merely a category.

In \cref{sec.double_cat_Int} we show that for any uniform traced category $\cat{U}$ (see \cref{def.strict_uniform}), the free compact category $\Int(\cat{U})$ forms the loose part of a thin double category $\IInt(\cat{U})$, whose tight maps are just pairs of maps in $\cat{U}$. We also show that $\IInt(\cat{U})$ is compact in the sense of \cite[Definition 1]{patterson2024toward}.
The orthogonal factorization system on $\cat{C}_\star$ of \cref{lem.ofs_cstar} will be useful in \cref{sec.operational,sec.config_tokens,sec.trajectories}, where we show an application of the double category: it lets us model control-flow \emph{trajectories} through a wiring diagram.

\section{Double category structure on \texorpdfstring{$\termraw{IInt}(\cat{U})$}{IInt(U)}}\label[section]{sec.double_cat_Int}

\begin{theorem}\label[theorem]{thm.IInt}
Let $(\cat{U},I,\otimes,\Tr)$ be a uniform traced category. There is a thin, locally discrete, symmetric monoidal double category $(\defineTerm{IInt}(\cat{U}),(I,I),(\otimes,\otimes))$, for which the category of tight maps is $\cat{U}\times\cat{U}$, for which loose maps are given by the $\Int$ construction \eqref{eqn.int_construction}, and for which a cell exists iff the obvious diagram commutes. Explicitly,
\begin{description}[labelindent=1.5em]
	\item[Object:] a pair of objects $(A^-,A^+)\in\ob(\cat{U}\times\cat{U})$.
	\item[Monoidal:] unit is $(I,I)$ and product is $(A^-,A^+)\otimes(B^-,B^+)\coloneqq(A^-\otimes B^-,A^+\otimes B^+)$.
	\item[Tight map:] a pair of maps $s^-\colon A^-\to B^-$ and $s^+\colon A^+\to B^+$ in $\cat{U}\times\cat{U}$.
	\item[Loose map:] a map $f\colon B^-\otimes A^+\to B^+\otimes A^-$ in $\cat{U}$.
	\item[Cell:]
	\begin{equation}\label{eqn.cell}
	\begin{tikzcd}
		(A_0^-,A_0^+)\ar[r, tick, "f"]\ar[d, "{(s^-,s^+)}"']&(A_1^-,A_1^+)\ar[d, "{(t^-, t^+)}"]\\
		(B_0^-,B_0^+)\ar[r, tick, "g"']&(B_1^-,B_1^+)\ar[ul, phantom, "\Downarrow"]
	\end{tikzcd}
	\qqmeans
	\begin{tikzcd}
		A_1^-\otimes A_0^+\ar[r, "f"]\ar[d, "t^-\otimes s^+"']&A_1^+\otimes A_0^-\ar[d, "t^+\otimes s^-"]\\
		B_1^-\otimes B_0^+\ar[r, "g"']&B_1^+\otimes B_0^-\ar[ul, phantom, "\tiny\textit{commutes}"]
  \end{tikzcd}
	\end{equation}
\end{description}
\end{theorem}
\begin{proof}
The double category is thin because a cell is the property that \eqref{eqn.cell} commutes, and locally discrete because with identity legs \eqref{eqn.cell} commutes iff $f=g$, in which case the unique cell is $\id_f$. It is clear that cells compose vertically,%
\footnote{
One may be tempted to define $\IInt(\cat{T})$ for any traced category by letting the tight maps be pairs of strict maps in $\cat{T}$, however these do not necessarily compose; see \cref{warn.cant_compose_strict}.
}
so to show that we have a double category, it remains to show that they compose horizontally. Suppose given the cell as \eqref{eqn.cell} and also
\[
	\begin{tikzcd}
		(A_1^-,A_1^+)\ar[r, tick, "h"]\ar[d, "{(t^-,t^+)}"']&(A_2^-,A_2^+)\ar[d, "{(u^-, u^+)}"]\\
		(B_1^-,B_1^+)\ar[r, tick, "k"']&(B_2^-,B_2^+)\ar[ul, phantom, "\Downarrow"]
	\end{tikzcd}
	\qqie
	\begin{tikzcd}
		A_2^-\otimes A_1^+\ar[r, "h"]\ar[d, "u^-\otimes t^+"']&A_2^+\otimes A_1^-\ar[d, "u^+\otimes t^-"]\\
		B_2^-\otimes B_1^+\ar[r, "k"']&B_2^+\otimes B_1^-\ar[ul, phantom, "\tiny\textit{commutes}"]
  \end{tikzcd}
\]
Recall from \eqref{eqn.int_composite} that to form the composites $f\then h$ and $g\then k$ in $\Int(\cat{U})$, one traces $A_1^+$ from the top and $B_1^+$ from bottom composites in the following commutative diagram
\[
	\begin{tikzcd}[column sep=40pt]
		A_2^-\otimes A_1^+\otimes A_0^+\ar[r, "h\otimes A_0^+"]\ar[d, "u^-\otimes t^+\otimes s^+"']&
		A_2^+\otimes A_1^-\otimes A_0^+\ar[r, "A_2^+\otimes f"]\ar[d, "u^+\otimes t^-\otimes s^+"']&
		A_2^+\otimes A_1^+\otimes A_0^-\ar[d, "u^+\otimes t^+\otimes s^-"]\\
		B_2^-\otimes B_1^+\otimes B_0^+\ar[r, "k\otimes B_0^+"']&
		B_2^+\otimes B_1^-\otimes B_0^+\ar[r, "B_2^+\otimes g"']&
		B_2^+\otimes B_1^+\otimes B_0^-
  \end{tikzcd}
\]
The resulting diagram 
\[
	\begin{tikzcd}
		(A_0^-,A_0^+)\ar[r, tick, "f\then h"]\ar[d, "{(s^-,s^+)}"']&(A_2^-,A_2^+)\ar[d, "{(u^-, u^+)}"]\\
		(B_0^-,B_0^+)\ar[r, tick, "g\then k"']&(B_2^-,B_2^+)\ar[ul, phantom, "\Downarrow"]
	\end{tikzcd}
	\qqie
	\begin{tikzcd}
	A_2^-\otimes A_0^+\ar[r, "f\then h"]\ar[d, "u^-\otimes s^+"']&A_2^+\otimes A_0^-\ar[d, "u^+\otimes s^-"]\\
	B_2^-\otimes B_0^+\ar[r, "g\then k"']&B_2^+\otimes B_0^-\ar[ul, phantom, "\tiny\textit{commutes}"]
\end{tikzcd}
\]
commutes by a combination of naturality \eqref{eqn.tr_naturality} and uniformity \eqref{eqn.uniformity}. Checking that this double category is symmetric monoidal is straightforward, completing the proof.
\end{proof}

\begin{theorem}\label[theorem]{thm.compact}
The monoidal double category $\IInt(\cat{U})$ is compact in the sense of \cite[Definition 1]{patterson2024toward}, for any uniform traced category $\cat{U}$.
\end{theorem}
\begin{proof}
We need to provide a dual $(-)^*\colon\IInt(\cat{U})\co\to\IInt(\cat{U})$, where $-\co$ means loose-opposite, and then check several axioms. Define $(A^-,A^+)^*\coloneqq(A^+,A^-)$; it is clearly functorial and satisfies
\[
\IInt(\cat{U})(A\otimes B,C)\cong\IInt(\cat{U})(A, B^*\otimes C),
\]
since there is a bijection between maps $C^-\otimes A^+\otimes B^+\To{f^\sharp} C^+\otimes A^-\otimes B^-$ and maps $B^+\otimes C^-\otimes A^+\To{f^\flat} B^-\otimes C^+\otimes A^-$. 

It is also easy to check that for tight maps $s\colon A_0\to B_0$, $t\colon A_1\to B_1$ and $u\colon A_2\to B_2$, the left diagram below commutes iff the right one does:
\[
\begin{tikzcd}
  A_0\otimes A_1\ar[r, tick, "f^\sharp"]\ar[d, "s\otimes t"']&
  A_2\ar[d, "u"]\\
  B_0\otimes B_1\ar[r, tick, "g^\sharp"']&
  B_2\ar[ul, phantom, "\Downarrow"]
\end{tikzcd}
\qqiff
\begin{tikzcd}
  A_0\ar[r, tick, "f^\flat"]\ar[d, "s"']&
  A_1^*\otimes A_2\ar[d, "t^*\otimes u"]\\
  B_0\ar[r, tick, "g^\flat"']&
  B_1^*\otimes B_2\ar[ul, phantom, "\Downarrow"]
\end{tikzcd}
\]
Finally, for any maps $a\colon A'\tickar A$, $b\colon B'\tickar B$, and $c\colon C\tickar C'$, the axioms of $\Int(\cat{U})$ as a compact 1-category imply the remaining stated condition, that $a\then f^\flat\then(b^*\otimes c)=((a\otimes b)\then f^\sharp\then c)^\flat$. We leave the remaining coherences to the reader.
\end{proof}

\section{Operational structure}\label[section]{sec.operational}

Suppose given a wiring diagram like \eqref{eqn.WD_set}, reproduced here, along with a trajectory through it as shown in green:
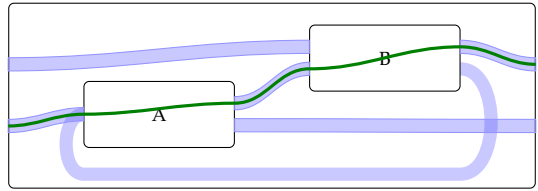
\begin{equation}\label{eqn.traj_wd}
\begin{tikzpicture}[WD, baseline=(BoxA)]
  \begin{pgfonlayer}{background}
    % only one region per side on Outer2 now
    \wdboxUnits[minimum width=7cm]{Outer2}{}{0,0}{0,0}{5.5}
  \end{pgfonlayer}

  % Box A (half size, left)
  \begin{scope}[shift={(Outer2.center)},xshift=-1.5cm,
                yshift=-.25cm, transform shape,scale=\innerScale]
    \wdboxUnits[minimum width=4cm]{BoxA}{}{0}{0,0}{0}
  \end{scope}

  % Box B (half size, right, lifted)
  \begin{scope}[shift={(Outer2.center)},xshift=1.5cm,
                yshift=.5cm,transform shape,scale=\innerScale]
    \wdboxUnits[minimum width=4cm]{BoxB}{}{0,0}{0,0}{0}
  \end{scope}

	% Tubes

  \tubeN{Outer2}{L}{2}{BoxA}{L}{1}{0}
 
	\tubeN{Outer2}{L}{1}{BoxB}{L}{1}{0}

  \tubeN{BoxA}{R}{2}{Outer2}{R}{2}{0}

  \tubeN{BoxA}{R}{1}{BoxB}{L}{2}{0}

	\tubeN{BoxB}{R}{1}{Outer2}{R}{1}{0}

  \loopTubeSmooth[below]{BoxB}{R}{2}{BoxA}{L}{1}{0}{BoxA}
	\begin{scope}[draw=green!50!black, very thick, 
		every to/.style={out=0,in=180}]
    \draw 
   		(Outer2-L-reg2-mid) to
   		(BoxA-L-reg1-mid) to 
  		(BoxA-R-reg1-mid) to 
  		(BoxB-L-reg2-mid) to
  		(BoxB-R-reg1-mid) to 
  		(Outer2-R-reg1-mid);
%  	\draw
%  		(BoxB-L-reg1-mid) to
%  		(BoxB-R-reg2-mid);	
	\end{scope}
\end{tikzpicture}
\end{equation}
The trace formula will tell us the output of the entire process, but how can we find this trajectory itself? In this section we show how it arises operationally---step by step from the input---using structures within the double category $\IInt(\cat{U})$. Namely, as the process reaches an inner box or exits the inner box, different cells in the double category will be \emph{enabled} to factor as the composite of a \emph{step} followed by a \emph{remainder}, which then leaves another cell enabled. This process drives the token machine of \cref{sec.config_tokens} and the trajectories of \cref{sec.trajectories}, such as \eqref{eqn.traj_wd}, which will become our running example.

We call the two loose sides of a cell its \emph{top} and \emph{bottom}, its two tight sides its \defWord[legs]{leg}, and their domains its \defWord[corners]{corner}; given a loose map $g$, we say a cell is \emph{over} $g$ if its \emph{bottom} is $g$.

\begin{definition}[Executable category]\label[definition]{def.executable}
We will say a category $\cat{C}$ is \defWord{executable} if it is lextensive with all points isolated and its category $\cat{C}_\star$ of decidably partial maps (\cref{lemma.lex_star}) is uniform traced, with trace \wordref{computed elementwise by iteration}.

In that case let $\IInt(\cat{C}_\star)$ be the compact double category from \cref{thm.IInt,thm.compact}. Say a tight map of $\IInt(\cat{C}_\star)$ is \emph{total} if each component is, i.e.\ if it lies in $\cat{C}\times\cat{C}\subseteq\cat{C}_\star\times\cat{C}_\star$; as in \cref{sec.partial_maps}, we draw it as $\inj$ if each component is a coproduct inclusion.
\end{definition}

\begin{proposition}\label[proposition]{prop.setL_executable}
For any small category $\cat{L}$, the category $\smset[\cat{L}]$ is \word{executable}; in particular so are $\smset$ and $\poly$.
\end{proposition}
\begin{proof}
It is lextensive with all points isolated by \cref{prop.complemented_points}, and $\smset[\cat{L}]_\star$ is uniform traced with trace \word{computed elementwise by iteration}
by \cref{cor.polystar_traced}. The last claim holds by \eqref{eqn.set_L_basics}.
\end{proof}

\cref{def.operational} below, defined for arbitrary $T_0,T_1$, could  equivalently be written only for the case $T_0=0=T_1$, because $\IInt(\cat{C}_\star)$ is thin and $+$ is a coproduct in $\cat{C}_\star$. While that version would be simpler to check, we only need to make that check once (in \cref{thm.operational}), whereas this version has the advantage that it applies directly to the cells $\phi_i$ we will need to deal with often in practice.

\begin{definition}[Operational structure]\label[definition]{def.operational}
Let $\cat{C}$ be \word{executable} and let $\phi$ be a cell of $\IInt(\cat{C}_\star)$, and suppose given objects $s_0,s_1\in\cat{C}$ and
total tight maps $\inleg_0\colon T_0+(0,s_0)\to X_0$ and $\outleg_1\colon T_1+(s_1,0)\to X_1$ such that the \word{legs} of $\phi$ factor as on the left of
\eqref{eqn.step_1}:
\begin{equation}\label{eqn.step_1}
\begin{tikzcd}
  T_0\ar[r, tick, "f", ""' name=tick1]\ar[d, tail]&T_1\ar[d, tail]\\
  T_0+(0,s_0)\ar[d, "\inleg_0"']&T_1+(s_1,0)\ar[d, "\outleg_1"]\\
  X_0\ar[r, tick, "g"', "" name=tick2]&X_1
  \ar[from=tick1, to=tick2, shorten =8mm, Rightarrow, "\phi"]
\end{tikzcd}
\hspace{.7in}
\begin{tikzcd}[row sep=15pt]
  T_0\ar[r, tick, "f", ""' name=tick1]\ar[d, tail]&[20pt]T_1\ar[d, tail]\\
  T_0+(0,s_0)\ar[d, tail]&T_1+(s_1,0)\ar[d, tail]\\
  T_0+(s_0',s_0)\ar[r, tick, "f'", "" name=tick2, ""' name=tick3] \ar[d, "\outleg_0'"'] &T_1+(s_1,s_1')\ar[d, "\inleg_1'"]\\
  X_0\ar[r, tick, "g"', "" name=tick4]&X_1
  \ar[from=tick1, to=tick2, shorten =5mm, Rightarrow, "\stp(\phi)"]
  \ar[from=tick3, to=tick4, shorten =1mm, Rightarrow, "\rem(\phi)"]
\end{tikzcd}
\end{equation}
Put $s\coloneqq s_1+s_0$ and $v\coloneqq\outleg_1^-|_{s_1}+\inleg_0^+|_{s_0}\colon s\to X_1^-+X_0^+$. We say $(\phi,v)$ is \defWord{enabled}, or simply \emph{$\phi$ is enabled with \emph{value} $v$}, if $s\ncong0$.

An \emph{operational structure} on $\IInt(\cat{C}_\star)$ assigns to each \word{enabled} $(\phi,v)$ objects $s_0',s_1'\in\cat{C}$ and a factorization $\phi=\defineTerm{stp}(\phi,v)\then\defineTerm{rem}(\phi,v)$ as on the right of \eqref{eqn.step_1}, with $\outleg_0'$ and $\inleg_1'$ again total. We call $\stp(\phi,v)$ the \emph{step} (in the sense of operational semantics) and $\rem(\phi,v)$ the \emph{remainder}; we say $\IInt(\cat{C}_\star)$ is \defWord{operational} when it is equipped with an operational structure.
We write $\stp(\phi)$ and $\rem(\phi)$ when $v$ is clear.
\end{definition}

The step reorganizes the current information $(s,v)$, which we think of as input to the step, into new information $(s',v')$, which we think of as output from the step. To do so, it takes the input summands $s_1,s_0$, which are indexed by where $v$ enters (from $X_1^-$ or $X_0^+$, respectively), and restricts to the summand $s'\coloneqq\dom(g\circ v)$ on which $g\colon X_1^-+X_0^+\parto X_1^++X_0^-$ is defined. It then finds a new decomposition $s'=s_1'+s_0'$, based on where $v'\coloneqq\tot{g\circ v}\colon s'\to X_1^++X_0^-$ lands (into $X_1^+$ or $X_0^-$, respectively).

Recall from \cref{def.extensive} that
a decidably partial map $k\colon c\parto d$ is a span $c\jni\dom(k)\to d$ whose domain of definition is a coproduct summand of $c$ and whose second leg $\tot{k}\colon\dom(k)\to d$ is total.

\begin{theorem}\label[theorem]{thm.operational}
If $\cat{C}$ is \word{executable} then $\IInt(\cat{C}_\star)$ is \word{operational}. Indeed, for an \word{enabled} $(\phi,v)$ presented as in \eqref{eqn.step_1}, put $s'\coloneqq\dom(g\circ v)$ and $v'\coloneqq\tot{g\circ v}$; the objects $s_1'$ and $s_0'$ are the pullbacks
\begin{equation}\label{eqn.pullbacks}
\begin{tikzcd}
s_1'\ar[d, "v_1'"']\ar[r]&
s'
\ar[d, "v'"]&s_0'\ar[l]\ar[d, "v_0'"]\\
X_1^+\ar[r]&X_1^++X_0^-\ar[ul, phantom, very near end, "\lrcorner"]\ar[ur, phantom, very near end, "\llcorner"]&X_0^-\ar[l]
\end{tikzcd}
\end{equation}
and the \word{legs} $\outleg_0'$ and $\inleg_1'$ of $\rem(\phi)$ are total, restricting to their two components as:
\begin{equation}\label{eqn.rem_legs}
\begin{aligned}
	\outleg_0'|_{T_0+(0,s_0)}&=\inleg_0\\
	\outleg_0'|_{(s_0',0)}&=(v_0',\bang)
\end{aligned}
\qqand
\begin{aligned}
	\inleg_1'|_{T_1+(s_1,0)}&=\outleg_1\\
	\inleg_1'|_{(0,s_1')}&=(\bang,v_1').
\end{aligned}
\end{equation}
\end{theorem}
\begin{proof}
Write $A\coloneqq T_1^-+T_0^+$,
$B\coloneqq T_1^++T_0^-$, $\rho\coloneqq\outleg_1^++\inleg_0^-\colon B\to X_1^++X_0^-$ and $\lambda\coloneqq\outleg_1^-+\inleg_0^+$, so that $A+s=(T_1^-+s_1)+(T_0^++s_0)$ is the domain of $\lambda$, and recall that a cell as on the left of \eqref{eqn.step_1} is precisely a commuting square in $\cat{C}_\star$ as on the left below:
\begin{equation}\label{eqn.factorization}
\begin{tikzcd}[column sep=32pt, row sep=20pt]
  A\ar[d, tail]\ar[r, parto, "f"]&B\ar[dd, "\rho"]\\
  A+s\ar[d, "\lambda"']\\
  X_1^-+X_0^+\ar[r, parto, "g"']&X_1^++X_0^-\ar[uul, phantom, "\tiny\textit{commutes}"]
\end{tikzcd}
\qquad
\begin{tikzcd}[row sep=20pt]
  A\ar[d, tail]\ar[r, parto, "f"]&B\ar[d, tail]\\
  A+s\ar[dd, "\lambda"']\ar[r, parto, "f+\id_s"]\ar[rrd, dashed, gray, parto, bend right=10pt, pos=.3, "{\color{gray}f'}"']&B+s\ar[dd, parto, pos=.7,  "{[\rho,g\circ v]}"']\ar[lu, phantom, pos=.05, "\ulcorner"]\ar[dr, parto, bend left=15pt]\\[-10pt]&&[-24pt]
  B+s'\ar[dl, bend left=15pt]\\[-2pt]
  X_1^-+X_0^+\ar[r, parto, "g"']&X_1^++X_0^-
\end{tikzcd}
\end{equation}
In a cocartesian category the top square of the righthand diagram is a pushout; the map it induces to $X_1^++X_0^-$ is
$\copair{\rho,\,g\circ v}$, where $v=\lambda|_s\colon s\to X_1^-+X_0^+$. Now $\rho$ is total, and $v'=v_1'+v_0'$ by \eqref{eqn.pullbacks}, so \cref{lem.ofs_cstar} gives
\begin{equation}\label{eqn.split_W}
\dom\copair{\rho,\,g\circ v}=B+s'=\big(T_1^++s_1'\big)+\big(T_0^-+s_0'\big), 
\end{equation}
the summands lying over $X_1^+$ and $X_0^-$ by $\copair{\outleg_1^+,v_1'}$ and $\copair{\inleg_0^-,v_0'}$, respectively.

By \eqref{eqn.cell}, a loose map $Y_0\tickar Y_1$ in $\IInt(\cat{C}_\star)$ is a map $Y_1^-+Y_0^+\parto Y_1^++Y_0^-$ in $\cat{C}_\star$. The map $f'$ of \eqref{eqn.factorization} has domain $A+s=(T_1^-+s_1)+(T_0^++s_0)$ and codomain $B+s'=(T_1^++s_1')+(T_0^-+s_0')$ by \eqref{eqn.split_W}, so it is the loose map
\[
	T_0+(s_0',s_0)\tickar T_1+(s_1,s_1')
\]
with \word{legs} satisfying \eqref{eqn.rem_legs}. Thus $f'$ has the form required by the right-hand diagram of \eqref{eqn.step_1}. The upper square of that diagram commutes because the top square of \eqref{eqn.factorization} does, and the lower one commutes because $\copair{\rho,g\circ v}$ is the map induced on that pushout; so both are cells, and their composite, having the \word{legs} of $\phi$, must be $\phi$ by thinness.
\end{proof}

\section{Configurations and tokens}\label[section]{sec.config_tokens}

\begin{definition}[Pre-configuration]\label[definition]{def.preconfiguration}
Let $\cat{C}$ be \word{executable}, let $n\geq0$, and let $X_0,\ldots,X_n$ be objects and $g_i\colon X_i\tickar X_{i+1}$, for $0\leq i<n$, composable loose maps in $\IInt(\cat{C}_\star)$. A \defWord{pre-configuration} over $(g_0,\ldots,g_{n-1})$ is a tuple $(\inleg_0,\phi_0,\ldots,\phi_{n-1},\outleg_n)$ consisting of
\begin{itemize}
	\item a cell $\phi_i$ over each $g_i$, with top $f_i\colon T^\outleg_i\tickar T^\inleg_{i+1}$, left \word{leg} $\outleg_i\colon T^\outleg_i\parto X_i$, and right \word{leg} $\inleg_{i+1}\colon T^\inleg_{i+1}\parto X_{i+1}$, and
	\item \defWord{environment legs} $\inleg_0\colon T^\inleg_0\parto X_0$ and $\outleg_n\colon T^\outleg_n\parto X_n$.
\end{itemize}
Thus each $X_i$ carries two \word{legs}: its \defWord{inner leg} $\inleg_i$ and its \defWord{outer leg} $\outleg_i$. The pre-configuration is \defWord[quiescent at $i$]{quiescent} if $\inleg_i=\outleg_i$, and it is \emph{quiescent} if it is quiescent at each $0\leq i\leq n$.
If $n=0$, there are no $g_i$ or $\phi_i$, and the tuple is $(\inleg_0,\outleg_0)$.
\end{definition}

For $n=2$ a \word{pre-configuration} $(\inleg_0,\phi_0,\phi_1,\outleg_2)$ can be drawn as follows:
\begin{equation}\label{eqn.preconfig}
\begin{tikzcd}[row sep=20pt, column sep=18pt]
	T^\inleg_0\ar[d, parto, "\inleg_0"']&
	T^\outleg_0\ar[dl, parto, "\outleg_0"]\ar[r, tick, "f_0", ""' name=f0]&
	T^\inleg_1\ar[dr, parto, "\inleg_1"']&&
	T^\outleg_1\ar[dl, parto, "\outleg_1"]\ar[r, tick, "f_1", ""' name=f1]&
	T^\inleg_2\ar[dr, parto, "\inleg_2"']&
	T^\outleg_2\ar[d, parto, "\outleg_2"]\\
	X_0\ar[rrr, tick, "g_0"', "" name=g0]&&&
	X_1\ar[rrr, tick, "g_1"', "" name=g1]&&&
	X_2
	\ar[from=f0, to=g0-|f0, shorten=2mm, Rightarrow, "\phi_0"]
	\ar[from=f1, to=g1-|f1, shorten=2mm, Rightarrow, "\phi_1"]
\end{tikzcd}
\end{equation}
The names record nesting: each $X_i$ is the inner box of the wiring diagram $g_i$ and the outer box of $g_{i-1}$ (\cref{rem.wd_formal}), so the cell $\phi_{i-1}$ lies inside the box $X_i$ and $\phi_i$ outside it, each keeping the \word{leg} on its side---the environment supplying the missing layer at $X_0$ and $X_n$.

\begin{definition}[Token]\label[definition]{def.token}
Let $\cat{C}$ be \word{executable}, let $r$ be a connected object, and let $\kappa$ be a \word{pre-configuration} over $(g_0,\ldots,g_{n-1})$ in $\IInt(\cat{C}_\star)$. For $0\leq i\leq n$, an $r$-\defWord{token} of index $i$ in $\kappa$ is an isomorphism
\[
	T^\outleg_i\cong T^\inleg_i+(r,0)
	\qqor
	T^\inleg_i\cong T^\outleg_i+(0,r)
\]
under which the \word{legs} at $X_i$ factor as shown:
\begin{equation}\label{eqn.token}
\begin{array}{c@{\hspace{3em}}c}
\text{value }v\colon r\parto X_i^-&\text{value }v\colon r\parto X_i^+\\
\begin{tikzcd}[row sep=8pt]
	T^\inleg_i\ar[dd, parto, "\inleg_i"']\ar[dr, tail]\\
	&[-15pt]T^\inleg_i+(r,0)\ar[dl, parto, "\outleg_i", "{\color{green!50!black}\bullet}"']&[-10pt](r,0)\ar[l, tail]\ar[dll, parto, bend left=15, pos=.3,"{(v,\bang)}"]\\
	(X_i^-,X_i^+)
\end{tikzcd}
&
\begin{tikzcd}[row sep=8pt]
	&[-10pt]&[-15pt]T^\outleg_i\ar[dd, parto, "\outleg_i"]\ar[dl, tail]\\
	(0,r)\ar[r, tail]\ar[drr, parto, bend right=15, pos=.3, "{(\bang,v)}"']&[-10pt]T^\outleg_i+(0,r)\ar[dr, parto, "\inleg_i"', "{\color{green!50!black}\bullet}"]\\
	&&(X_i^-,X_i^+)
\end{tikzcd}
\end{array}
\end{equation}
We refer to the map $v$ as the token's \defWord{value}. We say the token \defWord{enables} $\phi_{i-1}$ if $v\in X_i^-$ and $i>0$, and \defWord{enables} $\phi_i$ if $v\in X_i^+$ and $i<n$.
\end{definition}

Thus a \word{token} \word{enables} $\phi_j$ just when its \word{value} lies in the source $X_{j+1}^-+X_j^+$ of $g_j$ \eqref{eqn.cell}, and no cell just when it lies in $X_n^++X_0^-$.

\begin{example}[Values as query answers]\label[example]{ex.token_values}
Suppose $\cat{C}\coloneqq\smset[\cat{L}]$ and $r\coloneqq\yon^D$, which is connected (\cref{prop.complemented_points}), where $D\colon\cat{L}\to\smset$ is an arbitrary $\cat{L}$-instance. We think of $D$ as the domain of discourse: it assigns to each type $\ell:\cat{L}$ of one's programming language a set $D(\ell)$ of values, such as $\nn$ or the set of strings, and it assigns to each map in $\cat{L}$ an \emph{atomic function}, e.g.\ incrementing a natural number. A total \word{value} $v\colon\yon^D\to X_i^-$ is, by the Yoneda lemma, a control region $j:X_i^-(1)$ together with a data assignment $X_i^-[j]\to D$.%
\footnote{When $\cat{L}=0$ then
$D\colon 0\to\smset$ is unique,
the ``data assignment'' is vacuous, and $v$ just chooses a summand of $X_i^-$.}
In the reading of \cref{rem.duc_query}, $D$ is the instance and $v$ is one answer, drawn from $D$, to the query $X_i^-[j]$. The shape $r$ is fixed, so $D$ never changes; only the query answered does.
Concretely, when $\cat{L}=\{\yon\}$ and $D\coloneqq\nn$, a \word{value} is a control region together with a natural number in each of its data slots; the computation inside a box replaces one such tuple by another.
\end{example}

\begin{definition}[Configuration]\label[definition]{def.configuration}
With notation as above, a \word{quiescent} \defWord{configuration} is a \word{quiescent} \word{pre-configuration} for which every \word{leg} is total. An \emph{$r$-configuration of index $i$} is a \word{pre-configuration}, all of whose \word{legs} are total, equipped with an $r$-\word{token} of index $i$ and \word{quiescent} at every other index.

The vocabulary for an $r$-\word{configuration} is chosen to match that of its unique \word{token}: by definition it has the same index $i$, the same shape $r$, and we also define its \word{value} to be that of its token, and the cell it \emph{enables} (if any) to be that of its token. The \word{value} of a \word{quiescent} \word{configuration} is undefined, and it \word{enables} no cell.
\end{definition}

For $n=3$, we draw an $r$-\word{configuration} of index $2$ whose \word{token} has \word{value} in $X_2^-$; at every other $X_i$ (i.e.\ for $i\neq 2$) we have $\inleg_i=\outleg_i$, hence a single \word{corner} $T_i\coloneqq T^\inleg_i=T^\outleg_i$, and the \word{leg} is drawn once. At $X_2$ the two \word{legs} differ, and we write $T_2$ for $T^\inleg_2$ and $T_2+(r,0)$ for $T_2^\outleg$:
\begin{equation}\label{eqn.config}
\begin{tikzcd}[row sep=25pt, column sep=30pt]
	T_0\ar[d, "\inleg_0"']\ar[r, tick, "f_0", ""' name=f0]&
	T_1\ar[d]\ar[r, tick, "f_1", ""' name=f1]&[5pt]
	T_2\ar[d, "\inleg_2"']\ar[r, tail]&
	T_2+(r,0)\ar[dl, "\outleg_2", "{\color{green!50!black}\bullet}"']\ar[r, tick, "f_2", ""' name=f2]&[-12pt]
	T_3\ar[d, "\outleg_3"]\\
	(X_0^-,X_0^+)\ar[r, tick, "g_0"', "" name=g0]&
	(X_1^-,X_1^+)\ar[r, tick, "g_1"', "" name=g1]&
	(X_2^-,X_2^+)\ar[rr, tick, "g_2"', "" name=g2]&&
	(X_3^-,X_3^+)
	\ar[from=f0, to=g0-|f0, shorten=2mm, Rightarrow, "\phi_0"]
	\ar[from=f1, to=g1-|f1, shorten=2mm, Rightarrow, "\phi_1"]
	\ar[from=f2, to=g2-|f2, shorten=2mm, Rightarrow, "\phi_2"]
\end{tikzcd}
\end{equation}
This \word{token} \word{enables} cell $\phi_1$, whose \word{leg} $\inleg_2$ factors through $\outleg_2$. It is helpful to see it drawn this way:
\[
\begin{tikzcd}[row sep=20pt, column sep=60pt] 
	T_1\ar[dd]\ar[r, tick, "f_1", ""' name=f1]&
	T^\inleg_2\mathrlap{{}=T_2}\ar[dd, dashed, bend right=40, "\inleg_2"']\ar[d, tail]\\
	&T^\outleg_2\mathrlap{\;={T_2+(r,0)}}\ar[d, "\outleg_2", "{\color{green!50!black}\bullet}"']\\
	(X_1^-,X_1^+)\ar[r, tick, "g_1"', "" name=g1]& 
	(X_2^-,X_2^+) 
	\ar[from=f1, to=g1-|f1, shorten=8mm, Rightarrow, "\phi_1"] 
\end{tikzcd} 
\]

\section{Adequacy of the trajectory semantics}\label[section]{sec.trajectories}

Fix an \word{executable} category $\cat{C}$ (\cref{def.executable}), so that $\IInt(\cat{C}_\star)$ is \word{operational} (\cref{thm.operational}), and fix composable loose maps $g_i\colon X_i\tickar X_{i+1}$, for $0\leq i<n$, in $\IInt(\cat{C}_\star)$, with composite
$g\coloneqq g_{n-1}\circ\cdots\circ g_0$.
Fix also a connected object $r$ (\cref{def.extensive}).

The following proposition is the central idea of our operational semantics: one \word{step} (\cref{def.run}) runs one $g_i$ on one \word{value}. The rest of this section just repeats it: \cref{lem.trajectory} chains steps into iteration of the sum $g_\Sigma$ of the $g_i$, and \cref{thm.adequacy} identifies that iteration with the trace $\Tr^M g_\Sigma=g$ of \cref{lem.execution}.

\begin{proposition}[One step]
\label[proposition]{prop.one_hop}
If an $r$-\word{configuration} $\kappa$ with \word{value} $v$ \word{enables} a cell $\phi_i$, then its \word{legs} $\inleg_i$ and $\outleg_{i+1}$ present $\phi_i$ as on the left of \eqref{eqn.step_1}, \word{enabled} with \emph{value} $v$. The \word{pre-configuration}
\[
	\kappa'\coloneqq(\inleg_0,\phi_0,\ldots,\phi_{i-1},\rem(\phi_i),\phi_{i+1},\ldots,\phi_{n-1},\outleg_n)
\]
obtained by replacing $\phi_i$ with $\rem(\phi_i)$ as on the right of \eqref{eqn.step_1} and leaving every other cell and both \word{environment legs} unchanged is then a \word{quiescent} \word{configuration} if $g_i(v)$ is undefined, and an $r$-\word{configuration} of \word{value} $g_i(v)$ otherwise.
\end{proposition}
\begin{proof}
Since $\kappa$ \word{enables} $\phi_i$, its \word{value} $v$ lies in $X_i^+$ or in $X_{i+1}^-$; set $s_0\coloneqq r$, $s_1\coloneqq0$ in the first case and $s_0\coloneqq0$, $s_1\coloneqq r$ in the second; these are defined so that we have:
\begin{equation}\label{eqn.enabled_corners}
	T^\inleg_i=T^\outleg_i+(0,s_0)
	\qqand
	T^\outleg_{i+1}=T^\inleg_{i+1}+(s_1,0).
\end{equation}
Each equation holds by \cref{def.token} if the \word{token} sits at that index and because $\kappa$ is \word{quiescent} there if not. This presents $\phi_i$ as on the left of \eqref{eqn.step_1}, with total \word{legs} $\inleg_i$ and $\outleg_{i+1}$ and
$s=s_0+s_1\cong r$; so $\phi_i$ is \word{enabled} with \emph{value} the \word{token}'s \word{value} $v$, and $\rem(\phi_i)$ is defined (\cref{def.operational}). Let $s_0',s_1',s',v',v_0',v_1'$ and the \word{legs} $\outleg_i'$, $\inleg_{i+1}'$ of $\rem(\phi_i)$ be as in \cref{thm.operational}. By \cref{lem.elementwise}, since $r$ is connected,
$g_i(v)$ is total or undefined. If it's total then $g_i(v)\in X_{i+1}^+$ or $g_i(v)\in X_i^-$, not both, so \eqref{eqn.pullbacks} gives
\[
	(s_0',s_1')=
	\begin{cases}
	(0,0)&\text{if }g_i(v)\text{ is undefined,}\\
	(0,r)&\text{if }g_i(v)\in X_{i+1}^+,\\
	(r,0)&\text{if }g_i(v)\in X_i^-.
	\end{cases}
\]
In the last two cases $s'\coloneqq\dom(g_i(v))$ is the nonzero one of $s_0',s_1'$, and the map out of it in \eqref{eqn.pullbacks} is $v'=\tot{g_i(v)}$. Now $\kappa'$ replaces only the \word{legs} $\outleg_i$ and $\inleg_{i+1}$, so it is \word{quiescent} at every index other than $i$ and $i+1$. Their replacements $\outleg_i'$ and $\inleg_{i+1}'$ are total by \cref{thm.operational}, and by \eqref{eqn.step_1} and \eqref{eqn.enabled_corners} their \word{corners} are
\begin{equation}\label{eqn.new_corners}
	T^{\outleg\,\prime}_i=T^\outleg_i+(s_0',s_0)=T^\inleg_i+(s_0',0)
	\qqand
	T^{\inleg\,\prime}_{i+1}=T^\inleg_{i+1}+(s_1,s_1')=T^\outleg_{i+1}+(0,s_1').
\end{equation}
It remains to interpret \eqref{eqn.new_corners}.

If $s_0'=s_1'=0$ then \eqref{eqn.rem_legs} reads $\outleg_i'=\inleg_i$ and $\inleg_{i+1}'=\outleg_{i+1}$, so $\kappa'$ is \word{quiescent} at both $i$ and $i+1$, hence it is \word{quiescent}.

If $s_0'=r$ (resp.\ if $s_1'=r$), \eqref{eqn.new_corners} is an $r$-\word{token} (\cref{def.token}) of index $i$ (resp.\ index $i+1$) with 
\word{value} $\outleg_i'|_{(r,0)}=v_0'$ (resp.\ $\inleg_{i+1}'|_{(0,r)}=v_1'$) by \eqref{eqn.rem_legs}. This is $v'=g_i(v)$ in either case, and since the other of $s_0',s_1'$ is $0$, \eqref{eqn.rem_legs} makes $\kappa'$ \word{quiescent} at $i+1$ (resp.\ $i$); so $\kappa'$ is an $r$-\word{configuration}.
\end{proof}

\begin{definition}[Step; run; result]\label[definition]{def.run}
Suppose an $r$-\word{configuration} $\kappa$ \word{enables} a cell $\phi_i$. Then $\kappa$ \defWord[steps]{step} to the \word{configuration} $\kappa'$ of \cref{prop.one_hop}.

If instead $\kappa$ \word{enables} no cell---because its \word{value} lies in $X_n^++X_0^-$, or because $\kappa$ is \word{quiescent}---then $\kappa$ does not \word{step} and $\kappa'$ is undefined.

The \defWord{run} $\vec\kappa$ of an $r$-\word{configuration} $\kappa$ is the maximal sequence of configurations $\vec\kappa=(\kappa^0,\kappa^1,\ldots)$ in which $\kappa^0\coloneqq\kappa$ and $\kappa^{t+1}\coloneqq(\kappa^t)'$.
The \defWord{result} $\kappa{\Downarrow}$ of $\kappa$ is the \word{value} of the last \word{configuration} of its \word{run} if that run is finite, and undefined otherwise.
\end{definition}

A \word{run} has a \word{result} just when it is finite and its last \word{configuration} is not \word{quiescent}; since that configuration does not \word{step}, the result is a map $r\to X_0^-$ or $r\to X_n^+$.
From now on, an equation between possibly-undefined expressions, such as $\kappa{\Downarrow}=g(v)$, asserts that either both sides are defined and equal, or neither is defined.

Recall that since $r$ is connected, a map $v\colon r\to X_n^-+X_0^+$ either factors through $r\to X_n^-$ or through $r\to X_0^+$. There is exactly one \word{leg} out of the \word{corner} $(0,0)$, and it is total, and the identity $(r,0)=(0,0)+(r,0)$, resp.\ $(0,r)=(0,0)+(0,r)$, presents a \word{token} \eqref{eqn.token}, with \word{value} in $X_i^-$, resp.\ $X_i^+$.

\begin{definition}[Initial configuration]\label[definition]{def.init_config}
With notation as above, for an element $v\colon r\to X_n^-+X_0^+$, the \defWord{initial configuration} of \word{value} $v$ over $(g_0,\ldots,g_{n-1})$ is the \word{configuration} $\kappa_v$ all of whose \word{corners} are $(0,0)$, except the one presenting its \word{token}: $T^\outleg_n=(r,0)$ with $\outleg_n=(v,\bang)$ presents a \word{token} of index $n$ when $v\in X_n^-$; $T^\inleg_0=(0,r)$ with $\inleg_0=(\bang,v)$ presents a \word{token} of index $0$ when $v\in X_0^+$.
\end{definition}

\begin{example}[Assembling a configuration]\label[example]{ex.init_config}
Take $\cat{C}\coloneqq\smset$, where the only connected object is $r\coloneqq1$. Read \eqref{eqn.traj_wd} as the \word{initial configuration} (\cref{def.init_config}) built from
\begin{itemize}
	\item an outer wiring diagram $X_1\xtickar{g_1} X_2$, where $X_1\coloneqq X_{1,1}+X_{1,2}=(3,4)$ and $X_2=(2,2)$,
	\item fillers $(0,0)\xtickar{g_{0,1}} X_{1,1}$ and $(0,0)\xtickar{g_{0,2}} X_{1,2}$, namely the functions $1\to 2$ choosing the top port and $2\to 2$ swapping ports, as shown, and
	\item an element $v\colon 1\to X_2^-$ drawn as a green dot that begins the trajectory.
\end{itemize}
These form the \word{configuration} $\kappa_v$ over the pair $(g_{0,1}+g_{0,2},\ g_1)$ in $\IInt(\smsetstar)$, drawn as a wiring diagram on the left and as cells on the right:
\begin{equation}\label{eqn.init_config}
\begin{tikzpicture}[WD, baseline=(BoxA)]
  \begin{pgfonlayer}{background}
    % only one region per side on Outer2 now
    \wdboxUnits[minimum width=4cm]{Outer2}{}{0,0}{0,0}{5.5}
  \end{pgfonlayer}

  % Box A (half size, left)
  \begin{scope}[shift={(Outer2.center)},xshift=-.75cm,
                yshift=-.25cm, transform shape,scale=\innerScale]
    \wdboxUnits[minimum width=2cm]{BoxA}{\Large $X_{1,1}$}{0}{0,0}{0}
  \end{scope}

  % Box B (half size, right, lifted)
  \begin{scope}[shift={(Outer2.center)},xshift=.75cm,
                yshift=.5cm,transform shape,scale=\innerScale]
    \wdboxUnits[minimum width=2cm]{BoxB}{\Large $X_{1,2}$}{0,0}{0,0}{0}
  \end{scope}

	% Tubes

  \tubeN{Outer2}{L}{2}{BoxA}{L}{1}{0}

	\tubeN{Outer2}{L}{1}{BoxB}{L}{1}{0}

  \tubeN{BoxA}{R}{2}{Outer2}{R}{2}{0}

  \tubeN{BoxA}{R}{1}{BoxB}{L}{2}{0}

	\tubeN{BoxB}{R}{1}{Outer2}{R}{1}{0}

  \loopTubeSmooth[below]{BoxB}{R}{2}{BoxA}{L}{1}{0}{BoxA}

\begin{scope}
	\colorlet{wdCtrl}{wdCtrlInner}
  \tubeN{BoxA}{L}{1}{BoxA}{R}{1}{0}
  \tubeN{BoxB}{L}{1}{BoxB}{R}{2}{0}
  \tubeN{BoxB}{L}{2}{BoxB}{R}{1}{0}
\end{scope}
	\node[circle, draw=green!50!black, inner sep=1.5pt, fill=green!50!black] at (Outer2-L-reg2-mid) {};
\end{tikzpicture}
\hspace{.35in}
\begin{tikzcd}
  (0,0)\ar[r, tick, ""' name=tick11]\ar[d, equal]&[14pt]
  	(0,0)\ar[r, tick, ""' name=tick21]\ar[d, equal]&
			(0,0)\ar[d, tail]\\
  (0,0)\ar[d, equal]&
  	(0,0)\ar[d]&
			(1,0)\ar[d, "{(v,\bang)}", "{\color{green!50!black}\bullet}"']\\
  (0,0)\ar[r, tick, "" name=tick14, "g_{0,1}+g_{0,2}"']&
  	(X_{1,1}^-+X_{1,2}^-,X_{1,1}^++X_{1,2}^+)\ar[r, tick, "" name=tick22, "g_1"']&
			(X_2^-,X_2^+)
  \ar[from=tick11, to=tick14-|tick11, Rightarrow, shorten=7mm, "\phi_0"]
  \ar[from=tick21, to=tick21|-tick22, Rightarrow, shorten=7mm, "\phi_1"]
\end{tikzcd}
\end{equation}
Nothing has run yet, so both cells have all \word{corners} $(0,0)$ and every \word{leg} is a zero map except the environment's: its leg $\outleg_2=(v,\bang)\colon(1,0)\to(X_2^-,X_2^+)$ records the input to the machine, marked by the green dot. Under the identification $T^\outleg_2=(1,0)=T^\inleg_2+(1,0)$---the coproduct
inclusion---the green dot marks a \word{token} of index $2$ (\cref{def.token}) with \word{value} $v\colon 1\to X_2^-$, \word{enabling} $\phi_1$. Here each cell is drawn with its \word{legs} factored, but only $\outleg_2$ is nontrivially factored. The \word{configuration}'s first \word{step} is \eqref{eqn.first_hop} in \cref{ex.two_hops}, and the rest of its \word{run} is drawn in \cref{ex.rest_of_run}.
\end{example}

\begin{example}[The first two steps]
\label[example]{ex.two_hops}
Continuing \cref{ex.init_config}, the \word{token} begins at the outer input $X_2^-$, on the right-hand square carrying the wiring diagram $g_1\colon X_1\tickar X_2$, and \word{enables} $\phi_1$.
\Word{stepping} (\cref{def.run}) applies $g_1$ (\cref{prop.one_hop}), carrying the \word{token} to the input of $X_{1,1}$:
\begin{equation}\label{eqn.first_hop}
\begin{tikzpicture}[WD, baseline=(BoxA)]
  \begin{pgfonlayer}{background}
    % only one region per side on Outer2 now
    \wdboxUnits[minimum width=4cm]{Outer2}{}{0,0}{0,0}{5.5}
  \end{pgfonlayer}

  % Box A (half size, left)
  \begin{scope}[shift={(Outer2.center)},xshift=-.75cm,
                yshift=-.25cm, transform shape,scale=\innerScale]
    \wdboxUnits[minimum width=2cm]{BoxA}{\Large $X_{1,1}$}{0}{0,0}{0}
  \end{scope}

  % Box B (half size, right, lifted)
  \begin{scope}[shift={(Outer2.center)},xshift=.75cm,
                yshift=.5cm,transform shape,scale=\innerScale]
    \wdboxUnits[minimum width=2cm]{BoxB}{\Large $X_{1,2}$}{0,0}{0,0}{0}
  \end{scope}

	% Tubes

  \tubeN{Outer2}{L}{2}{BoxA}{L}{1}{0}

	\tubeN{Outer2}{L}{1}{BoxB}{L}{1}{0}

  \tubeN{BoxA}{R}{2}{Outer2}{R}{2}{0}

  \tubeN{BoxA}{R}{1}{BoxB}{L}{2}{0}

	\tubeN{BoxB}{R}{1}{Outer2}{R}{1}{0}

  \loopTubeSmooth[below]{BoxB}{R}{2}{BoxA}{L}{1}{0}{BoxA}

\begin{scope}
	\colorlet{wdCtrl}{wdCtrlInner}
  \tubeN{BoxA}{L}{1}{BoxA}{R}{1}{0}
  \tubeN{BoxB}{L}{1}{BoxB}{R}{2}{0}
  \tubeN{BoxB}{L}{2}{BoxB}{R}{1}{0}
\end{scope}
	\begin{scope}[draw=green!50!black, very thick,
		every to/.style={out=0,in=180}]
    \draw
   		(Outer2-L-reg2-mid)
   		to (BoxA-L-reg1-mid)
%  		to (BoxA-R-reg1-mid)
%  		to (BoxB-L-reg2-mid)
%  		to (BoxB-R-reg1-mid)
%  		to (Outer2-R-reg1-mid)
		;
	\end{scope}

	\node[circle, draw=green!50!black, inner sep=1.5pt, fill=green!50!black] at (BoxA-L-reg1-mid) {};
\end{tikzpicture}
\hspace{.35in}
\begin{tikzcd}
  (0,0)\ar[r, tick, ""' name=tick11]\ar[d, equal]&[15pt]
  	(0,0)\ar[r, tick, ""' name=tick21]\ar[d, equal]&
			(0,0)\ar[d, tail]\\
  (0,0)\ar[d, equal]&
  	(0,0)\ar[d, tail]&
			(1,0)\ar[d, equal]\\
  (0,0)\ar[d, equal]&
  	(1,0)\ar[r, tick, "" name=tick22, ""' name=tick23]\ar[d, "\outleg_1'", "{\color{green!50!black}\bullet}"']&
  		(1,0)\ar[d, "{(v,\bang)}"]\\
  (0,0)\ar[r, tick, "" name=tick14, "g_{0,1}+g_{0,2}"']&
  	(X_{1,1}^-+X_{1,2}^-,X_{1,1}^++X_{1,2}^+)\ar[r, tick, "" name=tick24, "g_1"']&
			(X_2^-,X_2^+)
  \ar[from=tick11, to=tick14-|tick11, Rightarrow, shorten=12mm, "\phi_0"]
  \ar[from=tick21, to=tick22, Rightarrow, shorten=7mm, "\stp(\phi_1)"]
  \ar[from=tick23, to=tick24-|tick23, Rightarrow, shorten=2mm, "\rem(\phi_1)"]
\end{tikzcd}
\end{equation}
Here the \word{leg} $\inleg_1\colon(0,0)\to(X_{1,1}^-+X_{1,2}^-,X_{1,1}^++X_{1,2}^+)$ factors through $\outleg_1'$ by a coproduct inclusion, so by \eqref{eqn.token} we have a \word{token} of index $1$, and its \word{value} is now at the input of $X_{1,1}$, \word{enabling} $\phi_0$. In other words, control has passed from the wiring diagram to one of its fillers.

The \word{token} now sits at the left-hand square, the one carrying the fillers $g_{0,1}+g_{0,2}\colon (0,0)\tickar X_{1,1}+X_{1,2}$; \word{stepping} (\cref{prop.one_hop}) applies the filler, carrying the token through the box $X_{1,1}$ to its output:
\begin{equation}\label{eqn.second_hop}
\begin{tikzpicture}[WD, baseline=(BoxA)]
  \begin{pgfonlayer}{background}
    \wdboxUnits[minimum width=4cm]{Outer2}{}{0,0}{0,0}{5.5}
  \end{pgfonlayer}
  \begin{scope}[shift={(Outer2.center)},xshift=-.75cm,
                yshift=-.25cm, transform shape,scale=\innerScale]
    \wdboxUnits[minimum width=2cm]{BoxA}{\Large $X_{1,1}$}{0}{0,0}{0}
  \end{scope}
  \begin{scope}[shift={(Outer2.center)},xshift=.75cm,
                yshift=.5cm,transform shape,scale=\innerScale]
    \wdboxUnits[minimum width=2cm]{BoxB}{\Large $X_{1,2}$}{0,0}{0,0}{0}
  \end{scope}
  \tubeN{Outer2}{L}{2}{BoxA}{L}{1}{0}
	\tubeN{Outer2}{L}{1}{BoxB}{L}{1}{0}
  \tubeN{BoxA}{R}{2}{Outer2}{R}{2}{0}
  \tubeN{BoxA}{R}{1}{BoxB}{L}{2}{0}
	\tubeN{BoxB}{R}{1}{Outer2}{R}{1}{0}
  \loopTubeSmooth[below]{BoxB}{R}{2}{BoxA}{L}{1}{0}{BoxA}
\begin{scope}
	\colorlet{wdCtrl}{wdCtrlInner}
  \tubeN{BoxA}{L}{1}{BoxA}{R}{1}{0}
  \tubeN{BoxB}{L}{1}{BoxB}{R}{2}{0}
  \tubeN{BoxB}{L}{2}{BoxB}{R}{1}{0}
\end{scope}
	\begin{scope}[draw=green!50!black, very thick,
		every to/.style={out=0,in=180}]
    \draw
   		(Outer2-L-reg2-mid)
   		to (BoxA-L-reg1-mid)
  		to (BoxA-R-reg1-mid)
			;
	\end{scope}
	\node[circle, draw=green!50!black, inner sep=1.5pt, fill=green!50!black] at (BoxA-R-reg1-mid) {};
\end{tikzpicture}
\hspace{.35in}
\begin{tikzcd}
  (0,0)\ar[r, tick, ""' name=tick11]\ar[d, equal]&[15pt]
  	(0,0)\ar[r, tick, ""' name=tick21]\ar[d, equal]&
			(0,0)\ar[d, tail]\\
  (0,0)\ar[d, equal]&
  	(0,0)\ar[d, tail]&
			(1,0)\ar[d, equal]\\
  (0,0)\ar[d, equal]&
  	(1,0)\ar[r, tick, "" name=tick22, ""' name=tick23]\ar[d, tail]&
  		(1,0)\ar[d, equal]\\
  (0,0)\ar[r, tick, "" name=tick14, ""' name=tick15]\ar[d, equal]&
  	(1,1)\ar[d, "\inleg_1'"', "{\color{green!50!black}\bullet}"]&
  		(1,0)\ar[d, "{(v,\bang)}"]\\
  (0,0)\ar[r, tick, "" name=tick16, "g_{0,1}+g_{0,2}"']&
  	(X_{1,1}^-+X_{1,2}^-,X_{1,1}^++X_{1,2}^+)\ar[r, tick, "" name=tick24, "g_1"']&
			(X_2^-,X_2^+)
  \ar[from=tick11, to=tick14-|tick11, Rightarrow, shorten=12mm, "\stp(\phi_0)"]
  \ar[from=tick15, to=tick16-|tick15, Rightarrow, shorten=2mm, "\rem(\phi_0)"]
  \ar[from=tick21, to=tick22, Rightarrow, shorten=7mm, "\stp(\phi_1)"]
  \ar[from=tick23, to=tick24-|tick23, Rightarrow, shorten=7mm, "\rem(\phi_1)"]
\end{tikzcd}
\end{equation}
The \word{configuration} has now taken a \word{step} at each cell---$\phi_1$ in the first step, $\phi_0$ in this one. The factorization $(1,0)\to(1,1)\to X_1$ records the history of how \word{tokens} flowed in the wiring diagram. This step adjoins $(0,1)$, which
by \eqref{eqn.token} presents a \word{token} of index $1$, the green dot now on $\inleg_1'$, \word{enabling} $\phi_1$. The rest of the \word{run} is drawn in \cref{ex.rest_of_run}.
\end{example}

Recall from \cref{lem.execution} that the composite $g\coloneqq g_{n-1}\circ\cdots\circ g_0$ is given by $g=\Tr^M g_\Sigma$, where $M\coloneqq\sum_{0<i<n}(X_i^-+X_i^+)$ is the sum of the inner interfaces and where
\[
	g_\Sigma\coloneqq g_{n-1}+\cdots+g_0\colon(X_n^-+X_0^+)+M\parto(X_n^++X_0^-)+M
\]
is the sum, which acts by $g_i$ on the summand $X_{i+1}^-+X_i^+$ of its source.
Its source and $X_n^++X_0^-$ partition the interfaces $\sum_{0\leq i\leq n}(X_i^-+X_i^+)$, so by \cref{lem.elementwise}\eqref{elt.exactly_one} a \word{token}'s \word{value} lies in exactly one: in the source it \word{enables} a cell, and in $X_n^++X_0^-$ it is a \word{result}.

\begin{lemma}[Trajectory of a run]\label[lemma]{lem.trajectory}
Let $\vec\kappa=(\kappa^0,\kappa^1,\ldots)$ be the \word{run} of an $r$-\word{configuration} $\kappa$ over $(g_0,\ldots,g_{n-1})$, and let $v^t$ denote
the \word{value} of $\kappa^t$ (\cref{def.configuration}). The
$v^t$ values iterate $g_\Sigma$ in the following sense:
\[
	v^{t+1}=g_\Sigma(v^t)=g_i(v^t)\qquad\text{when }v^t\in X_{i+1}^-+X_i^+,
\]
and the \word{result} $\kappa{\Downarrow}$ is the first $v^t\in X_n^++X_0^-$, or it is undefined if there is none.
\end{lemma}
\begin{proof}
When $v^t\in X_{i+1}^-+X_i^+$ for some $0\leq i<n$, the $r$-\word{configuration} $\kappa^t$ \word{enables} $\phi_i$, and \cref{prop.one_hop} computes $\kappa^{t+1}\coloneqq(\kappa^t)'$ together with its \word{value}: it is \word{quiescent} when $g_i(v^t)$ is undefined, so $v^{t+1}$ is undefined and the \word{run} ends, and otherwise it has \word{value} $v^{t+1}=g_i(v^t)$. This equals $g_\Sigma(v^t)$ since $g_\Sigma$ acts by $g_i$ on $X_{i+1}^-+X_i^+$.
The other possibility is that $v^t\in X_n^++X_0^-$, so $\kappa^t$ \word{enables} no cell (\cref{def.token}), and the \word{run} halts with \word{result} $\kappa{\Downarrow}=v^t$.
\end{proof}

\begin{corollary}[Corners log the trajectory]\label[corollary]{cor.corners_log}
Let $\vec\kappa=(\kappa^0,\kappa^1,\ldots)$ be the \word{run} of an \word{initial configuration} $\kappa_v$, with \word{values} $v^t$ as in \cref{lem.trajectory}, and let $0\leq i\leq n$. The \word{corners} of $\kappa^t$ at $X_i$ are
\begin{equation}\label{eqn.corners_log}
	T^\outleg_i=\Big(\sum_{\substack{u\leq t\\v^u\in X_i^-}}r,\ \sum_{\substack{u<t\\v^u\in X_i^+}}r\Big)
	\qqand
	T^\inleg_i=\Big(\sum_{\substack{u<t\\v^u\in X_i^-}}r,\ \sum_{\substack{u\leq t\\v^u\in X_i^+}}r\Big),
\end{equation}
with $\outleg_i$ and $\inleg_i$ restricting to $v^u$ on the summand $u$, and, for $i<n$, the top $f_i$ of $\phi_i$, unfolded by \eqref{eqn.cell}, is the partial map
\[
	f_i\colon\sum_{\substack{u<t\\v^u\in X_{i+1}^-+X_i^+}}r\parto\sum_{\substack{u\leq t\\v^u\in X_{i+1}^++X_i^-}}r
\]
that is defined on the summand $u$ iff $v^{u+1}$ is defined, and there is $\id_r$ onto the summand $u+1$.
\end{corollary}
\begin{proof}
We induct on $t$. For $t=0$ this is \cref{def.init_config}. Suppose the claim holds for $\kappa^t$, which \word{enables} $\phi_i$, and let $s_0,s_1,s_0',s_1'$ be as in the proof of \cref{prop.one_hop}, so that $v^{t+1}=g_i(v^t)$ by \cref{lem.trajectory} and
\begin{align*}
	s_0=r\iff v^t\in X_i^+,\quad &
	s_1=r\iff v^t\in X_{i+1}^-,\\
	s_0'=r\iff v^{t+1}\in X_i^-,\quad &
	s_1'=r\iff v^{t+1}\in X_{i+1}^+.
\end{align*}
By \eqref{eqn.new_corners}, $\kappa^{t+1}$ adjoins $(s_0',s_0)$ to $T^\outleg_i$ and $(s_1,s_1')$ to $T^\inleg_{i+1}$ and changes no other \word{corner}; the \word{legs} $\outleg_i'$ and $\inleg_{i+1}'$ are given by \eqref{eqn.rem_legs}, with $v_0'$, resp.\ $v_1'$, equal to $v^{t+1}$ by \cref{prop.one_hop}; and the top $f'$ of $\rem(\phi_i)$ is $f_i+\id_s$ followed by restriction to $B+s'$ \eqref{eqn.factorization}. With $s$ indexed by $t$ and $s'$ by $t+1$, this establishes the claim for $\kappa^{t+1}$.
\end{proof}

The \word{corners} are sums of copies of $r$, which is best intuited as the domain of discourse (see \cref{ex.token_values}). The $t$th copy at $X_i$ logs what crossed a wall of $X_i$ at time $t$---entering or leaving, with the control region and data represented by $v^t$---and the cell which is \word{enabled} with value $v^t$, whose top sends copy $t$ to $t+1$, logs how this crossing produces the next.

\begin{theorem}[Computational adequacy]\label[theorem]{thm.adequacy}
For every element $v\colon r\to X_n^-+X_0^+$, the
\word{initial configuration}
$\kappa_v$ of \cref{def.init_config} satisfies
\[
	\kappa_v{\Downarrow}=g(v)
\]
as (possibly-undefined) elements of $X_n^++X_0^-$.
\end{theorem}
\begin{proof}
By \cref{lem.trajectory}, $\kappa_v{\Downarrow}$ is the first element of the sequence $v^0\coloneqq v$, $v^{t+1}\coloneqq g_\Sigma(v^t)$ for which $v^t\in X_n^++X_0^-$; it is undefined if there is none.
By \cref{lem.elements_functor}, $\Gamma_r=\cat{C}(r,-)$ is a strong monoidal functor; it sends $g_\Sigma$ to the partial function $u\mapsto g_\Sigma(u)$. So the iterates of 
\[\Gamma_r g_\Sigma\colon\Gamma_r(X_n^-+X_0^+)+\Gamma_r M\parto\Gamma_r(X_n^++X_0^-)+\Gamma_r M
\]
starting at $v$ are the $v^t$, and $v^t$ lies in the summand $\Gamma_r(X_n^++X_0^-)$ iff it factors through the inclusion $X_n^++X_0^-\inj(X_n^++X_0^-)+M$, i.e.\ iff $v^t\in X_n^++X_0^-$ (\cref{def.extensive}). Hence $\kappa_v{\Downarrow}=\Tr^{\Gamma_r M}(\Gamma_r g_\Sigma)(v)$ by \eqref{eqn.elementwise_trace}. Since $\cat{C}$ is \word{executable}, its trace is \word{computed elementwise by iteration}, i.e.\ $\Gamma_r$ is a traced functor (\cref{def.trace_iteration}). Hence, by the execution formula (\cref{lem.execution}) in $\cat{C}_\star$, where the monoidal product is $+$,
\[
	\kappa_v{\Downarrow}
	=\Tr^{\Gamma_r M}(\Gamma_r g_\Sigma)(v)
	=\Gamma_r\big(\Tr^{M}g_\Sigma\big)(v)
	=\Gamma_r(g)(v)
	=g(v). \qedhere
\]
\end{proof}

\begin{example}[The rest of the run]\label[example]{ex.rest_of_run}
We now watch the rest of the machine's \word{run} on the \word{initial configuration} \eqref{eqn.init_config} of \cref{ex.init_config}: a single \word{token} threads through the diagram, each \word{step} an application of \cref{prop.one_hop}. From here on we label the \word{steps} and cells by stage, writing $\phi_i^t$ for the cell over $g_i$ in $\kappa^t$. After five steps, the \word{value} is $z\in X_2^+$, so by \cref{lem.trajectory} the \word{run} halts with \word{result} $z$:
\[
\begin{tikzpicture}[WD, baseline=(BoxA)]
  \begin{pgfonlayer}{background}
    % only one region per side on Outer2 now
    \wdboxUnits[minimum width=4cm]{Outer2}{}{0,0}{0,0}{5.5}
  \end{pgfonlayer}

  % Box A (half size, left)
  \begin{scope}[shift={(Outer2.center)},xshift=-.75cm,
                yshift=-.25cm, transform shape,scale=\innerScale]
    \wdboxUnits[minimum width=2cm]{BoxA}{\Large $X_{1,1}$}{0}{0,0}{0}
  \end{scope}

  % Box B (half size, right, lifted)
  \begin{scope}[shift={(Outer2.center)},xshift=.75cm,
                yshift=.5cm,transform shape,scale=\innerScale]
    \wdboxUnits[minimum width=2cm]{BoxB}{\Large $X_{1,2}$}{0,0}{0,0}{0}
  \end{scope}

	% Tubes

  \tubeN{Outer2}{L}{2}{BoxA}{L}{1}{0}

	\tubeN{Outer2}{L}{1}{BoxB}{L}{1}{0}

  \tubeN{BoxA}{R}{2}{Outer2}{R}{2}{0}

  \tubeN{BoxA}{R}{1}{BoxB}{L}{2}{0}

	\tubeN{BoxB}{R}{1}{Outer2}{R}{1}{0}

  \loopTubeSmooth[below]{BoxB}{R}{2}{BoxA}{L}{1}{0}{BoxA}

\begin{scope}
	\colorlet{wdCtrl}{wdCtrlInner}
  \tubeN{BoxA}{L}{1}{BoxA}{R}{1}{0}
  \tubeN{BoxB}{L}{1}{BoxB}{R}{2}{0}
  \tubeN{BoxB}{L}{2}{BoxB}{R}{1}{0}
\end{scope}
	\begin{scope}[draw=green!50!black, very thick,
		every to/.style={out=0,in=180}]
    \draw
   		(Outer2-L-reg2-mid)
   		to (BoxA-L-reg1-mid)
  		to (BoxA-R-reg1-mid)
  		to (BoxB-L-reg2-mid)
  		to (BoxB-R-reg1-mid)
  		to (Outer2-R-reg1-mid)
		;
	\end{scope}

	\node[circle, draw=green!50!black, inner sep=1.5pt, fill=green!50!black] at (Outer2-R-reg1-mid) {};
\end{tikzpicture}
\hspace{.35in}
\begin{tikzcd}[row sep=small]
  (0,0)\ar[r, tick, ""' name=tick11]\ar[d, equal]&[15pt]
  	(0,0)\ar[r, tick, ""' name=tick21]\ar[d, equal]&
			(0,0)\ar[d, tail]\\
  (0,0)\ar[d, equal]&
  	(0,0)\ar[d, tail]&
			(1,0)\ar[d, equal]\\
  (0,0)\ar[d, equal]&
  	(1,0)\ar[r, tick, "" name=tick22, ""' name=tick23]\ar[d, tail]&
  		(1,0)\ar[d, equal]\\
  (0,0)\ar[r, tick, "" name=tick14, ""' name=tick15]\ar[d, equal]&
  	(1,1)\ar[d, tail]&
  		(1,0)\ar[d, equal]\\
  (0,0)\ar[d, equal]&
  	(2,1)\ar[r, tick, "" name=tick25, ""' name=tick26]\ar[d, tail]&
  		(1,0)\ar[d, equal]\\
  (0,0)\ar[d, equal]&
  	(2,2)\ar[d, equal]&
  		(1,0)\ar[d, tail]\\
  (0,0)\ar[r, tick, "" name=tick16, ""' name=tick17]\ar[d, equal]&
  	(2,2)\ar[r, tick, "" name=tick27, ""' name=tick28]\ar[d]&
  		(1,1)\ar[d, "{\color{green!50!black}\bullet}"]\\
  (0,0)\ar[r, tick, "" name=tick18, "g_{0,1}+g_{0,2}"']&
  	(X_{1,1}^-+X_{1,2}^-,X_{1,1}^++X_{1,2}^+)\ar[r, tick, "" name=tick24, "g_1"']&
			(X_2^-,X_2^+)
  \ar[from=tick11, to=tick14-|tick11, Rightarrow, shorten=12mm, "\stp(\phi_0^1)"]
  \ar[from=tick15, to=tick16-|tick15, Rightarrow,
shorten=12mm,
"\stp(\phi_0^3)"]
  \ar[from=tick17, to=tick18-|tick17, Rightarrow, shorten=2mm, "\phi_0^5"]
  \ar[from=tick21, to=tick22, Rightarrow, shorten=7mm, "\stp(\phi_1^0)"]
  \ar[from=tick23, to=tick25, Rightarrow, shorten=7mm, "\stp(\phi_1^2)"]
  \ar[from=tick26, to=tick27, Rightarrow, shorten=7mm, "\stp(\phi_1^4)"]
  \ar[from=tick28, to=tick24-|tick28, Rightarrow, shorten=2mm, "\phi_1^5"]
\end{tikzcd}
\]
The sequence of \word{values}---the additional $(1,0)$'s and $(0,1)$'s entering at each step---is precisely the green path, alternating $g_1$- and $g_0$-applications as prescribed by \cref{lem.trajectory}. The \word{legs} of the final \word{configuration} (bottom right) log every crossing: the \word{corner} $T^\inleg_2$ over $X_2$ has grown to $(1,1)$---the input $v$ and the output $z$---while the corner over $X_1$ has grown to $(2,2)$, one point for each time the path passed through a box wall. 

The final configuration carries its \word{token} at index $2$, so it is \word{quiescent} at index $1$ (\cref{def.configuration}). Hence its two cells share the leg $(2,2)\to X_1$ and paste (\cref{thm.IInt}) to a single cell over $(g_{0,1}+g_{0,2})\then g_1$ whose right leg $\inleg_2=(v,z)\colon(1,1)\to(X_2^-,X_2^+)$ records the input and the output: a double-categorical witness that this composite sends $v\mapsto z$.
\end{example}

The partial-function algebra $\sem{-}_X$ of \eqref{eqn.partial_fn_algebra} is the composite of a traced functor with the hom out of the monoidal unit, and only the first map is special to $\smsetstar$. More generally, let $\cat{C}$ and $r$ be the \word{executable} category and connected object fixed at the start of this section, let $\cat{L}:\smcat$, and let $F\colon\smset[\cat{L}]_\star\to\cat{C}_\star$ be a traced functor. Since $\Int$ is functorial and $(0,0)$ is the monoidal unit of $\Int(\cat{C}_\star)$, the composite
\[
	\sem{-}_F\colon\Int(\smset[\cat{L}]_\star)\To{\Int(F)}\Int(\cat{C}_\star)\To{\Int(\cat{C}_\star)((0,0),-)}\smset
\]
is lax monoidal, exactly as in \eqref{eqn.partial_fn_algebra}: it sends a box $(p^-,p^+)$ to the set $\cat{C}_\star\big(F(p^-),F(p^+)\big)$ of partial maps, and it sends a wiring diagram
$\Phi\in\Int(\smset[\cat{L}]_\star)\big((p^-,p^+),(q^-,q^+)\big)$ to postcomposition with $\Int(F)(\Phi)$, so that 
\begin{equation}\label{eqn.sem_F}
	\sem{\Phi}_F(h)=h\then\Int(F)(\Phi)
	\qquad\text{for }h\in\cat{C}_\star\big(F(p^-),F(p^+)\big).
\end{equation}
\begin{corollary}\label[corollary]{cor.adequacy_F}
Let $\Phi\colon(p^-,p^+)\to(q^-,q^+)$ be a morphism of $\Int(\smset[\cat{L}]_\star)$, let $h\in\sem{(p^-,p^+)}_F$, and let $v\colon r\to F(q^-)$. Then the \word{initial configuration} $\kappa_v$ over $\big(h,\Int(F)(\Phi)\big)$ satisfies
\[
	\kappa_v{\Downarrow}=\sem{\Phi}_F(h)(v).
\]
\end{corollary}
\begin{proof}
In order to use local notation, take $g_0\coloneqq h\colon X_0\tickar X_1$ and $g_1\coloneqq\Int(F)(\Phi)\colon X_1\tickar X_2$, where $X_0=(0,0)$, $X_1=\big(F(p^-),F(p^+)\big)$, and $X_2=\big(F(q^-),F(q^+)\big)$. Then $X_2^-+X_0^+=F(q^-)$ and $v$ is an element $r\to X_2^-+X_0^+$.
\Cref{thm.adequacy} and \eqref{eqn.sem_F} give $\kappa_v{\Downarrow}=(g_0\then g_1)(v)=\big(h\then\Int(F)(\Phi)\big)(v)=\sem{\Phi}_F(h)(v)$.
\end{proof}
\begin{corollary}[Adequacy for the partial-function algebra]\label[corollary]{cor.adequacy_sem}
Let $\Phi\colon(p^-,p^+)\to(q^-,q^+)$ be a morphism of $\W$, let $X$ be a set, let $h\in\sem{(p^-,p^+)}_X$, and let $v\in q^-(X)$. Then the \word{initial configuration}
$\kappa_v$ over $\big(h,\Int(-_X)(\Phi)\big)$ satisfies
\[
	\kappa_v{\Downarrow}=\sem{\Phi}_X(h)(v).
\]
\end{corollary}
\begin{proof}
Take $\cat{L}\coloneqq\{\yon\}$, $\cat{C}\coloneqq\smset$, which is \word{executable} (\cref{prop.setL_executable}), $r\coloneqq1$, its only connected object up to isomorphism (\cref{prop.complemented_points}), and $F\coloneqq(-_X)$, which is traced (\cref{cor.polystar_traced}), in \cref{cor.adequacy_F}: then $\sem{-}_F=\sem{-}_X$, and an element $1\to q^-(X)$ is an element $v\in q^-(X)$.
\end{proof}

\chapter{Conclusion}\label[section]{chap.conclusion}

In this paper we have considered $\Int(\polystar)$ and more generally $\Int(\smset[\cat{L}]_\star)$ as a \emph{language} for control flow together with data transformations, and equipped it with structure and meaning from two complementary directions.

The foundation is found in \cref{chap.background} where we show that the category $\polystar$ of pointed polynomial functors, and its multivariate generalization $\smset[\cat{L}]_\star$, is uniform traced (\cref{cor.polystar_traced}). The trace is the categorical form of feedback---the ``while loop''---which is the foundation of everything else in this paper. From there we can take the free compact closed category $\Int(\polystar)$. The uniformity of the trace is exactly what is needed to extend this to a double-categorical refinement in \cref{chap.functoriality_and_int}.

\Cref{chap.int_poly_wds} treats $\W=\Int(\polystar)$ as syntax and gives it a denotational semantics. As an operad (\cref{sec.int_polystar_wd_syntax}), $\W$ has objects, ``boxes,'' that are pairs of polynomials, and morphisms that are nestable wiring diagrams. A wiring diagram carries both control flow, drawn as blue regions that encode the case structure at each interface, and data flow, drawn as wires along which data can be copied, discarded, permuted, and transformed by functions. Both are encoded by the polynomials: their positions and directions supply the case structure at each interface and the data carried on the wires (which is typed, in the multivariate version), while the trace of $\Int$ provides the looping control. The algebras of $\W$ turn this syntax into semantics: every $\W$-algebra is a traced category with  a supply of commutative monoids (\cref{prop.supply_monoids}), so that branches of control may be merged. Then the partial-function algebra $\sem{-}_X$ of \eqref{eqn.partial_fn_algebra} interprets each diagram as an honest partial function, as illustrated by the factorial of \cref{ex.factorial}. Bypassing (\cref{sec.bypass}) is simply there to make things more modular and accommodate storing information from other computations.

\Cref{chap.functoriality_and_int} supplies an operational semantics, complementary to the denotational semantics defined above. For a uniform traced category $\cat{U}$ we promote $\Int(\cat{U})$ to a thin double category $\IInt(\cat{U})$ (\cref{thm.IInt}), which is compact in the sense of \cite[Definition 1]{patterson2024toward}, by adding tight maps that are simply pairs of maps in $\cat{U}$. When $\cat{C}$ is \word{executable} in our sense (as $\smset$ and $\poly$ are; see \cref{prop.setL_executable}) each cell of $\IInt(\cat{C}_\star)$ \word{enabled} with a value---started off by the one enabled with the value that was input to the machine---has a chosen factorization into a step and a remainder (\cref{thm.operational}). The step factorization drives a token machine (\cref{sec.config_tokens}): a \word{configuration} carries a \word{token}---an additional summand $r$ in the domain of one \word{leg}---and \word{steps} by factoring the \word{legs} of the cell \word{enabled} by that \word{token} (\cref{prop.one_hop}), which thereby relocates the token. The corresponding \word{run} is the path traced by control as it threads through a wiring diagram (\cref{sec.trajectories}).

The two semantics agree. \Cref{thm.adequacy} shows that the \word{run} of the \word{initial configuration} $\kappa_v$ has \word{result} $\kappa_v{\Downarrow}=g(v)$, where $g$ is the composite loose map, the two sides being simultaneously defined or undefined; and \cref{cor.adequacy_sem} instantiates this at the partial-function algebra
$\sem{-}_X\colon\W\to\smset$ \eqref{eqn.partial_fn_algebra}. That is, for a wiring diagram $\Phi$ filled by $h$ and an input $v$, the run of $\kappa_v$ has result $\sem{\Phi}_X(h)(v)$. This is computational adequacy in the sense of Plotkin \cite[Theorem 3.1]{Plotkin:1977a}. This makes $\Int(\polystar)$ not merely a presentation of a compact closed category, but a usable language for wiring together control and data, with meaning attached.

Several directions remain open.
One is a precise comparison with the geometry-of-interaction and game-semantics traditions (\cref{rem.goi}): \cref{thm.adequacy} says that the run of a token computes the trace, hence the execution formula of \cite[eq.~(1), p.~262]{haghverdi2006categorical}; it remains to ask whether the run itself---the sequence of configurations, not merely its result---coincides with the path of a GoI token machine \cite{mackie1995geometry} or with the play of a strategy in game semantics. 
Another is to present $\W$ by generators and relations, in the style of presented PROPs \cite{lack2004composing}; this is not needed for the syntax to be formal (\cref{rem.wd_formal}), but it would supply local rewrite rules for diagrammatic reasoning and a bridge to the equational theories of flowchart schemes. We hope the present foundations make such questions tractable.

\paragraph{Competing Interests.}
The authors declare that they have no competing interests.

\printbibliography

\end{document}